\documentclass[11pt,twoside]{preprint}
\usepackage[full]{textcomp}
\usepackage[osf]{newtxtext}

\usepackage{amssymb}
\usepackage{mathtools}
\usepackage{hyperref}
\usepackage{breakurl}
\usepackage{mhenvs}
\usepackage{mhequ}
\usepackage{mhsymb}
\usepackage{booktabs}
\usepackage{tikz}
\usepackage{mathrsfs}

\usepackage{microtype}
\usepackage{comment}
\usepackage{wasysym}
\usepackage{centernot}
\usepackage{enumitem}
\usepackage{bm}
\usepackage{stackrel}

\DeclareMathAlphabet{\mathbbm}{U}{bbm}{m}{n}

\DeclareFontFamily{U}{BOONDOX-calo}{\skewchar\font=45 }
\DeclareFontShape{U}{BOONDOX-calo}{m}{n}{
  <-> s*[1.05] BOONDOX-r-calo}{}
\DeclareFontShape{U}{BOONDOX-calo}{b}{n}{
  <-> s*[1.05] BOONDOX-b-calo}{}
\DeclareMathAlphabet{\mcb}{U}{BOONDOX-calo}{m}{n}
\SetMathAlphabet{\mcb}{bold}{U}{BOONDOX-calo}{b}{n}

\setlist{noitemsep,topsep=4pt}

\makeatletter
\def\DeclareSymbol#1#2#3{\expandafter\gdef\csname MH@symb@#1\endcsname{\tikzsetnextfilename{symbol#1}%
\tikz[baseline=#2,scale=0.15,line join=round]{#3}}\expandafter\gdef\csname MH@symb@#1s\endcsname{\scalebox{0.7}{\tikzsetnextfilename{symbol#1}%
\tikz[baseline=#2,scale=0.15,line join=round]{#3}}}}
\def\<#1>{\csname MH@symb@#1\endcsname}
\makeatother

\DeclareSymbol{0}{-2.4}{\node[dot] {};}
\DeclareSymbol{1}{0}{\draw[white] (-.5,0) -- (.5,0); \draw (0,0)  -- (0,1.5) node[sdot] {};}
\DeclareSymbol{11}{0}{\draw (-0.5,1.2) node[sdot] {} -- (0,0) -- (0.5,1.2) node[sdot] {};}
\DeclareSymbol{111}{0}{\draw (0,0) -- (0,1.2) node[sdot] {}; \draw (-.7,1) node[sdot] {} -- (0,0) -- (.7,1) node[sdot] {};}
\DeclareSymbol{131}{-3}{\draw (0,0) -- (0,-1) -- (1,0) node[sdot] {}; \draw (0,0) -- (0,-1) -- (-1,0) node[sdot] {}; \draw (0,0) -- (0,1.2) node[sdot] {}; \draw (-.7,1) node[sdot] {} -- (0,0) -- (.7,1) node[sdot] {};}
\DeclareSymbol{11}{0}{\draw (-0.5,1.2) node[sdot] {} -- (0,0) -- (0.5,1.2) node[sdot] {};}
\DeclareSymbol{12}{-3}{\draw (-.8,1) node[sdot] {} -- (0,0) -- (0,-1); \draw (1,0) node[sdot] {} -- (0,-1) -- (-1,0) node[sdot] {};}
\DeclareSymbol{30}{-3}{\draw (0,0) -- (0,-1); \draw (0,0) -- (0,1.2) node[sdot] {}; \draw (-.7,1) node[sdot] {} -- (0,0) -- (.7,1) node[sdot] {};}
\DeclareSymbol{10}{-3}{\draw (-.8,1) node[sdot] {} -- (0,0) -- (0,-1);}
\DeclareSymbol{22}{-3}{\draw (0,0.3) -- (0,-1) -- (1,0) node[sdot] {}; \draw (0,0.3) -- (0,-1) -- (-1,0) node[sdot] {};\draw (-.7,1) node[sdot] {} -- (0,0.3) -- (.7,1) node[sdot] {};}
\DeclareSymbol{211}{0}{\draw (-0.5,2.4) node[sdot] {} -- (-1,1.6);\draw (-1.5,2.4) node[sdot] {} -- (-0.5,0.8); \draw (0,1.6) node[sdot] {} -- (-0.5,0.8) -- (0,0) -- (0.5,0.8) node[sdot] {};}

\newcommand{\cut}{\mathfrak{C}}

\def\Alg{\mathrm{Alg}}

\def\mainroot{\rho_{\ast}}

\newcommand{\mcE}{\mathcal{E}}

\newcommand{\mcV}{\mathcal{V}}

\newcommand{\mcC}{\mathcal{C}}
\newcommand{\mcB}{\mathcal{B}}
\newcommand{\J}{\mathcal{J}}
\newcommand{\mcS}{\mathcal{S}}
\newcommand{\mcU}{\mathcal{U}}

\newcommand{\mcT}{\mathcal{T}}

\newcommand{\mcF}{\mathcal{F}}

\newcommand{\mcN}{\mathcal{N}}

\newcommand{\mcD}{\mathcal{D}}

\newcommand{\mcG}{\mathcal{G}}

\newcommand{\mcQ}{\mathcal{Q}}
\newcommand{\mcZ}{\mathcal{Z}}

\newcommand{\mbbF}{\mathbb{F}}

\newcommand{\mbbM}{\mathbb{M}}

\newcommand{\mbbG}{\mathbb{G}}

\newcommand{\mbn}{\mathbf{n}}

\newcommand{\mbj}{\mathbf{j}}
\newcommand{\mbk}{\mathbf{k}}

\newcommand{\T}{\mathbf{T}}

\def\Lab{\mathfrak{L}}

\def\Deltam{\Delta^{\!-}}

\def\${|\!|\!|}

\def\J{\mathscr{I}}

\def\CF{\mathcal{F}}
\def\CT{\mcb{T}}

\def\noise{\boldsymbol{\xi}}

\newcommand{\mfn}{\mathfrak{n}}

\newcommand{\mfL}{\mathfrak{L}}

\newcommand{\mfC}{\mathfrak{C}}
\newcommand{\mfR}{\mathfrak{R}}

\newcommand{\mft}{\mathfrak{t}}

\newcommand{\mfM}{\mathfrak{M}}

\newcommand{\mfp}{\mathfrak{p}}

\newcommand{\mfJ}{\mathfrak{J}}
\newcommand{\mfl}{\mathfrak{l}}

\newcommand{\mfq}{\mathfrak{q}}

\newcommand{\mfG}{\mathfrak{G}}

\def\cC{\mathscr{C}}

\def\cG{\mathscr{G}}

\def\cS{\mathscr{S}}
\def\cD{\mathscr{D}}

\def\cB{\mathscr{B}}

\def\enS{\mathscr{S}} 

\def\can{\textnormal{\tiny{can}}}
\def\neut{\textnormal{\tiny{neut}}}
\def\BPHZ{\textnormal{\tiny \textsc{bphz}}}
\def\opp{\textnormal{\tiny \textsc{opp}}}

\def\SG{\textnormal{\tiny \textsc{SG}}}

\def\Div{\mathrm{Div}}

\def\DDom{\mathcal{D}}

\newcommand{\ch}{\mathrm{c}}
\newcommand{\p}{\mathrm{p}}
\newcommand{\inte}{\mathrm{int}}
\newcommand{\exte}{\mathrm{ext}}

\newcommand{\Coll}{\mathrm{Coll}}

\newcommand{\go}[1]{\mathbf{#1}}
\newcommand{\bo}[1]{\mathbf{#1}}
\newcommand{\mvert}{\mathcal{V}}
\newcommand{\mtree}{\CU}

\def\combplus[#1,#2,#3,#4]{\binom{#1\ {\scriptstyle #4} }{#2\ #3}}

\newcommand{\leavesleft}[2]{L(#1,#2)}
\newcommand{\kernelsleft}[2]{K(#1,#2)}
\newcommand{\nodesleft}[2]{N(#1,#2)}

\newcommand{\genvert}[1]{H_{#1}}

\def\singlescalegenvert[#1,#2]{\hat{H}^{#2}_{#1}}
\def\multiscalegenvert[#1,#2]{H^{#2}_{#1}}

\newcommand{\mmax}[1]{\overline{#1}}

\def\nr[#1]{\tilde{N}[#1]} 
\def\inn[#1]{\mathring{N}[#1]}
\def\nrinn[#1]{\hat{N}_{#1}} 
\def\nrmod[#1,#2]{\tilde{N}_{#1}(#2)}
\def\nrinnmod[#1,#2]{\hat{N}_{#1}(#2)}

\def\ident[#1]{\underline{#1}}

\def\mylink#1#2{\mathrel{\vbox{\offinterlineskip\ialign{%
    \hfil##\hfil\cr
    $\scriptscriptstyle#1$\cr
    \noalign{\kern0.1ex}
    $#2$\cr
}}}}
\def\mysublink[#1]#2#3{\mathrel{\vbox{\offinterlineskip\ialign{%
    \hfil##\hfil\cr
    $\scriptscriptstyle#2$\cr
    \noalign{\kern0.1ex}
    $#3$\cr
    \noalign{\kern-0.2ex}
    \smash{\raisebox{-\height}{\hbox{$\scriptscriptstyle #1$}}}\cr
    \noalign{\kern0.2ex}
}}}}
\newcommand{\link}[3]{#1\mylink{#3}{\longleftrightarrow}#2}

\newcommand{\ulink}[3]{#1\mylink{#3}{\leftrightsquigarrow}#2}

\def\allf{\mcb{C}}

\def\fon[#1]{\cD_{#1}}

\def\mincompproj[#1]{\mfp_{#1}}

\def\Proj_#1{\mathop{\mathrm{Proj}_{#1}}}

\def\negrenorm[#1]{\mfR_{#1}}
\def\topnegrenorm[#1]{\overline{\mfR}_{#1}}

\def\quotedge[#1]{E^{q}_{#1}}

\def\posrenorm[#1]{\mcC_{#1}}
\def\topposrenorm[#1]{\overline{\mcC_{#1}}}
\def\cutsmod[#1]{\mathbb{C}_{+,#1}}
\def\fullcuts{\cut}
\def\fullcutsmod[#1]{\cut_{#1}}

\def\slash{\unskip\kern0.2em/\penalty\exhyphenpenalty\kern0.2em\ignorespaces}
\def\dash{\unskip\kern0.2em--\penalty\exhyphenpenalty\kern0.2em\ignorespaces}

\def\emptyset{{\centernot\ocircle}}

\colorlet{symbols}{blue!90!black}
\colorlet{testcolor}{green!60!black}

\colorlet{redkernel}{red!80}

\def\symbol#1{\textcolor{symbols}{#1}}
\def\1{\mathbf{\symbol{1}}}

\usetikzlibrary{shapes.misc}
\usetikzlibrary{shapes.symbols}
\usetikzlibrary{shapes.geometric}
\usetikzlibrary{snakes}
\usetikzlibrary{decorations}
\usetikzlibrary{decorations.markings}

\usetikzlibrary{calc}

\newcommand*{\mathcolor}{}
\def\mathcolor#1#{\mathcoloraux{#1}}
\newcommand*{\mathcoloraux}[3]{%
  \protect\leavevmode
  \begingroup
    \color#1{#2}#3%
  \endgroup
}

\def\drawx{\draw[-,solid] (-3pt,-3pt) -- (3pt,3pt);\draw[-,solid] (-3pt,3pt) -- (3pt,-3pt);}
\tikzset{
  charge/.style={circle,draw=black,inner sep=0pt, minimum size=2.5mm},
  coalline/.style={semithick,draw=black},
  coalnode/.style={circle,fill=black!10,draw=black,inner sep=0pt, minimum size=2mm},
  coalnode2/.style={circle,fill=black!60,draw=black,inner sep=0pt, minimum size=2mm},
  shadednode/.style={circle,fill=blue!40,inner sep=0pt,minimum size=9pt},
  shadededge/.style={line width=6pt,blue!40, shorten >= -3pt,shorten <=-3pt},
  subtreenode/.style={circle,fill=gray!40,inner sep=0pt,minimum size=12pt},
  smallsubtreenode/.style={circle,fill=gray!40,inner sep=0pt,minimum size=9pt},
  subtreeedge/.style={line width=10pt,gray!40, shorten >= -3pt,shorten <=-3pt},
  smallsubtreeedge/.style={line width=6pt,gray!40, shorten >= -3pt,shorten <=-3pt},
  root/.style={circle,fill=testcolor,inner sep=0pt, minimum size=2mm},
  logof/.style={circle,fill=darkblue,inner sep=0pt, minimum size=2mm},
  dot/.style={circle,fill=black,inner sep=0pt,minimum size=1mm},
  sdot/.style={circle,fill=black,inner sep=0pt,minimum size=.5mm},
  int/.style={circle,fill=black,draw=black,inner sep=0pt,minimum size=0.7mm},
  circ/.style={circle,draw=black,inner sep=0pt, minimum size=1mm},
  var/.style={circle,fill=black!10,draw=black,inner sep=0pt, minimum size=2mm},
  dotred/.style={circle,fill=black!50,inner sep=0pt, minimum size=2mm},
  generic/.style={semithick,shorten >=1pt,shorten <=1pt},
  oddfunc/.style={generic, dotted},
  dist/.style={ultra thick,draw=testcolor,shorten >=1pt,shorten <=1pt},
  testfcn/.style={ultra thick,testcolor,shorten >=1pt,shorten <=1pt,<-},
  testfunction/.style={ultra thick,testcolor,shorten >=1pt,shorten <=1pt},
  testfcnx/.style={ultra thick,testcolor,shorten >=1pt,shorten <=1pt,<-,
    postaction={decorate,decoration={markings,mark=at position 0.6 with {\drawx}}}},
  kprime/.style={semithick,shorten >=1pt,shorten <=1pt,densely dashed,->},
  kprimex/.style={semithick,shorten >=1pt,shorten <=1pt,densely dashed,->,
    postaction={decorate,decoration={markings,mark=at position 0.4 with {\drawx}}}},
  kernel/.style={semithick,shorten >=1pt,shorten <=1pt,->,draw=black},
  cov/.style={semithick,shorten >=1pt,shorten <=1pt,draw=darkred},
  covi/.style={semithick,shorten >=1pt,shorten <=1pt,draw=darkblue},
  gone/.style={dotted},
  leaf/.style={decorate, decoration = {zigzag, amplitude=1.5pt, segment length=3pt}, semithick,shorten >=1pt,shorten <=1pt,draw=black},
  keps/.style={semithick,densely dashed,shorten >=1pt,shorten <=1pt,->},
  Pkernel/.style={ultra thick,shorten >=1pt,shorten <=1pt,->,draw=blue},
  PkernelBig/.style={very thick,shorten >=1pt,shorten <=1pt,decorate, draw=blue, decoration={zigzag,amplitude=1.5pt,segment length = 3pt,pre length=2pt,post length=2pt}},
  multx/.style={shorten >=1pt,shorten <=1pt,
    postaction={decorate,decoration={markings,mark=at position 0.5 with {\drawx}}}},
  kernelx/.style={semithick,shorten >=1pt,shorten <=1pt,->,
    postaction={decorate,decoration={markings,mark=at position 0.4 with {\drawx}}}},
  kernel1/.style={->,semithick,shorten >=1pt,shorten <=1pt,postaction={decorate,decoration={markings,mark=at position 0.5 with {\draw[-] (0,-0.2) -- (0,0.2);}}}},
  kernel2/.style={->,semithick,shorten >=1pt,shorten <=1pt,postaction={decorate,decoration={markings,mark=at position 0.45 with {\draw[-] (0.05,-0.2) -- (0.05,0.2);\draw[-] (-0.05,-0.2) -- (-0.05,0.2);}}}},
  kernel3/.style={->,semithick,shorten >=1pt,shorten <=1pt,postaction={decorate,decoration={markings,mark=at position 0.45 with {
    \draw[-] (0.075,-0.2) -- (0.075,0.2);
    \draw[-] (-0.075,-0.2) -- (-0.075,0.2);
    \draw[-] (0,-0.2) -- (0,0.2);}}}},
  kernel4/.style={->,semithick,shorten >=1pt,shorten <=1pt,postaction={decorate,decoration={markings,mark=at position 0.45 with {
    \draw[-] (0.15,-0.2) -- (0.15,0.2);
    \draw[-] (0.05,-0.2) -- (0.05,0.2);
    \draw[-] (-0.05,-0.2) -- (-0.05,0.2);
    \draw[-] (-.15,-0.2) -- (-.15,0.2);}}}},  
  kernelBig/.style={semithick,shorten >=1pt,shorten <=1pt,decorate, decoration={zigzag,amplitude=1.5pt,segment length = 3pt,pre length=2pt,post length=2pt}},
  kernelBigg/.style={thick,shorten >=1pt,shorten <=1pt,decorate, decoration={zigzag,amplitude=3.5pt,segment length = 7pt,pre length=2pt,post length=2pt}},
  kernelBigg1/.style={thick,shorten >=1pt,shorten <=1pt,decorate, decoration={saw,amplitude=3.5pt,segment length = 7pt,pre length=2pt,post length=2pt}},
  kernelBigg2/.style={thick,shorten >=1pt,shorten <=1pt,decorate, decoration={bumps,amplitude=3.5pt,segment length = 7pt,pre length=2pt,post length=2pt}},
  rho/.style={dotted,semithick,shorten >=1pt,shorten <=1pt},
  renorm/.style={shape=circle,fill=white,inner sep=1pt},
  labl/.style={shape=rectangle,fill=white,inner sep=1pt},
cumu2n/.style={inner sep=3pt},
cumu2/.style={draw=red!80,fill=red!40},
cumu3/.style={regular polygon, regular polygon sides=3,draw=red!80,rounded corners=3pt,fill=red!40,minimum size=5mm},
cumu4/.style={regular polygon, regular polygon sides=4,draw=red!80,rounded corners=3pt,fill=red!40,minimum size=7mm},
cumu5/.style={regular polygon, regular polygon sides=5,draw=red!80,rounded corners=3pt,fill=red!40,minimum size=7mm},
bcumu2n/.style={inner sep=3pt},
bcumu2/.style={draw=blue!80,fill=blue!40},
bcumu3/.style={regular polygon, regular polygon sides=3,draw=blue!80,rounded corners=3pt,fill=blue!40,minimum size=5mm},
bcumu4/.style={regular polygon, regular polygon sides=4,draw=blue!80,rounded corners=3pt,fill=blue!40,minimum size=7mm},
bcumu5/.style={regular polygon, regular polygon sides=5,draw=blue!80,rounded corners=3pt,fill=blue!40,minimum size=7mm},
  xi/.style={circle,fill=symbols!15,draw=symbols,inner sep=0pt,minimum size=1.2mm},
  xib/.style={circle,fill=symbols!10,draw=symbols,inner sep=0pt,minimum size=1.6mm},
  not/.style={circle,fill=symbols,draw=symbols,inner sep=0pt,minimum size=0.5mm},
  >=stealth,
  }

\colorlet{darkred}{red!90!black}
\colorlet{darkblue}{blue!90!black}
\colorlet{darkgreen}{green!50!black}
\let\comm\comment
\def\martin#1{}
\def\ajay#1{}
\def\hao#1{}

\def\CQ{\mathcal{Q}}

\let\J\CJ

\def\s{\mathfrak{s}}
\def\c{\mathfrak{c}}
\newcommand{\e}{\varepsilon}

\def\${|\!|\!|}
\def\Wick#1{{:}{#1}{:}}

\def\Ker{\mathrm{Ker}}
\def\RKer{\mathrm{RKer}}
\def\KerHat{\widehat{\RKer}}
\def\KerTilde{\widetilde{\RKer}}

\def\?{{\color{red}?}}

\def\restr{\mathord{\upharpoonright}}

\def\sign{\textnormal{sign}}

\def\Id{\mathrm{Id}}

\def\reg{\mathop{\mathrm{reg}}}

\def\PPi{\boldsymbol{\Pi}}

\def\Lab{\mathfrak{L}}

\def\Ke{\mathfrak{L}_{+}}
\def\Le{\mathfrak{L}_{-}}

\def\sT{{\overline{T}}} 
\def\sn{\overline{\mfn}} 
\def\sl{\overline{\mfl}} 
\def\allnodes{N^{\ast}}

\newcommand{\powroot}[3]{\mathrm{X}^{#1}_{#2,#3}}

\newcommand{\ke}[2]{\Ker^{#1}_{#2}}

\def\mfJ{\mathfrak{J}}

\def\one{\mathbbm{1}}

\newtheorem{example}[lemma]{Example}

\let\basepoint\logof
\def\logof{\mathord{{\basepoint}}} 
\def\back{\!\!\!}

\title{The dynamical sine-Gordon model\\ in the full subcritical regime}
\author{Ajay Chandra$^1$, Martin Hairer$^1$, and Hao Shen$^2$}

\institute{Imperial College London\and Columbia University}

\date{\today}
\begin{document}
\maketitle
\begin{abstract}
We prove that the dynamical sine-Gordon equation on the two dimensional torus introduced in \cite{HaoSG} 
is locally well-posed for the entire subcritical regime.
At first glance this equation is far out of the scope of the local existence theory available in the framework of regularity structures \cite{Regularity,BHZalg,CH,BCCH} since it involves a non-polynomial nonlinearity and the solution is expected to be a distribution (without any additional small parameter as in \cite{FurlanWeakUniversality,Weijun2}). 

In \cite{HaoSG} this was overcome by a change of variable, but the new equation that arises has a  multiplicative dependence on highly non-Gaussian noises which makes stochastic estimates highly non-trivial \dash as a result \cite{HaoSG} was only able to treat part of the subcritical regime. Moreover, the cumulants of these noises fall out of the scope of the later work \cite{CH}. 
In this work we systematically leverage ``charge'' cancellations specific to this model and obtain stochastic estimates that allow us to cover the entire subcritical regime.
\end{abstract}
\setcounter{tocdepth}{2}
\tableofcontents

\section{Introduction}\label{sec: intro}

This paper studies the local well-posedness of the \emph{sine-Gordon} equation
\begin{equ}[e:model]
\d_t u = \frac{1}{2}\Delta u + \sin\bigl( \beta u \bigr)
 + \zeta\;,
\end{equ}
over $\R_{+} \times \T^2$, where $\zeta$ denotes space-time white noise.

Equation~\eqref{e:model} is related to many models of equilibrium statistical mechanics in two dimensions. 
Most directly it is the natural (Langevin) dynamic for the sine-Gordon
(Euclidean) quantum field theory in two spatial dimensions \dash this field theory is a functional Fourier transform of the two dimensional Coulomb gas model. 

Previously, \cite{HaoSG} showed that for $\beta^2 \in (0,4\pi)$ the methods
of Da Prato and Debussche \cite{MR1941997,MR2016604} suffice for showing well-posedness while for $\beta^2 \in [4\pi,  \frac{16\pi}{3})$ local well-posedness can be obtained via the theory of regularity structures \cite{Regularity}.
In this paper we use the theory of \cite{Regularity} to give a proof of local well-posedness for the entire subcritical regime $\beta^2 \in (0,8\pi)$.

Equation~\eqref{e:model} fails to be classically well-posed as soon as $\beta$ is non-zero. 
Formally, writing $u=\Phi+v$ where $\Phi$ solves the linear equation (i.e. the equation with $\beta=0$), $v$ should solve the equation
\begin{equ}[e:model2]
\d_t v = \frac{1}{2}\Delta v  - \frac{i}{2} \bigl( e^{i\beta v} e^{i \beta \Phi} - e^{-i\beta v} e^{i \beta \Phi} \bigr)
 + \zeta\;.
\end{equ}
The field $\Phi$ is, at small scales, a logarithmically correlated Gaussian random field over space-time. In particular, realizations of $\Phi$ are not functions on space-time but rather distributions. 
However one can still try to give meaning to the expressions $e^{\pm i \beta \Phi}$ by employing Wick renormalisation. 
A similar construction can be found in \cite{RhodesVargas} where the authors call the processes constructed an \emph{imaginary Gaussian multiplicative process}.

Denoting the processes defined via Wick renormalisation by $\Wick{e^{\pm i \beta \Phi}}$, we note that these
make sense as long as $\beta^2 < 8\pi$. 
One can then apply Kolmogorov's theorem to show that the realizations of $\Wick{e^{\pm i \beta \Phi}}$ are regular enough for \eqref{e:model2} to be classically well-posed for $\beta^{2} \in (0,4\pi)$.
Realizations of $\Wick{e^{\pm i \beta \Phi}}$ become more singular as $\beta^{2}$ increases \dash  
when $\beta^2 \in [4\pi,8\pi)$ the product  $\Wick{e^{\pm i \beta \Phi}}\, e^{\pm i \beta v}$ fails to be canonically defined and so the equation~\eqref{e:model2} also fails to be classically well-posed. 
One must then use more sophisticated renormalisation procedures to define these products. 
By introducing a partial perturbative expansion for $v$ in terms of $\Wick{e^{\pm i \beta \Phi}}$, 
one sees that a definition of these products requires the construction, via renormalisation, 
of more complicated stochastic objects. 
In particular there are an infinite number of thresholds,
\begin{equ} [e:thresholds]
\beta^2 = \frac{8n \pi}{n+1}\qquad n=1,2,3,\cdots,
\end{equ}
where one encounters new divergent stochastic objects which must be renormalised. 
The paper \cite{HaoSG} was only able to handle the renormalisation of first such divergence so their well-posedness result fell short of the $n=2$ threshold.   

At each fixed value of $\beta^2 < 8 \pi$ there are only finitely many stochastic objects to renormalise but a proof that proves local well-posedness for the the entire regime $(0,8\pi)$ must be robust enough to implement the renormalisation of an arbitrarily large number of divergences. 

When $\beta^2 = 8\pi$ the equation~\eqref{e:model} is \emph{critical}\footnote{The value $\beta^2 = 8\pi$ corresponds to the critical point 
$(\beta,\gamma) = (0,\sqrt{2d})$ in \cite{RhodesVargas}.} and when $\beta^{2} > 8 \pi$ the model becomes super-critical.  
In these situations proving local well-posedness falls completely outside of the scope of our methods and in fact one expects that that any convegent renormalisation scheme will yield a solution which is ``trivial'' in the sense of field theory, that is a solution to a stochastic heat equation. 

The years following the publication of \cite{Regularity} have seen the development of robust, model-independent theorems that have automated many aspects of the renormalisation procedure. 
The paper \cite{BHZalg} described the general algebraic structure of renormalisations, the paper \cite{CH} developed systematic moment bounds on the renormalised models, and finally \cite{BCCH} characterizes how renormalisation modifies equations. 
The combination of these tools gives an automatic and self-contained ``black box'' for obtaining local well-posedness for a wide class of SPDEs.
However, while the algebraic and analytic results of \cite{Regularity,BHZalg,BCCH} do apply here,
the general method used for obtaining stochastic bounds in \cite{CH} does \textit{not} apply for the 
sine-Gordon model since there are several specific features of this model that makes it different in how, at an analytic level, divergences arise and are cancelled.

For the specialist we give a quick explanation of these differences.
Equation~\eqref{e:model2} can be interpreted as a generalised parabolic Anderson model driven by two non-Gaussian driving noises $\xi_{\pm} \eqdef \Wick{e^{\pm i \beta \Phi}}$\ . As in \cite{HS15,CS16}, one of the first steps of the approach of \cite{CH} when estimating a particular stochastic object is to perform a cumulant expansion. 
Each $n$-th cumulant is treated as a singular kernel of $n$-space-time variables and contributes a divergent power counting factor equal to the sum of the regularity exponents of the particular noises they connect. 
One then cancels divergent substructures inside of our stochastic object at the level of cumulants, the only terms in the cumulant expansion where we can exploit a renormalisation cancellation for a divergent structure are those when no noise inside the structure is connected to a noise outside the structure. 
Such a restriction is important due to the fact that we are renormalising individual stochastic objects before we take expectations as opposed to taking expectations and then renormalising all divergent structures appearing.

Performing such a cumulant expansion and power-counting analysis is a dead-end for the sine-Gordon model as the power-counting is too brutal and leads to apparent divergences we cannot renormalise.  
Our ability to suitably renormalise this model relies on taking advantage of the fact that the $\xi_{\pm}$ are imaginary Gaussian multiplicative chaoses related to each other via complex conjugation, rather than being somewhat
arbitrary fields with the same regularity. 
In particular, for any $N \ge 0$ and ``charge assignment'' $q:\{1,\dots,N\} \rightarrow \{+,-\}$ one has
\begin{equ}\label{eq: coulomb interaction}
\E\Big[ 
\prod_{i=1}^{N}
\xi_{q(i)}(z_{i})
\Big]
\approx
\prod_{
1 \le i < j \le N}
|z_{i} - z_{j}|^{-q_{i,j}\beta^{2}/8\pi}\;,
\end{equ}
where $q_{i,j}$ is $1$ for pairs $(i,j)$ of equal charge and $-1$ for pairs of opposite charge. 
One consequence of this power counting is that it is not just the regularity of the kernels and the noises present in a structure that determine whether it is divergent but also the charge assignment of the noises. Indeed, 
we will see that only ``neutral charge'' structures are divergent. 
This was already exploited in \cite{HaoSG} but is incompatible with the expansion of \cite{CH} where such 
charge cancellations cannot be exploited. 

A second key observation is that if we restrict ourselves to renormalising neutral divergent structures we can use take advantage of charge and parity cancellations to harvest renormalisation cancellations, 
even in situations where the noises of a divergent structure interact with those outside of that structure. 
Our proof then proceeds by modifying the approach of \cite{CH} in order to allow us to systematically 
take advantage of these observations. 

We now discuss previous work on related models. 
Work on the static sine-Gordon model includes the convergence of correlation functions for the continuum limit when $\beta^{2} \in (0,8\pi)$ and scaling limits for $\beta^{2} > 8 \pi$ \cite{Jurg,MR649810,MR702570,MR849210,MR1777310}. 
At $\beta^{2} = 8 \pi$ the scaling limit of the model is of great interest since it describes the critical point of Berezinskii-Kosterlitz-Thouless phase transition \cite{berezinskii1972destruction,KT,MR634447}.
Convergence of the free energy and proofs of critical exponents at this critical point were obtained in \cite{Falco,Falco2}.
Returning to the dynamic setting, in addition to its relation to equilibrium statistical mechanics and field theory the model \eref{e:model} has also been proposed as a model for 
the dynamic of crystal-vapour interfaces
at the roughening transition \cite{DynRough,Neudecker} and as a model of crystal
surface fluctuations in equilibrium \cite{SurfModel2,SurfModel1}. 
With regards to the dynamic setting the only previous mathematically rigorous work is \cite{HaoSG} and \cite{AlbRuss}
\footnote{
The article \cite{AlbRuss} considered a class of nonlinearities of the type $\lambda \Wick{f(A u)}$ for small parameters $\lambda$ and $A$.
The interpretation of the
solutions in \cite{AlbRuss} is that of a random Colombeau generalised function, but it is not
clear whether this generalised function represents an actual distribution. 
The construction given there is 
impervious to the presence of the Kosterlitz-Thouless transition and 
the sequence of thresholds   \eqref{e:thresholds}.
}  \dash we also mention \cite{GarbanLiouville} which treats a corresponding ``$\sinh$-Gordon''-type model.

The paper is organised as follows. 
The remainder of the first section gives our main result (Theorem~\ref{theo:v-equ}) and introduces some basic conventions used throughout the paper. 
In Section~\ref{sec: sg reg struct} we discuss how we formulate our problem in the setting of regularity structures: 
in addition to describing the specific regularity structure we work with, we also introduce some auxiliary notation to keep track of the special features of the sine-Gordon equation such as charge cancellation, and also describe how the main theorem will follow from combining certain stochastic estimates  
(Theorem~\ref{theorem:mom-bounds})
with the pre-existing machinery of the theory.

The remainder of the paper is then devoted to obtaining these stochastic estimates stated in Theorem~\ref{theorem:mom-bounds}. 
In Section~\ref{sec: explicit formula} we write out explicit formulas for the moments of the stochastic objects we want to estimate. 
We start with the formulae (given in Proposition~\ref{explicit formula})
that are similar with the formulae 
 of \cite{CH},
  and then show how we can modify these formulae (see Proposition~\ref{prop:mom-formula-z}) by leveraging charge and parity cancellations.
  
 Based on these moment formulae the proof of 
 Theorem~\ref{theorem:mom-bounds} is then carried out by a multiscale decomposition and grouping terms in this decomposition using the notion of ``intervals'' in Section~\ref{sec:Multiscale expansion}. In Section~\ref{sec: renormalisation bound} we show that summing over scales within each group gives the desired bounds.

\subsection*{Acknowledgements}

{\small
AC gratefully acknowledges financial support from the Leverhulme Trust via an Early Career Fellowship, ECF-2017-226.
HS is partially supported by the NSF award DMS-1712684.
MH gratefully acknowledges support from the Leverhulme Trust via a leadership award, as well as from the ERC 
via Consolidator grant 615897:CRITICAL.
}
\subsection{Main result}\label{sec: main result}
Our solution to \eqref{e:model} can be defined as a stochastic limit of solutions to \emph{renormalised}, \emph{mollified} approximations to \eqref{e:model}.
To that end we take a smooth function $\rho: \R \times \T^{2} \rightarrow \R$ supported on the ball of radius $1$, integrating to $1$, satisfying $\rho(t,x) = \rho(t,-x)$ for all $t \in \R$, $x \in \T^{2}$,
and set for $\eps > 0$
\begin{equ}\label{def: scaled mollifier}
\rho_{\eps}(t,x) \eqdef \eps^{-4}\rho(\eps^{-2}t,\eps^{-1}x)\,.
\end{equ}
We then mollify the equation by replacing the rough driving noise $\zeta$ 
with $\zeta_\e \eqdef \zeta \ast \rho_\e$ where ``$*$'' denotes space-time convolution.
Our result then says that for any $\beta^{2} \in (0,8\pi)$, any $u_{0} \in \CC^{\eta}(\T^{2})$, where $\eta \in (\frac{\beta^{2}}{8\pi}-1,0)$ and $\CC^{\eta}(\T^{2})$ is the classical Holder-Besov space, there is space-time process $u$ such that  the classical solution $u_\e$ to the Cauchy problem
\begin{equs}\label{eq: renormalised approx equation}
\d_t u_\e
&= \frac{1}{2}\Delta u_\e + \Wick{ \sin\bigl( \beta u_\e\bigr)}
 + \zeta_\e,   \quad u_\e(0,) = u_{0}, \\
&
 \mbox{with}
 \quad
 \Wick{\sin\bigl( \beta u_\e \bigr)}
 = C_{\beta,\rho,\e}   
  \sin\bigl( \beta u_\e \bigr),
\end{equs}
where $C_{\beta,\rho,\e}$ is a $\beta$,$\rho$ and $\e$ dependent \emph{renormalisation constant}, defined below, then as $\eps \downarrow 0$ the sequence $u_\e$ will  converge to $u$ in probability. 

Our analysis of \eqref{eq: renormalised approx equation} starts with the Da Prato--Debussche trick \cite{MR1941997,MR2016604}. We define the stationary space-time process
\begin{equ}[e:defPhi]
\Phi_\e \eqdef K * \zeta_\e \;,
\end{equ}
where $K\colon \R \times \R^2 \to \R$ is a compactly supported function which 
agrees with the heat kernel $\exp(-|x|^2/2t)/(2\pi t)$ in a ball of radius $\tfrac{1}{2}$ around the origin, is smooth
everywhere except at the origin, satisfies $K(t,x) = 0$ for $t < 0$, 
and has the property that $\int K(t,x) Q(t,x)\,dt\,dx = 0$
for every polynomial $Q$ of degree $2$. 

Note that
$R_\e \eqdef \d_t \Phi_\e - \frac{1}{2}\Delta \Phi_\e -  \zeta_\e$
is a smooth function and one can show that there exists a smooth function $R$ such that $R_\e$ converges to $R$.
We then set
\[
C_{\beta,\rho,\e} \eqdef e^{\frac{\beta^2}{2} \E(\Phi_\e (0)^2)}\;.
\]
It is shown in \cite[Lemma~3.1]{HaoSG}
that there exists a constant $\bar{C}_{\beta,\rho}$, independent of our choice of $K$, 
such that 
one has 
$C_{\beta,\rho,\e} = \bar{C}_{\beta,\rho} \eps^{-\frac{\beta^2}{4\pi}} (1+O(\eps^2))$.

Writing $u_\e = v_\e +  \Phi_\e$ and we have 
\begin{equ} [e:SG]
\partial_t v_\e= \frac12\Delta v_\e -\frac{i}{2} 
   \Big(e^{i\beta v_\e} \xi^{\e}_{+} 
   - e^{-i\beta v_\e} \xi^{\e}_{-}\Big)  + R_\e,
 \quad 
v_\e (0,\cdot) = v^\e_{0} \eqdef u_{0} - \Phi_\e(0,\cdot), 
\end{equ}
where we set 
\begin{equ} [e:def-xipm]
\xi^{\e}_{\pm} 
\eqdef \Wick{e^{\pm i\beta \Phi_\e }}
= C_{\beta,\rho,\e} e^{\pm i\beta \Phi_\e }
 \;.
\end{equ}
Note that the choice of $C_{\beta,\rho,\e}$ chosen  above
is such that $\E (\xi^{\e}_{\pm} (z))= 1$ for any $z\in \R\times \T^2$. 
It is proved in \cite[Theorem 2.1]{HaoSG} that
for $\beta^{2} < 8 \pi$, there exist random distributions 
\begin{equ}[e:def of xi]
\xi = (\xi_{-},\xi_{+})
\end{equ}
such that for any choice of $\rho$ and heat kernel truncation, 
$\xi^{\eps}_{\pm}$ converges in probability
as $\e\to 0$
to $\xi_{\pm} $ in the topology of $ \CC^{-\bar\beta} \times \CC^{-\bar\beta}$ for any $\bar \beta>\beta'  \eqdef  \frac{\beta^2}{4\pi}$.

We now fix a choice of $\beta \in (0,\sqrt{8\pi})$ for which we will seek to solve \eqref{e:model}; note that this means $\beta' \in(0,2)$. 
We then fix some choice of $\bar{\beta} \in ( \beta' ,2)$ which corresponds to our assumption on the pathwise regularity of $\xi_{\pm}$. 
The focus of the paper will be developing tools to prove the convergence as $\eps \downarrow 0$ for the $v_\e $ of Eq.~\eqref{e:SG}.

With the following theorem at hand, one then has that 
as $\e \to 0$, the solutions $u_\e $ converge to 
$v$ (given by the following theorem) plus the distributional (Gaussian) limit of $\Phi_\e$.

\begin{theorem} \label{theo:v-equ}
Assume that  $u_{0} \in \CC^\eta(\T^2)$ for 
some $\eta \in (\frac{\bar\beta}{2}-1,0)$. 
The sequence
$v_\e$ converges in probability and locally uniformly as $\eps\to 0$ to a 
limiting process $v$.

More precisely, there exists a stopping time $\tau > 0$ and a random variable 
$v \in \CD'(\R_+ \times \T^2)$
such that, for every $T>0$, one has $v_\e\to v$ in probability in 
$\CC([0,T], \CC^\eta(\T^2)) \cap \CC((0,T], \CC^{-\bar\beta+2}(\T^2))$
on the set $\tau > T$.
Finally, one has $\lim_{t \to \tau} \left\Vert v(t,\cdot)\right\Vert _{\mathcal{C}^{\eta}(\mathbf{T}^{2})}=\infty$
on the set $\{\tau < \infty\}$. 
The limiting process $v$ 
does not depend on the choice of mollifier $\rho$.
\end{theorem}

\subsection{Preliminary notation}  \label{sec:Preliminary notation}
\subsubsection{Space-time scaling}
Given $z \in \R_{+} \times \T^{2}$ or $z \in \R^{3}$ we often write $z = (z^{(0)},z^{(1)},z^{(2)})$. 
We also fix a parabolic space-time scaling $\s = (2,1,1)$ and a corresponding metric on $\R^{3}$
\[
|z|_{\s} \eqdef |z^{(0)}|^{\frac{1}{2}} + |z^{(1)}| + |z^{(2)}|\;.
\]
With this choice of metric, $\R^{3}$ has scaling dimension $|\s| = 4$. 
Let $\N$ be the set of natural numbers with convention $0\in \N$.
Given a multi-index $k = (k^{(0)},k^{(1)},k^{(2)}) \in \N^{3}$ we define the $\s$-degree of $k$ via 
\[
|k|_{\s}
\eqdef
2 k^{(0)} + k^{(1)} + k^{(2)}\;.
\]
Also, given a finite set A, a map $\mfn: A \rightarrow \N^{3}$, and any subset $B \subset A$ we define the shorthand $|\mfn(B)|_{\s} \eqdef \sum_{b \in B} |\mfn(b)|_{\s}$. 

Given  function $\psi :\R^3\to \R$ 
we set $\psi_z^\lambda(y)\eqdef \lambda^{-4} \psi(\tfrac{y^{(0)}-z^{(0)}}{\lambda^2},\tfrac{y^{(1)}-z^{(1)}}{\lambda},\tfrac{y^{(2)}-z^{(2)}}{\lambda})$.

\subsubsection{Singular Kernels}
Additionally, following \cite[Sec.~10.3]{Regularity}, for any $\zeta \in \R$
 and $m \in \N$ we introduce norms $\|\cdot\|_{\zeta,m}$ on smooth functions $K: \R^{3} \setminus \{ 0\} \rightarrow \R$ by setting 	
\begin{equ}\label{def: singular kernel norm}
\|K\|_{\zeta,m}
\eqdef
\max_{
\substack{
k \in \N^{3}\\
|k|_{\s} \le m
}
}
\sup_{x \in \R^{3}\setminus \{0\}}
|x|^{|k|_{\s} + \zeta}
|D^k K(x)|
\;.
\end{equ}
When $m=0$ we sometimes write $\|K\|_{\zeta}=\|K\|_{\zeta,0}$ for short.
%

%
\subsubsection{Set theory notation}
\label{sec:poset}
Throughout the paper we use the common shorthand {\it poset} for partially ordered sets. 
We introduce a variety of posets but use common notation with them all. Given a poset $P$ and a subset
\footnote{In this paper the notation $A\subset B$ includes the case $A=B$. We will not use notation such as $\subseteq$.}
 $A \subset P$ 
we say $a \in A$ is \emph{maximal} if there does not exist an element $a' \in A$ with $a' > a$. 
Similarly we say an element $a \in A$ is $\emph{minimal}$ if there does not exist an element $a' \in A$ with $a' < a$. 
We often write $\mathrm{Max}(A)$ or $\overline{A}$ for the set of maximal elements of $A$ and $\mathrm{Min}(A)$ or $\underline{A}$ for the set of minimal elements of $A$. 

For any set $A$ we denote by $A^{(2)}$ the collection
of all the subsets of $A$ of cardinality two.
\section{The sine-Gordon regularity structure}\label{sec: sg reg struct}
We quickly recall the notion of a \emph{regularity structure} introduced in \cite{Regularity}. 
We refer readers looking for a detailed exposition to \cite{hairer2015introduction}, \cite[Chapter~15]{FrizHairer}, and \cite{CW}; our description of the theory will be quite brief. 
The most basic object in the theory is a \emph{regularity structure} which consists of a pair $(\mcb{T}, \mathcal{G})$.
Here, $\mcb{T}$ is a graded vector space $\mcb{T} = \bigoplus_{\alpha \in A} \mcb{T}_{\alpha}$ 
for $A \subset \R$ a set of homogeneities assumed to be locally finite and bounded from below,
where each $\mcb{T}_{\alpha}$ is a Banach space which will in our case be finite-dimensional
and come with a distinguished basis. $\CG$ is a group of continuous linear transformations on $\mcb{T}$ with the property that for all $\alpha \in A,\ \tau \in \mcb{T}_{\alpha},$ and $\Gamma \in \CG$ one has $(\Gamma \tau - \tau) \in \mcb{T}_{<\alpha} = \bigoplus_{\beta < \alpha} \mcb{T}_{\beta}$. 
A regularity structure is used to describe ``jets of abstract Taylor expansions''; the vector space $\mcb{T}$ is the target space for the jets and the \emph{structure group} $\CG$ includes transformations on the target space corresponding to change of base-point operations.

\subsection{The trees of \texorpdfstring{$\mcb{T}$}{T}}
\label{sec: first description of trees}
In our setting the vector space $\mcb{T}$ will be a free vector space generated by a finite collection $\mcT$ of abstract elements, and the elements of $\mcT$ will be represented by certain decorated, rooted trees.
In Section~\ref{sec: invoking BHZalg}
we will explain how these combinatoric trees are generated 
from the equation via the construction in \cite{BHZalg}.

We often write one of such trees as $T^{\mfn\mfl}$.
Here, the underlying  tree $T$ consists of a set of nodes $N(T)$, a distinguished root node $\rho_{T} \in N(T)$, a set of edges $K(T) \subset N(T)^{2}$.
The superscript $\mfl$ is a  decoration, namely a map $\mfl: N(T) \rightarrow \{ +, 0, -\}$. 
The superscript $\mfn$ corresponds to a second decoration $\mfn:N(T) \rightarrow \N^{3}$.

We will frequently use the notation $\tilde{N}(T) \eqdef N(T) \setminus \{\rho_{T}\}$.

For such a tree $T$ we view $N(T) \cup K(T)$ as {\it partially ordered} (see Section~\ref{sec:Preliminary notation}) as follows: given two nodes / edges $u$ and $v$, 
one has $u \le v$ if and only if the unique path that connects $v$ to the root $\rho_{T}$ contains $u$. 
(In particular $\rho_{T}=\mathrm{Min}(T)$.)
We impose that edges $e\in K(T)$ are directed in such a way that 
$e = (e_{\ch}, e_{\p})$ 
with $e_{\ch} > e_{\p}$. 
We henceforth view $e_{\p}$, $e_{\ch}$ as maps from $K(T)$ into $N(T)$.
We also write 
\[
L(T) \eqdef \{ u \in N(T):\ \mfl(u) \not = 0\}\;.
\]
The {\it homogeneity} $|T^{\mfn\mfl}|_{\s}$ of a tree $T^{\mfn\mfl}$ is given by
\begin{equ} [e:def-snorm]
|T^{\mfn\mfl}|_{\s} \eqdef 2|K(T)|-\bar\beta |L(T)| + \sum_{u \in N(T)} |\mfn(u)|_{\s}\;.
\end{equ}
Here, $|\cdot|$ stands for the cardinality of sets and
the factor $2$ corresponds to the fact that 
the heat kernel ``improves regularity by two''.
The following is an example of a tree $T^{\mfn\mfl}$ with $|N(T)|=|L(T)|=6$,
$|K(T)|=5$ and
$|T^{\mfn\mfl}|_{\s} = 10-6\bar\beta$.
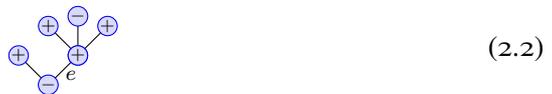
\begin{equ} [e:example-of-tree]
\begin{tikzpicture}[scale=0.13,baseline=0]
\draw (-3,0) node[xi] {\tiny $ +$}
	-- (0,-3) node[xi] {\tiny $ -$} 
	-- (3,0) node[xi] {\tiny $+$}
	 -- (0,3) node[xi] {\tiny $+$};
\draw (3,4) node[xi] {\tiny $-$} -- (3,0)   node[xi] {\tiny $+$}
	-- (6,3) node[xi] {\tiny $+$};
\node at (2.3,-2) {\scriptsize $e$};
\end{tikzpicture}
\end{equ}
(Whenever we use such a graphical notation, we implicitly set $\mfn = 0$.)
Using the fact that $|K(T)| = |N(T)| - 1$ one can also write \footnote{This formula is useful in Sections~\ref{sec: renormalisation bound}.}
\hao{footnote:  make more precise reference to that section}
\[
|T^{\mfn\mfl}|_{\s} = - 2|K(T)| - \bar\beta |L(T)| + \sum_{u \in N(T)} |\mfn(u)|_{\s} + (|N(T)| - 1)|\s|
\] 

We define $\tilde{\mcT}$ to be the collection of all such trees $T^{\mfn\mfl}$ and then set $\mcT \subset \tilde{\mcT}$ as
\begin{equ} [e:def-mcT]
\mcT \eqdef \{ T^{\mfn\mfl} \in \tilde{\mcT}:
	|T^{\mfn\mfl}|_{\s} < {\mu} \}
\end{equ}
 for $\mu \in (\bar{\beta},2)$. 
 Finally, we denote by $\mcb{T}$ the free vector space generated by $\mcT$. 

We write $\Xi_{+}$ (resp. $\Xi_{-}$) for the tree $T^{0}$ with $T = \{\rho_{T}\}$ and $\mfl(\rho_{T}) = +$ (resp. $\mfl(\rho_{T}) = -$). 
For $n \in \N^{3}$ we write $\mathbf{X}^{n}$ for the tree $T^{\mfn\mfl}$ with $T = \{\rho_{T}\}$, $\mfn(\rho_{T}) = n$, and $\mfl(\rho_{T}) = 0$.  

It is natural to view $\tilde{\mcT}$ as being iteratively generated via two operations applied to the set $\{\Xi_{+}, \Xi_{-}\} \sqcup \{ \mathbf{X}^{n}: n \in \N^{3}\}$ as follows.

1. Given $\tau = T^{\mfn} \in \tilde{\mcT}$ we define a tree $\CI(\tau)\in \tilde{\mcT}$ as follows. 
Writing $\tilde{T}^{\tilde{\mfn}\tilde{\mfl}} \eqdef \CI(\tau)$, we build this tree 
by introducing a new node $\rho_{\tilde{T}}$ (which is defined as the root of $\tilde T$) and
setting $N(\tilde{T})=N(T)\sqcup \{\rho_{\tilde{T}}\}$
and $K(\tilde{T}) = K(T)\sqcup \{(\rho_{\tilde{T}},\rho_{T})\}$. 
We then set $\mfl(\rho_{\tilde{T}}) = 0$, $\tilde{\mfn}(\rho_{\tilde{T}}) = 0$ and the other nodes of $\tilde T$ inherit labels $\mfl$ and $\mfn$ from $T$.

2. Given two trees $T_{1}^{\mfn_1\mfl_1}, T_{2}^{\mfn_2\mfl_2} \in \tilde{\mcT}$ such that at least one of the two labels $\mfl(\rho_{T_1})$ and $\mfl(\rho_{T_2})$ is $0$,
 we define the tree product  
 $T_{3}^{\mfn_3\mfl_3} \eqdef T_{1}^{\mfn_1\mfl_1} \cdot T_{2}^{\mfn_2\mfl_2}\in \tilde{\mcT}$ to be  the disjoint union of $T_1$ and $T_2$
with roots $\rho_{T_1}$ and $\rho_{T_2}$ identified giving the root $\rho_{T_3}$ of $T_3$;
the decorations $\mfl$ and $\mfn$ are pushed through to $\tilde{N}(T_3)$ from $\tilde{N}(T_1)$ and $\tilde{N}(T_2)$,
and  
we set $\mfn_3 (\rho_{T_3}) = \mfn_1(\rho_{T_1}) + \mfn_2(\rho_{T_2})$ and 
$\mfl_3 (\rho_{T_3}) = \mfl_1(\rho_{T_1}) + \mfl_2(\rho_{T_2})$ (in the latter case``plus'' is understood in the natural way e.g.: $-$ plus $0$ equals $-$).


The following picture illustrates the above two operations.
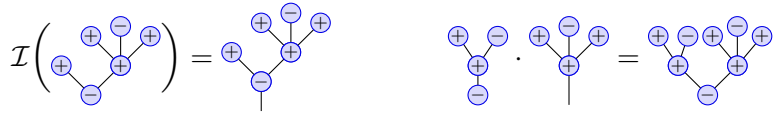
\begin{equ} 
\CI\bigg(
\begin{tikzpicture}[scale=0.13,baseline=0]
\draw (-3,0) node[xi] {\tiny $ +$}
	-- (0,-3) node[xi] {\tiny $ -$} 
	-- (3,0) node[xi] {\tiny $+$}
	 -- (0,3) node[xi] {\tiny $+$};
\draw (3,4) node[xi] {\tiny $-$} -- (3,0)   node[xi] {\tiny $+$}
	-- (6,3) node[xi] {\tiny $+$};
\end{tikzpicture}
\bigg)
=
\begin{tikzpicture}[scale=0.13,baseline=-5]
\draw (-3,0) node[xi] {\tiny $ +$}
	-- (0,-3) node[xi] {\tiny $ -$} 
	-- (3,0) node[xi] {\tiny $+$}
	 -- (0,3) node[xi] {\tiny $+$};
\draw (3,4) node[xi] {\tiny $-$} -- (3,0)   node[xi] {\tiny $+$}
	-- (6,3) node[xi] {\tiny $+$};
\draw (0,-3) node[xi] {\tiny $ -$} -- (0,-6) {};
\end{tikzpicture}
\qquad
\qquad
\begin{tikzpicture}[scale=0.13,baseline=0]
\draw 
	 (3,-3)  node[xi] {\tiny $-$}
	-- (3,0) node[xi] {\tiny $+$}
	 -- (1,3) node[xi] {\tiny $+$};
\draw (5,3) node[xi] {\tiny $-$} -- (3,0)   node[xi] {\tiny $+$};
\end{tikzpicture}
\cdot
\begin{tikzpicture}[scale=0.13,baseline=0]
\draw 
	 (3,-4)  
	-- (3,0) node[xi] {\tiny $+$}
	 -- (0,3) node[xi] {\tiny $+$};
\draw (3,4) node[xi] {\tiny $-$} -- (3,0)   node[xi] {\tiny $+$}
	-- (6,3) node[xi] {\tiny $+$};
\end{tikzpicture}
=
\begin{tikzpicture}[scale=0.13,baseline=0]
\draw (-3,0) node[xi] {\tiny $ +$}
	-- (0,-3)  node[xi] {\tiny $ -$}
	-- (3,0) node[xi] {\tiny $+$}
	 -- (0.5,3) node[xi] {\tiny $+$};
\draw (-5,3) node[xi] {\tiny $ +$}
	-- (-3,0) node[xi] {\tiny $ +$}
	-- (-2,3) node[xi] {\tiny $ -$};
\draw (3,4) node[xi] {\tiny $-$} -- (3,0)   node[xi] {\tiny $+$}
	-- (5.5,3) node[xi] {\tiny $+$};
\end{tikzpicture}
\end{equ}

Note that we have $|\CI[\tau]|_{\s} = |\tau|_{\s} + 2$ and $|\tau \cdot \bar{\tau}|_{\s} = |\tau|_{\s} + |\bar{\tau}|_{\s}$.

We extend the operations $\CI[\cdot]$ and the tree product $\cdot$ to $\mcb{T}$ by linearity 
(truncated at homogeneity $\mu$, i.e. when the operation yields a tree in  $\tilde{\mcT} \setminus \mcT$ then it is set to $0$).

The following crucial lemma reflects 
the sub-criticality of \eqref{e:SG}.

\begin{lemma}\label{lem: SG trees are nice}
For any $\tau \in \tilde{\mcT} \setminus \{ \Xi_{+}, \Xi_{-}\}$ one has
\begin{equ}\label{eq: SG trees are nice}
|\tau|_{\s} > |\Xi_{\pm}|_{\s} = -\bar{\beta} > -\frac{|\s|}{2} = - 2\;.
\end{equ}
\end{lemma}
\begin{proof} 
First we observe that the claim is obvious if $\tau = \mathbf{X}^{k}$ for $k \in \N$.
We prove the claim for $\tau = T^{\mfn\mfl}$ with $|K(T)| \ge 2$ via induction in $|K(T)|$. 
The base case, when $|K(T)| = 2$, is easily verified.

For the inductive step fix $n > 2$ and suppose the claim holds for any $S^{\mfn\mfl} \in \tilde{\mcT}$ with $2 \le |K(S)| < n$. We then try to prove the claim for some fixed $T^{\mfn\mfl} \in \tilde{\mcT}$ with $|K(T)| = n$. Without loss of generality we can assume $\mfl(\rho_{T}) = +$ and $\mfn(\rho_{T}) = 0$. 
Now note that one can find $j \ge 1$ such that we can write $T^{\mfn\mfl} $ as a tree product 
\[
T^{\mfn\mfl} = \Xi_{+} \cdot \CI(\tau_{1}) \cdots \CI(\tau_{j})
\] 
where $\tau_{1},\dots,\tau_{j} \in \tilde{\mcT}$, $|K(\tilde{T})| = 0$, and for each $i \in [j]$ one has $\tau_{i} = T_{i}^{\mfn_{i}\mfl_{i}}$ with $|K(T_{i})| < n$. 
Then we have
\[
|T^{\mfn\mfl}|_{\s} =
- \bar{\beta} + \Big( \sum_{i=1}^{j} 2 + |T_{i}^{\mfn_{i}\mfl_{i}}|_{\s} \Big)
\ge - (j+1) \bar{\beta} + 2j > -\bar{\beta}\;,
\]
since $|T_{i}^{\mfn_{i}\mfl_{i}}|_{\s} \ge - \beta$ for every $i \in [j]$ \dash this is immediate if $|K(T_{i})| = 1$ and via our inductive hypothesis otherwise.
\end{proof}
We define $\mcT^{-} \eqdef \{T^{\mfn\mfl} \in \mcT : |T^{\mfn\mfl}|_{\s} < 0 \}$.

\begin{lemma} \label{lem:neg-trees}
Let $T^{\mfn\mfl} \in \mcT^{-}$, then: 
\begin{itemize}
\item One has $L(T) = N(T)$.
\item Either $\mfn = 0$, or $\mfn$ is non-vanishing at precisely one node $u \in N(T)$ with $\mfn(u) \in \{(0,1,0),(0,0,1)\}$. 
\end{itemize}
\end{lemma}
\begin{proof}
For the first statement, suppose that $T^{\mfn\mfl} \in \mcT^{-}$
and there exists $u\in N(T)\setminus L(T)$.
Define $T^{\mfn\mfl'} $ to be the same tree as $T^{\mfn\mfl} $ except that we set $\mfl'(u) = +$ and $\mfl'(v)=\mfl(v)$ for $v\neq u$. Obviously $T^{\mfn\mfl'} $ 
also belongs to the regularity structure, but by \eqref{e:def-snorm} one has 
$| T^{\mfn\mfl'} |_\s = |T^{\mfn\mfl} |_\s -\bar\beta<  -\bar\beta$,
contradicting Lemma~\ref{lem: SG trees are nice}.

For the second statement,
if $|T^{\mfn\mfl} |_\s<0$ and $T^{\mfn\mfl}$ doesn't satisfy the condition of the statement, then
$T^{0\mfl}$ also belongs to the regularity structure, but by \eqref{e:def-snorm} one has
$|T^{0\mfl} |_\s<-2$, contradicting again Lemma~\ref{lem: SG trees are nice}.
\end{proof}
%
%
%
\hao{Since in this subsection we didn't really give  the groups,
I changed the title from ``The structure group and renormalisation group'' to ``Invoking the algebraic machinery''.}


\subsection{Invoking the algebraic machinery}
\label{sec: invoking BHZalg}
The above description of the trees of $\mcT$ is fairly simple,
 but a systematic construction of a sufficiently rich structure group $\CG$ for our regularity structure and renormalisation group $G_{-}$ is fairly non-trivial. 
A very general (equation-independent) construction was given in \cite{BHZalg}; we now describe the trees introduced in the earlier section in the language of \cite{BHZalg} and describe how we invoke the machinery of that paper. 
Here we will not give a full exposition of the construction of \cite{BHZalg}. Readers who want a more detailed and pedagogical explanation should look at \cite{KPZg}. 
Readers who are willing to take the construction of \cite{BHZalg} as a complete black box can skip this subsection. 

The trees introduced in Section~\ref{sec: first description of trees} should be thought of as \emph{placeholders} for multi-linear functionals of the driving noise $\xi$ that appear after writing \eqref{e:SG} in its mild formulation and then performing Picard iteration to generate a formal expansion. 
Given a tree $T^{\mfn\mfl}$, the edges $K(T)$ correspond to convolutions with the heat kernel, and at each node the decoration $\mfn$ describes what polynomial is at that node while the decoration $\mfl$ tells us if there is no noise or an instance of $\xi_{\pm}$\footnote{In the framework of \cite{BHZalg} the presence of driving noises in a tree sometimes represented by fictitious edges which are required to be maximal in the tree $T$. Encoding this data via the node decoration $\mfl$ instead is only a cosmetic difference}.

In \cite{BHZalg} the first ingredient is a finite set of \emph{types} $\Lab$ with a partition $\Lab = \Ke \sqcup \Le$. The elements of $\Ke$ and $\Le$ are respectively called \emph{``kernel types''} and \emph{``noise types''}.
Since we have a single equation 
  rather than a system of equations (and thus only one kernel which is the heat kernel), we write 
 $\mfL_{+} \eqdef \{ \mft \}$.
Since we have two noises $\xi_{+}$ and $\xi_{-}$ in our equation \eqref{e:SG}, we have $\mfL_{-} \eqdef \{+, -\}$. 
We also assign a homogeneity  $|\cdot|_{\s}$ on $\Lab$: we set $|\mft|_{\s} = 2$ and $|+|_{\s} = |-|_{\s} = - \bar{\beta}$. 

The set of all ``edge type'' in \cite{BHZalg} is given by  $\mfL \times \N^{3}$.
The component in $\N^{3}$ for a fixed edge type indicates the presence of a space-time derivative on the respective kernel\slash noise. 
The regularity structure we work with will not involve any derivatives of its integration map (i.e. the component in $\N^{3}$ will always be zero), so we use the 
shorthands $\CI \eqdef (\mft,0)$, $\Xi_{\pm} = (\pm,0)$.

Recall from \cite[Sec.~5]{BHZalg} that a {\it rule} is a map  from the set $\Lab$ into the collections of products of elements of $\Lab \times \N^{3}$ where the product here is commutative.
The rule $R_{\SG}$ we use for the sine-Gordon equation is given by
\begin{equs}
R_{\SG}(+) \eqdef R_{\SG}(-) \eqdef \{\mathbf{1}\} \;,\ 
R_{\SG}(\mft) \eqdef \Big\{ 
\CI^n, \Xi_- \CI^n , \Xi_+ \CI^n  \,:\, n \in \N 
\Big\}\;.
\end{equs}
Here $\mathbf{1}$ denotes the trivial (i.e. empty) product.
In plain words, $R_{\SG}(\mft)$ describes 
how the nodes and edges are allowed to attach from above to 
a given edge; taking the tree in \eqref{e:example-of-tree} as an example, the edge $e$ is attached from above a node of type $\Xi_+$ and $3$ edges, corresponding to a product $\Xi_+\CI^3$.
This rule is clearly \emph{normal} in the sense of \cite[Def.~5.7]{BHZalg} and it is 
also \emph{subcritical} in the sense of \cite[Def.~5.14]{BHZalg}.
\footnote{
To check  subcriticality as in  \cite[Section~5.2]{BHZalg}
we can take the function $\reg$  therein as  $\reg(\mft) \eqdef  7(2 - \beta)/8$ and $\reg(\pm) = -(2 + 7\bar{\beta})/8$.
}

Let $\bar{R}_{\SG}$ be the completion of $R_{\SG}$ as given in \cite[Prop.~5.20]{BHZalg}. 
We then let $\tilde{\mathscr{T}}$ be the reduced regularity structure built from $\bar{R}_{\SG}$ with truncation at homogeneity $\bar\beta$.
Note that the vector space generated by trees which strongly conform to $R_{\SG}$ forms a sector of $\tilde{\mathscr{T}}$, and so we set $\mathscr{T} = (A,\mcb{T},G)$ to be the regularity structure obtained by restricting to this sector\footnote{We work on a sector because the reduced regularity structure will involve trees with edges of the form $(\mft,k)$ for $|k|_{\s} =1$ since $\overline{R}_{SG}$ has to be $\ominus$-complete but we don't need to control the action of models on such trees.}. 
The vector space $\mcb{T}$ here can canonically be identified with the vector space $\mcb{T}$ introduced in Section~\ref{sec: first description of trees}.

\subsection{Charge interactions, function \texorpdfstring{$\mcb{q}$}{q} and homogeneity \texorpdfstring{$|\cdot|_{\SG}$}{SG}} \label{sec:charge}
The correlation structure given by \eqref{eq: coulomb interaction} leads us to interpret the nodes of $L(T)$, for a given a tree $T^{\mfn\mfl} \in \mcT$, as electromagnetic point charges which are either positive or negative. In this subsection we make this interpretation more precise.

%
%
%

For any $\e,\bar\e>0$ we define $J^{\e,\bar{\e}}:\R \times \T^{2} \setminus \{0\} \rightarrow \R_{+}$ via 
\begin{equ}\label{def: single mollified interaction}
J^{\e,\bar{\e}}(z)
\eqdef
\E (\xi^{\e}_{\pm}(0)\xi^{\bar{\e}}_{\pm}(z))
\end{equ}
One then has
\[
\E (\xi^{\e}_{\pm}(0)\xi^{\bar{\e}}_{\mp}(z))
=
\frac{1}{J^{\e,\bar{\e}}(z)}\;.
\]
As in \cite[Eq.~(3.7)]{HaoSG}, one should think of $J^{\e,\bar{\e}}$ as essentially a power function of power $2\beta' = \frac{\beta^2}{2\pi}$. More precisely we will need the following bounds.
\begin{lemma}\label{lem:bound-J}
%
Let $m \ge 0$.  
Then $\|J^{\e,\bar\e}\|_{- 2\beta',m} $, 
$\|(J^{\e,\bar\e})^{-1}\|_{2\beta',m} $ are bounded uniformly in $\e,\bar\e$,
 and one has
\begin{equ}
\|J^{\bar{\e},\bar{\bar{\e}}} - J^{\e,\e} \|_{- 2\beta' +\kappa,m} 
\; \vee \;
\|(J^{\bar{\e},\bar{\bar{\e}}})^{-1} - (J^{\e,\e})^{-1}  \|_{2\beta' +\kappa,m}
 \;\; \lesssim \;\;
(\eps\vee \bar{\e} \vee \bar{\bar{\e}} )^\kappa
\end{equ}
uniformly in $\e,\bar\e,\bar{\bar{\e}} >0$
for some sufficiently small $\kappa>0$.
\end{lemma}

\begin{proof}
These bounds are proved in the same way as in \cite[Proof of Theorem~3.2]{HaoSG}. Indeed, letting $\CQ_{\e,\bar\e} (z) \eqdef \E(\Phi_\e(0)\Phi_{\bar\e}(z))$
where $\Phi_\e$ is as in \eqref{e:defPhi}, 
with the limiting kernel $\CQ(z)$ given by $\int K(z+\bar z) K(\bar z)\,d\bar z$,
we have $\CQ_{\e,\bar\e} = \rho_{\e} \ast \CQ \ast \hat{\rho}_{\bar\e}$ where $\hat{\rho}_{\bar\e}$ is the space-time reflection of $\rho_{\bar\e}$, and
 $J^{\e,\bar\e}=e^{-\beta^2 \CQ_{\e,\bar\e} }$.
 By \cite[Lemma~3.9]{HaoSG}, 
$ \CQ(z) $ equals  $-{1\over 2\pi} \log |z|_\s $ plus smooth functions.
From this we deduce bounds $|D^k \CQ(z)| \lesssim |z|_\s^{-|k|_\s}$ for $|k|_\s >0$. The bounds on $J^{\e,\bar\e}$ and $1/J^{\e,\bar\e}$ follow immediately.

One also has 
$(J^{\bar{\e},\bar{\bar{\e}}})^{-1} - (J^{\e,\e})^{-1} 
= \CJ^{-1} (e^{\CQ_{\bar{\e},\bar{\bar{\e}}} - \CQ} - e^{\CQ_{\e,\e} -\CQ})$.
As in \cite[Lemma~3.7]{HaoSG} one can easily prove
$
|\CQ_{\e,\bar\e} (z) - \CQ(z) | \lesssim \frac{\e^\kappa}{|z|_\s^\kappa} \wedge (1+|\log(\e/|z|)|) 
$ for $\kappa\in [0,1]$ and $\bar\e \le \e$,
so that one has 
$| (J^{\bar{\e},\bar{\bar{\e}}})^{-1} (z)- (J^{\e,\e})^{-1} (z) |
\lesssim |z|_\s^{-2\beta'} (\frac{\e^\kappa}{|z|_\s^\kappa} \wedge 1) $.
Combining with the derivative bounds on $\CQ$ 
we obtain the claimed bound on $(J^{\bar{\e},\bar{\bar{\e}}})^{-1} - (J^{\e,\e})^{-1} $. The bound on $J^{\bar{\e},\bar{\bar{\e}}} - J^{\e,\e}$ is proved in the same way.
%
%
%
\end{proof}

%
%
Given a set $A$ which has been associated  with a map $\mfl: A \rightarrow \{+,-\}$ (for instance  $A$ is a subset of $L(T)$ for some tree $T$), we adopt the convention of \cite{CH} and define
\[
|A|_{\s}
\eqdef
-\bar{\beta}|A|  \;.
\]
We also define the following integer valued function 
\footnote{Of course evaluating on a {\it single node} $a$ there is not much difference between $\mfq(a)$ and $\mfl(a)$, but we would like to think of $\mfq$ as a function on sets while $\mfl$ is simply a decoration of nodes.}
\begin{equ} [e:total-charge]
\mcb{q}(A) \eqdef 
\sum_{a \in A} \mcb{q}(a)
\quad
\mbox{ where }
\;
\mcb{q}(a) \eqdef \mathbbm{1}\{\mfl(a) = + \} - \mathbbm{1}\{\mfl(a) = - \}
\end{equ}
which is thought as the ``total charges'' of the set $A$.
For any  $\{a,b\} \subset A$ we define
 \begin{equ} [e:def-sign]
 \sign(\{a,b\}) = \mcb{q}(a) \cdot \mcb{q}(b)\;.
 \end{equ}
We then also define
\begin{equs}
|A|_{\SG}
&\eqdef
2\bar{\beta}
\sum_{e \in A^{(2)}}
\sign(e) 
=
2\bar\beta {m_+ \choose 2} +2\bar\beta {m_- \choose 2} - 2\bar\beta m_+ m_-  
\\
&=
-\bar{\beta}|A| + \bar{\beta}\mcb{q}(A)^{2}\;,
    \label{e:formula-ASG}
\end{equs}
where $m_\pm = \#\{a\in A: \mfl(a) = \pm \}$, 
so that $|A|$ (the cardinality of $A$) equals $m_+ + m_-$,
 and $\mcb{q}(A)=m_+ - m_-$.
We overload notation and also use the symbol $|\cdot|_{\SG}$ to denote a {\it new homogeneity} on the trees $T^{\mfn\mfl} \in \mcT$ by setting
\begin{equ}[e:SGnorm-diff-snorm]
|T^{\mfn\mfl}|_{\SG} 
\eqdef
2 \, |K(T)| + |L(T)|_{\SG}
+
\sum_{u \in N(T)} |\mfn(u)|_{\s}\;.
\end{equ}
Here, the factor $2$ again corresponds to the fact that 
the heat kernel ``improves regularity by two''.
Note that comparing with the homogeneity $|\cdot |_{\s}$ 
defined in \eqref{e:def-snorm} the only difference is the second term on the right hand side. For example, for the tree $T^{\mfn\mfl}$ in \eqref{e:example-of-tree}, 
while $|T^{\mfn\mfl}|_{\s} = 10-6\bar\beta$ which is strictly negative in our interested regimes where $\bar\beta$ is smaller but close to $2$,
one has $|T^{\mfn\mfl}|_{\SG}  = 10-2\bar\beta$ which is strictly positive.

We say a tree $T^{\mfn\mfl}$ is \emph{neutral} if $\mcb{q}(L(T)) = 0$. 
We then have the following lemma.

\begin{lemma}\label{non-neutral trees are nice} 
For every $T^{\mfn\mfl} \in \mcT$, the following two statements are equivalent:

1) $|T^{\mfn\mfl}|_{\SG} < 0$.

2) $|T^{\mfn\mfl}|_\s <0$ and $T^{\mfn\mfl}$ is neutral.
\end{lemma}
\begin{proof}
Combining \eqref{non-neutral trees are nice}
with \eqref{e:formula-ASG}, we have
\begin{equ}[e:SGnorm-formula]
|T^{\mfn\mfl}|_{\SG} 
 = |T^{\mfn\mfl}|_{\s} + \bar{\beta} \mcb{q}(L(T))^{2}\;,
\end{equ}
so the implication 2) $\Rightarrow$ 1) is immediate.
On the other hand, suppose that 1) holds and recall that
$ |T^{\mfn\mfl}|_{\s} > -\bar\beta$ by Lemma~\ref{lem: SG trees are nice}. 
If  $T^{\mfn\mfl}$ were not neutral, then \eqref{e:SGnorm-formula} would imply
\[
|T^{\mfn\mfl}|_{\SG} 
\ge |T^{\mfn\mfl}|_{\s} + \bar{\beta} > -\bar\beta + \bar{\beta}=0\;,
\]
in contradiction to 1).
\end{proof}

We define $\mcT^{-}_{\neut}$ to be the set of all $T^{\mfn\mfl} \in \mcT$ which satisfy $|T^{\mfn\mfl}|_{\SG} < 0$.
By Lemma~\ref{non-neutral trees are nice}  every tree in $\mcT^{-}_{\neut}$ is neutral, as indicated  by the notation,
and  $|T^{\mfn\mfl}|_{\s} < 0$
so that  the properties stated in Lemma~\ref{lem:neg-trees} hold for all $T^{\mfn\mfl}\in\mcT^{-}_{\neut}$.

\subsection{Models and renormalisation}\label{sec: models and renormalisation}
%
Let $K$ be the truncated heat kernel as in Section~\ref{sec: main result}.
We denote by $\mathscr{M}_{\infty}$ the non-linear metric space of smooth $K$-admissible models on $\mathscr{T}$ and by $\mathscr{M}_{0}$ the completion of $\mathscr{M}_{\infty}$. 
Given any \textit{smooth} functions $\noise = (\noise_{+},\noise_{-})$, we write $Z_{\can}^{\noise}$ for the canonical lift of $\noise$.
As is often the case, one does not expect the sequence of random models $Z^{\noise^{(\e)}}_{\can}$ to converge to a limit as $\e \downarrow 0$. 
We must instead work with lifts different from the canonical one. The works
\cite{BHZalg,CH} give definitions for the \emph{BPHZ lift} which can be seen as a map from a space of random 
stationary driving noises into a space of random models. 
In particular, the BPHZ lift is not just a measurable function of an underlying driving noise but also takes as input knowledge of the underlying distribution of the driving noise, albeit only through the expectations of finitely 
many multilinear functionals of the driving noise.

While it has been successfully used in many other problems treated by regularity structures, the BPHZ lift does not 
seem to be the natural renormalisation procedure for  the sine-Gordon equation. 
The drawback of the BPHZ lift in our context is that it tries to cancel the expectation of every tree in $\mcT^{-}$ 
despite the following.
\begin{itemize}
\item Only the trees of $\mcT_{\neut}^{-}$ have divergent expectations. \footnote{assuming all the subtrees of a tree in $\mcT_{\neut}^{-}$ have been suitably renormalised.}
\item Due to their net charge, it is unclear if the renormalisation of a tree in $\mcT^{-} \setminus \mcT^{-}_{\neut}$ 
produces much cancellation. \footnote{The need of certain cancellation is explained in the paragraphs around \eqref{e:explanation-IS}, and that is why we need to rewrite the moment formula from \cite{CH} to get a new formula as stated in Proposition~\ref{prop:mom-formula-z}.} 
\end{itemize}
We instead employ what we call the \emph{neutral} BPHZ lift in this paper \dash this modification of the BPHZ lift only renormalises trees in $\mcT_{\neut}^{-}$. 

For readers who are not familiar with the BPHZ lift and \cite{BHZalg,CH} the previous sentences may not mean much.
For that reason  we give a short sketch in the next section of how one can define renormalised models and define the specific neutral BPHZ lift we are interested in. 
Those readers familiar with \cite{BHZalg} can skip the following section and immediately go to Section~\ref{sec: neutral BPHZ}. 
\subsubsection{Renormalisation in regularity structures}
\label{sec: general renormalisation}
The complexity of the algebraic and analytic constraints encoded by the space of models make it difficult to directly and explicitly define renormalised lifts $\noise \mapsto Z^{\noise} \in \mathscr{M}_{\infty}$. 
 
The works \cite{Regularity,BHZalg} reframe the problem of exhibiting a rich class of renormalised 
lifts as one of finding a sufficiently rich group, called the renormalisation group $\mathfrak{R}$, which 
admits a continuous group action on $\mathscr{M}_{\infty}$. 
One then has a variety of smooth lifts $\noise \mapsto M Z^{\noise}_{\can}$ indexed by $M \in \mathfrak{R}$. 

Describing this formalism with precision would take us to far afield, we instead point the uninitiated but curious reader to \cite{martintakagilect}. We will only give a brief conceptual sketch\footnote{In particular we will completely ignore the role of the extended label and the more delicate aspects of the interplay of this renormalisation scheme with the recentering procedure for subtrees of positive homogeneity \dash this is a key part of the story of \cite{BHZalg}. One should also have in mind that our setting is a sector of the reduced regularity structure.} of the renormalisation group and its group action on models.

Instead of viewing $\mathscr{M}_{\infty}$ as consisting of pairs of maps $Z = (\Pi,\Gamma)$ we instead view each such a pair as originating from a single map $\PPi:\mcb{T} \rightarrow \mcb{C}(\R \times \T^{2})$ which is itself ``$K$-admissible'' in the sense that if both $\tau, \CI[\tau] \in \mcT$ one has $\PPi\CI[\tau] = 
K \ast (\PPi  \tau)$.

One can think of $\PPi$ as the base-point independent parent of the family $\{\Pi_{z}\}_{z \in \R \times \T^{2}}$, the 
family being obtained from the former via a ``recentering'' procedure around any given point $z$. 
In \cite{BHZalg} this recentering procedure along with the construction of the corresponding transport maps $\Gamma$ is encoded via a map $\PPi \mapsto \mathcal{L}(\PPi) = (\Pi,\Gamma)$ which is defined using Hopf-algebraic methods.  

The constraint of $K$-admissibility does not by itself imply that $\mathcal{L}(\PPi) \in \mathscr{M}_{\infty}$. However an important example is the following: if one defines $\PPi_{\can}^{\noise}$ to be the unique tree product multplicative, $K$-admissible map with $\PPi_{\can}^{\noise} \Xi_{\pm} = \noise_{\pm}$, then $\mathcal{L}(\PPi_{\can}^{\noise}) = Z^{\noise}_{\can}$. 

The renormalisation group $\mathfrak{R}$ of \cite{BHZalg} can then be realized as a particular group of linear operators $M: \mcb{T} \rightarrow \mcb{T}$ with the property that if $\PPi$ satisfies $\mathcal{L}(\PPi) \in \mathscr{M}_{\infty}$ then the same is true of $\PPi \circ M$. 

In particular, \cite{BHZalg} parameterizes $\mathfrak{R}$ via the collection of maps $\ell: \mcT_{-} \rightarrow \R$ \dash an element of the renormalisation group is determined by an assignment of a counterterm value 
to each tree of negative homogeneity. 
The correspondence $\ell \rightarrow M_{\ell} \in \mfR$ is given by setting
\begin{equ}
M_{\ell}
\eqdef
(\ell \otimes \Id)\Deltam\;,
\end{equ}
where $\Deltam$ is a \emph{comodule co-product}: acting on a given $T^{\mfn\mfl}$ it produces all linear combinations of simple tensors where the left factor consists of a product of  several disjoint subtrees of $T^{\mfn\mfl}$, each of which belongs to $\mcT_{-}$, and the right factor is a quotient tree obtained by contracting each of these subtrees to a point. 
A general recipe for constructing a renormalised lift is then given by $\mathcal{L}(\PPi_{\can}^{\noise} M_{\ell})$ where we have 
\begin{equ}\label{eq: renormalised model}
\PPi^{\noise}_{\can} M_{\ell} = (\ell \otimes \PPi_{\can}^{\noise}) \Deltam\;. 
\end{equ}

Given a \emph{random} smooth stationary $\noise$ with moments of all orders, the BPHZ lift 
$Z_{\BPHZ}^{\noise}$ is written as $Z_{\BPHZ} = \mathcal{L}(\PPi_{\BPHZ}^{\noise})$, where $\PPi_{\BPHZ}^{\noise}$ can be obtained as a probabilistic recentering of $\PPi_{\can}^{\noise}$. 
Writing  $\PPi_{\BPHZ}^{\noise}$ in the form of  \eqref{eq: renormalised model}, one chooses $\ell$ so that the expectation of the RHS evaluated on any element of $\mcT^{-}$ vanishes.  
This choice can be solved for inductively, working from smaller elements of $\mcT_{-}$ to larger ones. In \cite{BHZalg}, this recursive procedure is encoded using an algebra homomorphism\footnote{Again, we are simplifying the picture: the target space of $\tilde{\mcb{A}}_{-}$ is an algebra generated by a larger class of trees.} 
$\tilde{\mcb{A}}_{-}: \Alg(\mcT^{-}) \rightarrow \Alg(\mcT)$ called the \emph{negative twisted antipode} and $\bar{\PPi}_{\can}^{\noise}: \Alg(\mcT) \rightarrow \R$: 
\begin{equ}\martin{Show that this model differs from the BPHZ model by a finite character.} 
\PPi_{\BPHZ}^{\noise} = (\bar{\PPi}_{\can}^{\noise} \circ \tilde {\mcb{A}}_- \otimes \PPi_{\can}^{\noise})\Deltam \tau\;.
\end{equ}
The maps $\bar{\PPi}_{\can}^{\noise}$ and $\tilde {\mcb{A}}_-$ are again defined for trees and then extended multiplicatively. 
For any $\bar{\tau} \in \mcT$ one sets $\bar{\PPi}_{\can}^{\noise}\bar{\tau} \eqdef \E[(\PPi^{\noise}_{\can}\bar{\tau})(0)]$. 
The algebra morphism $\tilde {\mcb{A}}_-$ is defined in such a way that the BPHZ lift is guaranteed to satisfy
$\E[(\PPi^{\noise}_{\BPHZ}\tau)(0)] = 0$ for all $\tau \in \CT_-$ and all stationary $\noise$ with sufficiently many moments. 

If $\noise$ is a random driving noise which is very rough one does not expect to be able to define $\PPi^{\noise}_{\can}$ but in many cases of interest we can show that for families of $(\noise^{(n)})_{n \in \N}$ of approximations with $\lim_{n \rightarrow \infty} \noise^{(n)} = \noise$ appropriately one can define $\PPi^{\noise}_{\BPHZ} \eqdef \lim_{n \rightarrow \infty} \PPi^{\noise^{(n)}}_{\BPHZ}$ where the limit is taken in a space of random models. 
Moreover, this construction is robust in that it is insensitive to the particular choice of approximating sequence $(\noise^{(n)})_{n \in \N}$. 
\subsubsection{The neutral BPHZ lift and main estimates}
\label{sec: neutral BPHZ}
The noise we want to lift into a renormalised model is the noise $\xi = (\xi_{-},\xi_{+})$ defined in Theorem~\ref{e:def of xi}, but we need to define such a model via probabilistic limit since $\xi$ has rough realizations. 
The strategy mentioned would lead us to trying to define 
$\Pi^{\xi}_{\BPHZ} \eqdef \lim_{\e \to 0} \Pi^{\xi^{\e}}_{\BPHZ}$, but the convergence of the RHS falls out of the scope of \cite{CH} for the reasons given in the introduction and does not seem so straightforward to prove directly. 

We instead define so called {\it neutral BPHZ lift} which is specific to the sine-Gordon regularity structure. Given a smooth random $\noise$ the corresponding neutral BPHZ lift is given by 
\[
\PPi_{\overline \BPHZ}^{\noise} \eqdef (\bar{\PPi}_{\can}^{\noise} \circ \tilde {\mcb{A}}_- \circ \mcb{N} \otimes \PPi_{\can}^{\noise})\Deltam \tau
\]
Here $\tilde{\mcb{A}}_{-}$ is the negative twisted antipode of \cite{BHZalg}; $\PPi_{\can}^{\noise}$ and $\bar{\PPi}^{\noise}_{\can}$ are also as referenced above. 
The new map $\mcb{N}$ is an algebra homomorphism from the algebra of forests to itself \dash it is the projection onto those forests where every constituent tree is required to be neutral.

Switching henceforth to the notation $Z_{\overline{\BPHZ}}^{\noise}$ instead of $\PPi_{\overline{\BPHZ}}^{\noise}$, we then seek to define the neutral BPHZ lift of $\xi$ 
via $Z_{\overline{\BPHZ}}^{\xi} \eqdef \lim_{\e\to 0} Z_{\overline{\BPHZ}}^{\xi^{\e}}$. Our task in this paper is to show the limit on the RHS exists. 

%
\begin{remark}
We remark that the term ``forest'' refers to collection of trees as it does in \cite{BHZalg} but this section is the last time we use it in this way. The paper \cite{CH}, and all later sections of this paper will use the term forest for a similar but slightly ``stronger'' notion, see Definition~\ref{def:Div-forest}.
\end{remark}
\begin{remark}
Recalling Section~\ref{sec: general renormalisation}, the neutral BPHZ lift of $\xi^{\e}$ can be obtained from \eqref{eq: renormalised model} with $\noise = \xi^{\e}$ where the correct choice of $\ell$ is enforced by the following two constraints: (i) $\ell$ should vanish on all non-neutral trees of $\mcT^{-}$, (ii) the expectation of \eqref{eq: renormalised model} evaluated on any neutral tree in $\mcT^{-}$ should vanish. 
\end{remark}
\begin{remark}
It follows a posteriori from
Theorem~\ref{theorem:mom-bounds} below that the BPHZ lift
 $ Z_{\BPHZ}^{\xi^{\e}}$ also converges and in particular that the moments of
$(\Pi_{\BPHZ}^{\xi^{\e}} - \Pi_{\overline{\BPHZ}}^{\xi^{\e}})_{z}[\tau] $
evaluated at $z$ converge to finite limits. 
However, Proposition~\ref{prop:no-renormalisation-equation} below
does not hold in general for the BPHZ lift,
see Remark~\ref{rem:BPHZ-not-all-cancel}.
\end{remark}

The main result we seek to prove is given by 
Theorem~\ref{theorem:mom-bounds} below, from which we
immediately obtain that our neutral BPHZ lift can be extended to $\xi$ in the limit.
The following estimates are the key ingredient for proving this result. 
Below we use the notation $Z_{\overline{\BPHZ}}^{\e} = (\hat{\Pi}^{\e},\hat{\Gamma}^{\e})$. 
\begin{theorem} \label{theorem:mom-bounds}
Let $\tau \in \mcT^{-}$. 
For any $p \in \N$ there exists $C_{\tau,p}$ 
such that 
\begin{equs}[e:moment-bound-theorem]
\E \left( 
\hat{\Pi}^{\e}_{z} \tau (\psi^{\lambda}_{z}) \right)^{2p}
&\le
C_{\tau,p} 
\lambda^{2p (|\tau|_{\s}+\eta)}
\\
\E \left( \big(\hat{\Pi}^{\e}_{z}\tau -  \hat{\Pi}^{\bar{\e}}_{z} \tau \big)    (\psi^{\lambda}_{z})
\right)^{2p}
&\le
\e^\kappa
C_{\tau,p}
\lambda^{2p(|\tau|_{\s}+\eta)}
\end{equs}
for some sufficiently small $\kappa,\eta>0$,
uniformly in all $\bar{\e} < \eps \in (0,1]$, all $\lambda \in (0,1]$,
all continuous test functions $\psi$  supported on the unit ball in $\R \times \T^{2}$ with $L^\infty$ norm bounded by $1$.
%
\end{theorem}

The proof of Theorem~\ref{theorem:mom-bounds}
will be given in Section~\ref{sec:proof-mom-bounds}.

%
%
%
\subsection{Renormalised equation and proof of the main result}
\label{sec:proof-of-main}
To find the maximal solution $u_\e$ to \eqref{eq: renormalised approx equation} we have written $u_\e = v_\e +  \Phi_\e$,
where $\Phi_\e$ is defined in \eqref{e:defPhi} which has a distributional limit $\Phi$.
We would like to solve the equation \eqref{e:SG} for $v_\e$
from a initial data $v_{0}$ that is at least as rough as $\Phi(0)$,
 \footnote{This is in contrast with \cite{BCCH} where one is only able to start the equation from the stationary linear solution perturbated by some more regular data} 
which apparently poses a problem since our non-linearity is not just a polynomial. 
However we can take advantage of the  exponential functions to do some preprocessing. We write $v^{\e} = Gv^{\e}_{0} + w^{\e}$ and instead look for maximal solutions to the cauchy problem
\begin{equ}[e:equweps]
\partial_t w^{\e}= \frac12\Delta w^{\e} 
-\frac{i}{2} 
  \Big(e^{i\beta Gv^{\e}_{0}} e^{i\beta w^{\e}}  \xi^{\e}_{+}
    - e^{-i\beta Gv^{\e}_{0}} e^{-i\beta w^{\e}}  \xi^{\e}_{-}\Big)  + R^{\e}\;,\quad
    w^{\e}(0,\cdot) = 0 \;.
\end{equ}
%
%
Now to prove our main result of the paper (Theorem~\ref{theo:v-equ}), it suffices to establish the limiting maximal solution to \eqref{e:equweps}. Here we provide the proof to Theorem~\ref{theo:v-equ}, with all other technical details required by the proof given in the rest of the paper.

\begin{proof}[of Theorem~\ref{theo:v-equ}]
With the regularity structure defined  in Section~\ref{sec: sg reg struct},
we can formulate \eqref{e:equweps} as an abstract fixed point problem in the space $\CD^{\mu,0}_0$ 
(the modelled distribution space defined in \cite{Regularity};
recall that the subscript here stands for the lowest homogeneity, and $\mu \in (\bar{\beta},2)$ was introduced in \eqref{e:def-mcT}).
If we denote by $\CP$ the integration operator corresponding to convolution with 
the heat kernel (see \cite[Sec.~5]{Regularity}), \eqref{e:equweps} can be
described by the following fixed point problem:
\begin{equ}[e:FPW]
W = \CP \one_{t > 0} \Big(R_\eps  
-\frac{i}{2} 
  \Big(e^{i\beta Gv^{\e}_{0}} e^{i\beta W}  \Xi_{+}
    - e^{-i\beta Gv^{\e}_{0}} e^{-i\beta W}  \Xi_{-}\Big)\Big)\;.
\end{equ}
Indeed, as in the proof of \cite[Theorem~2.5]{HaoSG},
one has that as long as $\mu \in (0, 2]$, then
$e^{\pm i\beta Gv_0^{\e}}$ can be interpreted as an element in $\CD^{\mu,2\eta}_0(\bar T)$ where $\bar T$ is the space of abstract Taylor polynomials.
\footnote{The reason that it belongs to  $\CD^{\mu,2\eta}(\bar T)$ rather than $\CD^{\mu,\eta}(\bar T)$ is that a term $|\d_x G(v^\e_{0})|^2$ arising from differentiating $e^{\pm i\beta Gv_0^{\e}}$ twice
makes it a bit worse (see proof of \cite[Theorem~2.5]{HaoSG}).} 
For $W\in \CD^{\mu,0}_0$, one has $ e^{\pm i\beta W}\in \CD^{\mu,0}_0 $.
Also, $\xi_+$ and $\xi_-$ can be lifted as abstract noises $\Xi_{+}, \Xi_{-}$ with $|\Xi_{\pm}|_{\s}=-\bar{\beta}$. One then has
\[
e^{\pm i\beta Gv^\e_{0}} e^{\pm i\beta W} \Xi_{\pm}  
\in  \CD^{\mu-\bar{\beta},2\eta-\bar{\beta}}_{-\bar{\beta}} \;.
\]
Since $\mu-\bar{\beta}+2 >\mu$ and $2\eta-\bar{\beta}+2 > 0$ (by assumption $\eta \in (\frac{\bar\beta}{2}-1,0)$), by
\cite[Theorem~7.8]{Regularity} the fixed point problem
\eqref{e:FPW} admits a unique maximal solution  in $\CD^{\mu,0}_0$,
and the solution map from the space of models to the solutions is continuous.
The  fact that the reconstructed solution of the solution $W$ to \eqref{e:FPW} with model
$Z^{\e}_{\overline{\BPHZ}}$
follows from the discussion below, in particular Proposition~\ref{prop:no-renormalisation-equation}.
Therefore the main theorem follows from the  convergence of models, namely Theorem~\ref{theorem:mom-bounds}, which will be proved in the rest of the paper.
\end{proof}

%
Let $W$ be the modelled distribution which solves the fixed point problem \eqref{e:FPW} where the underlying model is taken to be $Z^{\e}_{\overline{\BPHZ}}$. Denoting the reconstruction of $W$ by $\bar{w}^{\e}$ it follows from \cite[Theorem~2.5]{BCCH} that $\bar{w}^{\e}$ is the maximal in time solution to the Cauchy problem
\begin{equs}
	\label{eq: generic form of renormalised equation}
\partial_t \bar{w}^{\e}
&= 
\frac{1}{2}
\Delta \bar{w}^{\e} -
\frac{i}{2} 
\Big(
e^{i\beta G v_{0}^{\e}}
e^{i\beta \bar{w}^{\e}} 
\xi_{+}^{\e} 
- 
e^{-i\beta G v_{0}^{\e}} 
e^{i\beta \bar{w}^{\e}}
\xi_{-}^{\e}
\Big)
\\
&\qquad\qquad\qquad\qquad
+
G^{\e}(\bar{w}^{\e},\nabla \bar{w}^{\e})
+ R^{\e}  \;,
\\
\bar{w}^{\e}(0,\cdot) 
& = 0\;.
\end{equs}
It appears that \eqref{eq: generic form of renormalised equation}
differs from \eqref{e:equweps}
 by a function called $G^{\e}$.
\cite{BCCH} shows that $G^{\e}:\R \times \R^{2} \rightarrow \R$ is of the form
\begin{equ}\label{eq: renormalisation term}
G^{\e}(w,\nabla w)
\eqdef
\sum_{\tau \in \mcT_{\neut}^{-}}
\frac{\ell^{\e}_{\overline \BPHZ}[\tau]}{S[\tau]}
 \Upsilon[\tau](w,\nabla w)\;,
\end{equ} 
where $\ell^{\e}_{\overline{\BPHZ}} \eqdef \bar{\PPi}^{\xi^{\e}} \circ \tilde{\mcb{A}}_{-} \circ \mcb{N}$,
and $\Upsilon[\tau]:\R \times \R^{2} \rightarrow \R$ are recursively defined in \cite[Eq~(4.3)]{BCCH}. 
Finally, $S[\tau]$ is the overall symmetry factor\footnote{It is simply the number of distinct decorated planar trees corresponding to the decorated combinatorial tree $\tau$} defined analogously to \cite[Eq~(2.16)]{BCCH}.

Finding a general formula for $\Upsilon[\tau]$ for arbitrary trees $\tau$ is straightforward exercise in induction \dash below we record the simplified formula one obtains when restricting to the trees  $T^{0} \in \mcT_{\neut}^{-}$. 
\begin{lemma}
For every $\tau \eqdef T^{0} \in \mcT^{-}_{\neut}$ one has that $\Upsilon[\tau]$ is just a constant, in particular 
\begin{equs}\label{eq: function for each tree}
\Upsilon[\tau]
\eqdef
(i\beta)^{-1}
\Big(
\prod_{u \in N(T)}
\frac{\beta \mcb{q}(u)^{|d(u,\tau)|}}{2}
\Big)\;,
\end{equs}
where for $u \in N(T)$ we define $d(u,\tau) \eqdef |\{e \in K(T):\ e_{\p} = u \}|$.
\end{lemma}
%
For $\tau \in \mcT$ we write $\tau_{\opp}$ for the symbol in $\mcT$ obtained by composing the map $\mfl: N(T) \rightarrow \{-,0,+\}$ with the involution that swaps $+$ and $-$ and leaves $0$ unchanged. 
For instance,
\[
\tau \;=\;
\begin{tikzpicture}[scale=0.13,baseline=-2]
\draw (-2,0) node[xi] {\tiny $ -$}
	-- (0,-3) node[xi] {\tiny $ -$} 
	-- (2,0) node[xi] {\tiny $+$}
	 -- (0,3) node[xi] {\tiny $+$};
\end{tikzpicture}
\qquad
\mapsto
\qquad
\tau_{\opp} \;=\;
\begin{tikzpicture}[scale=0.13,baseline=-2]
\draw (-2,0) node[xi] {\tiny $ +$}
	-- (0,-3) node[xi] {\tiny $ +$} 
	-- (2,0) node[xi] {\tiny $-$}
	 -- (0,3) node[xi] {\tiny $-$};
\end{tikzpicture}
\]
The following three lemmas describe the cancellation of renormalisation constants that occurs for our equation.
\begin{lemma}\label{lem:sameConstant}
For any $\tau \in \mcT^{-}_{\neut}$ one has 
$\ell^{\e}_{\overline \BPHZ}[\tau] = \ell^{\e}_{\overline \BPHZ}[\tau_{\opp}]$.
\end{lemma}
\begin{proof}
This is an immediate consequence of the fact that the two pairs of space-time random fields $(\xi_{+}^{\e} ,\xi_{-}^{\e})$ and $(\xi_{-}^{\e},\xi_{+}^{\e})$ are equal in distribution. 
\end{proof}

\begin{lemma}\label{lem: parity cancellation for ell}
If $\tau = T^{\mfn\mfl} \in \mcT^{-}_{\neut}$ with $\mfn \neq 0$ then $\ell^{\e}_{\overline \BPHZ}[\tau] = 0$. 
\end{lemma}
\begin{proof}
This is the content of Lemma~\ref{lem: actual parity cancellation for ell} below.
\end{proof}

\begin{lemma}\label{lem: sign from charge flip}
If $\tau = T^{0\mfl} \in \mcT^{-}_{\neut}$ then 
$\Upsilon[\tau]
=- \Upsilon[\tau_{\opp}]$.
\end{lemma}
\begin{proof}
 For every $u \in N(T)$ one has $d(u,\tau) = d(u,\tau_{\opp})$ and if $u \in N(T)$ has a noise of type $\pm$ in $\tau$ then it has a noise of type $\mp$ in $\tau_{\opp}$. 
Then from \eqref{eq: function for each tree} we see that
\begin{equs}
\Upsilon[\tau_{\opp}]
=
(-1)^{\sum_{u \in N(T)}|d(u,\tau)|}
\Upsilon[\tau]\;.
\end{equs}
The desired result follows by observing that $\sum_{u \in N(T)}|d(u,\tau)| = |N(T)| -1$ must be odd since $N(T) = L(T)$ and $\tau$ is neutral. 
\end{proof}

\begin{proposition} \label{prop:no-renormalisation-equation}
One has $G^{\e} = 0$.
\end{proposition}
\begin{proof}
First observe that for any $\tau \in \mcT$, one has $S[\tau] = S[\tau_{\opp}]$.
Combining this with Lemma~\ref{lem:sameConstant} we can rewrite \eqref{eq: renormalisation term} as
\begin{equ}
G^{\e}
\eqdef
\sum_{\tau \in \mcT_{\neut}^{-}} \frac{\ell^{\e}_{\overline \BPHZ}[\tau]}{2 S[\tau]} \bigl(
 \Upsilon[\tau]
 +
 \Upsilon[\tau_\opp]
 \bigr)
 \;,
\end{equ}
and the desired result then follows from Lemmas~\ref{lem: parity cancellation for ell} and~\ref{lem: sign from charge flip}.
\end{proof}

\begin{remark} \label{rem:BPHZ-not-all-cancel}
Proposition~\ref{prop:no-renormalisation-equation} 
does not hold anymore for the BPHZ lift. 
This is because in \eqref{eq: renormalisation term},
the sum would be over all $\tau\in \mcT^-$, and although
 $\ell^{\e}_{ \BPHZ}[\tau] = \ell^{\e}_{\BPHZ}[\tau_{\opp}]$
 still holds,
$\Upsilon[\tau]$ would be generally not equal to
$- \Upsilon[\tau_{\opp}]$.
For instance, for a non-neutral $\tau = T^{\mfn\mfl} \in \mcT^{-}$ with $\mfn = 0$,
$\Upsilon[\tau]$ is no longer a constant and is equal to the complex conjugate of $- \Upsilon[\tau_{\opp}]$.
\end{remark}
\subsection{Fixing our graph, forests, cut sets, and intervals}

To prove the moment bounds
as stated in Theorem~\ref{theorem:mom-bounds}, we first work to give
a formula for these moments in Proposition~\ref{explicit formula} below.
To this end we should first introduce some notation.

\subsubsection{Fixing a symbol and a moment}
\label{sec: our tree and moment} 
\ajay{What follows from here in the paper will be influenced by the \cite{CH}-rewrite. For instance, we will likely no longer have these $H$ operators with internal spatial integrations. To avoid issues with edits being redone, for now edits should be restricted to what occurs before this comment}

Fix a tree $\sT^{\sn\sl} \in \mcT^{-}$ and $p \in \N$. 
Since our goal (in view of Theorem~\ref{theorem:mom-bounds}) is to consider the $2p$-th moment of the random object represented by the tree $\sT^{\sn\sl}$, we introduce a setting where we work with $2p$ different ``copies'' of $\sT^{\sn\sl}$. 

We fix $p$ disjoint copies of $\sT^{\sn\sl}$, denoted by  
$\{(\sT_{j})^{\sn_{j}\sl_{j}}\}_{j=1}^{p}$, along with $p$ disjoint copies of $(\sT^{\sn\sl})_{\opp}$, denoted by $\{(\sT_{j})^{\sn_{j}\sl_{j}}\}_{j=p+1}^{2p}$. 
We then define $D_{2p}$ as the decorated graph formed by the disjoint union of $\{(\sT_{j})^{\sn_{j}\sl_{j}}\}_{j=1}^{2p}$,
namely $D_{2p}$ consists of node and edge sets
\[
N(D_{2p}) \eqdef \bigsqcup_{j=1}^{2p} N(\sT_{j})\;,
\qquad
K(D_{2p}) \eqdef \bigsqcup_{j=1}^{2p} K(\sT_{j})\;.
\]
We equip $N(D_{2p})$ with maps $\sn: N(D_{2p}) \rightarrow \N^{3}$ and $\sl: N(D_{2p}) \mapsto \{+,0,-\}$ obtained by concatenating the maps that came with the trees $\sT_{j}$. 

We also define the set of {\it noise nodes} of $D_{2p}$ via
\[
L(D_{2p}) \eqdef \bigsqcup_{j=1}^{2p} L(\sT_{j})\;.
\]
Note that we then have $\mcb{q}(L(D_{2p})) = 0$, because we have equal number of copies of $\sT^{\sn\sl}$ and $(\sT^{\sn\sl})_{\opp}$.

Since we assume $\sT^{\sn\sl} \in \mcT^{-}$, by Lemma~\ref{lem:neg-trees},  $N(\sT_{j}) = L(\sT_{j})$ for every $j$, that is every node is a noise node. 
However, to make it clear where we are treating the presence of the noise and to make explicit the relation of the analysis here with that of \cite{CH} we will often write $L(D_{2p})$ or $L(S)$ instead of $N(D_{2p})$ or $N(S)$ even though they are equal. 
Making the distinction will also be helpful when defining sets in \eqref{e:def-NFA-KFA-LFA}
and stating formulas such as \eqref{eq: explicit formula for single copy}.

\textbf{Convention.} From now on, whenever we talk of a subtree $S$, we are talking about a subtree of the graph $D_{2p}$ (which must necessarily be a subtree of some $\sT_{j}$ since trees and subtrees are all connected by definition).
We explicitly identify $N(S)$ and $K(S)$ as subsets $N(D_{2p})$ and $K(D_{2p})$, respectively. 
In particular, our subtrees are concrete subtrees rather than isomorphism classes of subtrees. 

We introduce the symbol $\logof$ to represent the base-point of the model
(namely it labels the space-time variable $z$ in $\hat{\Pi}^{\xi}_{z}[\tau]$) 
 and define 
\begin{equ}[e:def-allnodes]
\allnodes \eqdef N(D_{2p}) \sqcup \{\logof\}\;.
\end{equ}
\subsubsection{Functions and multi-indices}
We write $\mcb{C}$ for all scalar functions on $(\R_{+} \times \T^{2})^{\allnodes} \setminus \mathrm{Diag}_{\allnodes}$ which are smooth in their arguments \dash here $\mathrm{Diag}_{\allnodes}$ denotes the big diagonal, that is all tuples $(x_{v}: v \in \allnodes)$ for which one can find $v \neq v'$ with $x_{v} = x_{v'}$. We will often suppress the subtraction of big diagonals from quantifiers.   
We often use multi-indices in the set $(\N^{3})^{\allnodes}$ to indicate the orders of differentials in each variable $z_v$. 
Namely, given $k = (k_{v})_{v \in \allnodes} \in (\N^{3})^{\allnodes}$ and writing $k_{v} = (k_{v}^{(0)},k_{v}^{(1)},k_{v}^{(2)})$ for each $v \in \allnodes$, we define a differential operator $D^{k}$ on $\mcb{C}$  by setting, for each $F \in \mcb{C}$ 
and $z = (z_{v})_{v \in \allnodes} \in (\R_{+} \times \T^{2})^{\allnodes}$
\[
D^{k}F(z)
=
\Big(
\prod_{
\substack{
v \in \allnodes\\
0 \le j \le 2}}
\partial_{z_{v}^{(j)}}^{k_v^{(j)}}
\Big)
F(z)\;.
\]
In other words $k_v$ describes the order of derivatives in the time or space variables at node $v$.
We also define, for any $u \in \allnodes$ and $0 \le l \le 2$ the multi-index $\delta_{u,j} \in (\N^{3})^{\allnodes}$ by setting, for any $v \in \allnodes$ and $0 \le j \le 2$, 
\begin{equ} [e:def-delta-ul]
(\delta_{u,j})_{v}^{(j)} \eqdef \mathbbm{1}\{ u = v \textnormal{ and } j = l \}\;.
\end{equ}
We also use the shorthand $D_{u,j} \eqdef D^{\delta_{u,j}}$. 
\subsubsection{Forests and cut sets}
\begin{definition}[The sets $\Div_{j}$ and $\Div$ and forests]
\label{def:Div-forest}
For $j \in [2p]$ we write $\Div_{j}$ for the collection of all subtrees $S$ of $\sT_{j}$ with the property that $S$ is neutral and $|S^{0} |_{\s} < 0$.
 Equivalently, by Lemma~\ref{non-neutral trees are nice}, $\Div_{j}$ contains all subtrees $S$ of $\sT_{j}$ such that $|S^{0} |_{\SG} < 0$.

We then define $\Div \eqdef \bigsqcup_{j=1}^{2p} \Div_{j}$. 
We say that $\mcF \subset \Div$ is a \emph{forest} if for any two distinct trees $S,T \in \mcF$, 
{\it exactly one} of the following three conditions holds:
\[
N(S) \subset N(T)\quad \text{or}
\quad
N(T) \subset N(S)\quad \text{or}
\quad
N(T) \cap N(S) = \emptyset\;.
\]
In other words, any two trees in a forest $\mcF$ must be either nested or disjoint.

We write $\mathbb{F}_{j}$ for all subsets of $\Div_{j}$ which are forests and write $\mathbb{F}$ for all subsets of $\Div$ which are forests. 
\end{definition}

\begin{remark}
Note that our notation here is different from the notation in \cite{CH}:
here $\Div$ contains subsets of $D_{2p}$,
whereas in \cite{CH}, $\Div$ contains subtrees of the tree $\bar T$.
\end{remark}

In the sequel we will often write things such as $K(\mcF)$, $L(\mcF)$ and $N(\mcF)$
for a {\it forest} $\CF$; these notations stand for all
the edges or all the nodes in all the trees of $\mcF$.
More precisely,
given any  map $Z$ which maps a subtree of $D_{2p}$ to a set,
 we see $Z$ as inducing a map on $\mathbb{F}$ by setting
\begin{equ} [e:Z-on-forest]
Z(\mcF)
\eqdef
\bigcup_{T \in \mcF}
Z(T)\;.
\end{equ}
In particular $Z$ can stand for the maps $K(\cdot)$, $L(\cdot)$ or $N(\cdot)$.
 
\begin{definition} [$\mfC$ and cut sets]
\label{def:cut-sets}
We define a map $\gamma: K(D_{2p}) \rightarrow \Z$ by setting, for each $e \in K(D_{2p})$, 
\[
\gamma(e)
\eqdef
\Big\lceil 
2 | \{ e' \in K(D_{2p}): e' \ge e \}| + \sum_{
\substack{
u \in N(D_{2p})\\
u \ge e_{\ch}
}
}
(
|\sn(u)|_{\s} 
- \bar{\beta})
\Big\rceil\;.
\]
One can check that this definition of $\gamma$, when restricted to $K(\sT_{j})$ for any $j \in [2p]$,  agrees with the one defined in \cite[Section~4.2]{CH}. In particular, $\gamma(e)$, when strictly positive, tells us that the edge $e$ should be positively renormalised with a Taylor expansion of order $\gamma(e) - 1$. 
As an example, assuming $\bar\beta = {503 \over 300}$ and $p=1$,
 for the tree $T^{\mfn\mfl}$ as shown in \eqref{e:example-of-tree} with $|T^{\mfn\mfl}|_{\s} = 10-6\bar\beta<0$, we draw the picture for $D_{2p}=D_2$ with an edge $e\in K(D_{2p})$ as follows:
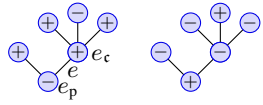
\[
\begin{tikzpicture}[scale=0.13,baseline=0]
\draw (-3,0) node[xi] {\tiny $ +$}
  -- (0,-3) node[xi] {\tiny $ -$} 
  -- (3,0) node[xi] {\tiny $+$}
   -- (0,3) node[xi] {\tiny $+$};
\draw (3,4) node[xi] {\tiny $-$} -- (3,0)   node[xi] {\tiny $+$}
  -- (6,3) node[xi] {\tiny $+$};
\node at (2.5,-2) {\footnotesize $e$};
\node at (2,-4) {\footnotesize $e_\p$};
\node at (5.4, -0.4) {\footnotesize $e_\c$};
\end{tikzpicture}
\quad
\begin{tikzpicture}[scale=0.13,baseline=0]
\draw (-3,0) node[xi] {\tiny $ -$}
  -- (0,-3) node[xi] {\tiny $ +$} 
  -- (3,0) node[xi] {\tiny $-$}
   -- (0,3) node[xi] {\tiny $-$};
\draw (3,4) node[xi] {\tiny $+$} -- (3,0)   node[xi] {\tiny $-$}
  -- (6,3) node[xi] {\tiny $-$};
\end{tikzpicture}
\]
In this case, one has $\gamma(e) = \lceil  2\cdot 4 -4\bar\beta \rceil =2$.
We then define
\[
\mfC \eqdef \{e \in K(D_{2p}): \gamma(e) > 0\} \;,
\qquad
\mfC_{j} \eqdef \mfC \cap K(\sT_{j}) \;.
\]
We call a subset $\cC \subset \mfC$ a \emph{cut set}, so $2^{\mfC}$ is the collection of all cut sets. 
In the sine-Gordon model, one can actually check that $\gamma(e)\in\{1,2\}$ for every $e \in K(D_{2p})$, thus $\mfC = K(D_{2p})$ and every kernel edge is a potential site of positive renormalisation. 
However we still often use the notation $\mfC$ in formulae to provide context and make the link to \cite{CH} clear.

For $\cC \subset \mfC$ we define 
\begin{equ} [e:def-FavoidC]
\mathbb{F}_{\cC} \eqdef \{ \mcF \in \mathbb{F}: \, K(\mcF) \cap \cC = \emptyset\} \;. 
\end{equ}
\end{definition}

We view $\mfC$ as a partially ordered set (poset) with that structure being inherited from $K(D_{2p})$. We also view $\Div$ as a poset with $T \le S \Leftrightarrow N(T) \subset N(S)$.
We view both $2^{\mfC}$ and $\mathbb{F}$ as posets equipped with the inclusion partial order. 
\section{An explicit formula for the neutral BPHZ model and its moments}\label{sec: explicit formula} 
The goal of this section is to provide a formula for moments of  the neutral BPHZ model,
as given in Propositions~\ref{explicit formula} and~\ref{prop:mom-formula-z} below.

\subsection{Notation for moment formula}
To state Propositions~\ref{explicit formula}, we introduce some notation for various functions in the integrands and the sets of nodes or edges appearing in the formula.

\subsubsection{Notation for interactions, (truncated) heat kernels and polynomials}
In the following, for each $e \in L(D_{2p})^{(2)}$ it is convenient to write the two elements of $e$ as $e_{<}$ and $e_{>}$.\footnote{More formally, we fix an (arbitrary) total order on $L(D_{2p})$, so $e_{<}$ is the smaller element in $e$ and $e_{>}$ is the larger element in $e$.}
In proving Theorem~\ref{theorem:mom-bounds},
upon taking $2p$-th moments of $\hat{\Pi}^{\e}_{z}[\tau](\psi^{\lambda}_{z})$, we will get an expectation of product of the noises in the set $L(D_{2p})$, which yields a product of pairwise ``interactions'' as in \eqref{eq: coulomb interaction}.  
Also, regarding the second bound in Theorem~\ref{theorem:mom-bounds}, when expanding the $2p$-th power
of $\hat{\Pi}^{\e}_{z}[\tau] -  \hat{\Pi}^{\bar{\e}}_{z}[\tau]$, there will be, say, $j$ copies of the trees carrying the noises $\xi^\e_{\pm}$ and the other  $2p-j$ copies of the trees 
carrying the noises   $\xi^{\bar\e}_{\pm}$.


In view of these we write $[2p]\eqdef \{1,\cdots,2p\} $ and
we call a map $\iota: [2p] \rightarrow \{\e,\bar\e\}$ an 
{\it ``$\e$-assignment''}; in proving the first bound of Theorem~\ref{theorem:mom-bounds} one only needs
the ``trivial'' $\e$-assignment, i.e. $\iota$ always take the value $\e$.
We then introduce the following notation for products of noises and interactions.

Products of the noises $\xi^\e_{\pm}$, $\xi^{\bar\e}_{\pm}$ are written in the following shorthands: 
for any $A \subset L(D_{2p})$ and an $\e$-assignment $\iota$ we define $\xi^{A,\iota} \in \mcb{C}$ as 
\begin{equ}[e:def-xi-A]
\xi^{A,\iota}(z) \eqdef \prod_{i = 1}^{2p} \prod_{a \in A \cap L(\sT_{i})}
\xi^{\iota(i)}_{\mfl(a)}(z_{a})\;,
\quad
\mbox{for }
z = (z_{a})_{a \in A} \in (\R_{+} \times \T^{2})^{A}\;.
\end{equ}
Note that for the trivial $\e$-assignment $\iota\equiv \eps$
one simply has $\xi^{A,\iota}(z) =
\prod_{a \in A}
\xi^{\e}_{\mfl(a)}(z_{a})$.

We define $\mathring{\mathfrak{J}}^{\pm}$ to be the semi-normed vector space of all smooth functions 
$J: \R \times \T^{2} \setminus \{ 0\} \rightarrow \R$ which satisfy 
$J(t,x) = J(t,-x)$,  and
 with semi-norm $\| \cdot\|_{\mp 2 \beta',m}$ with $m = 2$. 
We then set
\[
\mathfrak{J} \eqdef \bigoplus_{e \in L(D_{2p})^{(2)}} \mathring{\mathfrak{J}}^{\sign(e)}\;.
\]
An element of $\mathfrak{J}$ will often be written as $\CJ = (J_{e}: e \in L(D_{2p})^{(2)})$.
For $\CJ \in \mathfrak{J}$ and any $P \subset L(D_{2p})^{(2)}$, we write
\footnote{We do not mean that $\|\CJ\|$ is a norm, but instead we think of it
as simply a shorthand for product of norms $\|J_{e}\|_{\sign(e) \cdot 2 \beta',m}$.}
\begin{equs}
\CJ^{P} \eqdef \prod_{e \in P}J_{e}(x_{e_{>}}  &  - x_{e_{<}}) \; \in \mcb{C} \;,
\\
\|\CJ\|_{P} \eqdef \prod_{e \in P} \|J_{e}\|_{\sign(e) \cdot 2 \beta',m} \qquad
&  \mbox{and}
\qquad
\|\CJ\| \eqdef \|\CJ\|_{L(D_{2p})^{(2)}}\;.
\end{equs}
For most of the paper we will consider  
$\J\in \mfJ$ of the following form.
Given an $\e$-assignment $\iota$ we define   $\CJ(\iota) = (J(\iota)_{e})_{e \in L(D_{2p})^{(2)}}  \in \mfJ$ 
by setting, for each $e = (e_{<},e_{>})$ with $e_{<} \in L(\sT_{i})$ and $e_{>} \in L(\sT_{j})$
\begin{equs}
J(\iota)_{e}(z)
&\eqdef
J^{\iota(i),\iota(j)}(z)^{\sign(e)}
\\
&=\E (\xi^{\iota(i)}_{\mfl(e_<)}(0)\xi^{\iota(j)}_{\mfl(e_>)}(z))
\qquad
\mbox{for }
z\in \R_{+} \times \T^{2} \;,
\end{equs}
where $J^{\e,\e}$,  $J^{\e,\bar\e}$ or  $J^{\bar\e,\bar\e}$ is defined in \eqref{def: single mollified interaction}.
With this choice of $\J$ and the above notation
it is then straightforward to check that (see also \cite[proof of Theorem~3.2]{HaoSG})
\begin{equ}[e:E-xi-J] 
\E \left(\xi^{A,\iota}(z) \right)
=
\J(\iota)^{A^{(2)}}(z)\;.
\end{equ}
This is the moment formula given heuristically in \eqref{eq: coulomb interaction}.



For any $E \subset K(D_{2p})$ we define the function $\ke{E}{} \in \allf$ depending on $x_{v}$ with $v \in e_{\p}(E) \cup e_{\ch}(E)$, by
\begin{equ} [e:def-Ker-E]
\ke{E}{}(x) \eqdef
\prod_{e \in E}
K(x_{e_{\p}} - x_{e_{\ch}}).
\end{equ}
For any $\cD \subset \cut$ we define $\RKer^{\cD} \in \allf$, depending on $x=(x_{v})$ with $v \in e_{\ch}(\cD) \cup e_{\p}(\cD) \cup \{\logof\}$, via
\begin{equation}\label{e:def-RKer-D}
\begin{split}
&\RKer^{\cD}(
x
)
\eqdef
(-1)^{|\cD|}
\prod_{e \in \cD}
\sum_{
\substack{
k \in \N^{3}\\
|k|_{\s} < \gamma(e)}
}
\frac{(x_{e_{\p}} - x_{\logof})^{k}}{k!}
\,D^{k}K(x_{\logof} - x_{e_{\ch}}).
\end{split}
\end{equation} 

Finally, for any $A \subset N(D_{2p})$, $\mfn:A \rightarrow \N^{3}$ we define the functions $\powroot{A}{\mfn}{v}$ via $\powroot{A}{\mfn}{\logof}
(x) 
\eqdef
\prod_{u \in A}
(x_{u} - x_{\logof})^{\mfn(u)}$.
\subsubsection{Notation for nodes and edges}
\label{sec:nodes-and-edges}

Before stating Propositions~\ref{explicit formula} and~\ref{prop:mom-formula-z},
we recall more notation from \cite{CH}.
For a subtree $S$ of $D_{2p}$ 
(as mentioned above this implies that  $S$ must necessarily be a subtree of some $\sT_{j}$), 
and $\mcF \in \mathbb{F}$ we define
\[
\tilde N(\mcF)\eqdef 
\bigcup_{T\in\mcF} \tilde N(T)\;,
\qquad
C_{\mcF}(S)
\eqdef
\overline{\{T \in \mcF: T \subsetneq S\}}.
\]
In the last expression we are taking the set of maximal elements of a poset (see Section~\ref{sec:poset}).
We also define   for a subtree $S$  and $\mcF \in \mathbb{F}$
the following sets
\begin{equs}
\nrmod[\mcF,S] \eqdef \tilde{N}(S) 
\setminus 
\big(
\bigsqcup_{T \in C_{\mcF}(S)}
\tilde{N}(T) 
\big) \;,
& \quad
N_{\mcF}(S)\eqdef \nrmod[\mcF,S] \cup \{\rho_S\}\;,
\\
L_{\mcF}(S) \eqdef L(S) 
\setminus 
\big(   &
\bigsqcup_{T \in C_{\mcF}(S)}
L(T) 
\big)\;.
\end{equs}
We also define $K^{\downarrow}(S)$ to be the edges incoming to $S$ from above, and $ \bar{K}^{\downarrow}(S)$ to be these incoming edges together with the kernel edges in $S$:
\begin{equs}
K^{\downarrow}(S) &\eqdef 
\{ e \in K(\bar{T}):\ 
e_{\p} \in N(S),\ e_{\ch} \not \in N(S)  \} \;,
\\
 \bar{K}^{\downarrow}(S) &\eqdef K(S) \sqcup K^{\downarrow}(S)\;.
 \end{equs}
Finally we define
\begin{equs} 
K_{\mcF}(S) \eqdef K(S) \setminus &
\big(
\bigsqcup_{T \in C_{\mcF}(S)}  K(T)\big)\;,
\qquad
\mathring{K}_{\mcF}(S) 
\eqdef 
K(S) \setminus 
\big(
\bigsqcup_{T \in C_{\mcF}(S)}
\bar{K}^{\downarrow}(T)
\big)\;,
\\
\textnormal{ and }&
K^{\partial}_{\mcF}(S) 
\eqdef 
K(S) \cap
\big(
\bigsqcup_{T \in C_{\mcF}(S)}
K^{\downarrow}(T)
\big) \;.
\end{equs}
%

To make some of these definitions clearer we look at an example of $\mcF=\{T_1,T_2,T'\}$, where $T_1\cap T_2 =\emptyset$ and $T'\subset T_2$. The entire picture illustrates a {\it subtree}  $S$ (of some larger tree not drawn).
Below we have shaded  $T_1,T_2$ in light gray and, on top of this, $T'$ in dark gray. We have $C_{\mcF}(S) = \{T_1,T_2\}$.
The nodes of $\tilde N_{\mcF}(S)$ are shaded in light blue, 
the edges of $\mathring{K}_{\mcF}(S)$ in light green,
 and the edges of $K^{\partial}_{\mcF}(S) $ in red. 

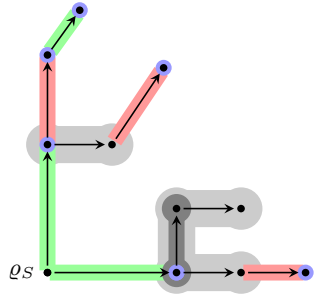
\[
\begin{tikzpicture}[scale=1.2,rotate=-45,subtreenode/.style={circle,fill=gray!40,inner sep=0pt,minimum size=16pt},  subtreeedge/.style={line width=14pt,gray!40, shorten >= -3pt,shorten <=-3pt},
   shadednodeblue/.style={circle,fill=blue!40,inner sep=0pt,minimum size=6pt},
   shadednodegray/.style={circle,fill=gray,inner sep=0pt,minimum size=11pt},
]
	\node (a1) at (0,0) {};
	\node (b1) at (-1,1) {};
	\node (b2) at (1,1) {};
	\node (c1) at (-0.5, 1.5) {};
	\node (c2) at (0.5, 1.5) {};	
	\node (c3) at (1.5, 1.5) {};
	\node (d1) at (-1.7, 1.7) {};
	\node (d2) at (-0.7,2.5) {};
	\node (d3) at (1,2) {};
	\node (d4) at (2,2) {};
	\node (e) at (-1.8,2.3) {};
	
    \node[subtreenode] at (b1) {};
    \node[subtreenode] at (b2) {};
    \node[subtreenode] at (c1) {};
    \node[subtreenode] at (c2) {};
    \node[subtreenode] at (c3) {};
    \node[subtreenode] at (d3) {};

	\draw[subtreeedge] (b1) to (c1);
	\draw[subtreeedge] (b2) to (c2);
	\draw[subtreeedge] (b2) to (c3);
	\draw[subtreeedge] (c2) to (d3);

    \draw[shadededge,green!40] (a1) to (b1);
     \draw[shadededge,green!40] (a1) to (b2);
     \draw[shadededge,green!40] (d1) to (e);
    
    \draw[shadededge,red!40] (b1) to (d1);
     \draw[shadededge,red!40] (c1) to (d2);
       \draw[shadededge,red!40] (c3) to (d4);

    \node[shadednodegray]  at (b2) {};
    \node[shadednodegray] at (c2) {};

    \draw[shadededge,gray] (c2) to (b2);
    
    \node[dot]  at (a1) {};
    \node[dot] at (b1) {};

    \node [shadednodeblue] at (d1) {};
    \node [shadednodeblue] at (d2) {};
    \node [shadednodeblue] at (d4) {};
    \node [shadednodeblue] at (b1) {};
    \node [shadednodeblue] at (b2) {};
        \node [shadednodeblue] at (b2) {};
        \node [shadednodeblue] at (e) {};

	\node[dot] (a1) at (0,0) {};
	\node[dot] (b1) at (-1,1) {};
	\node[dot] (b2) at (1,1) {};
	\node[dot] (c1) at (-0.5, 1.5) {};
	\node[dot] (c2) at (0.5, 1.5) {};	
	\node[dot] (c3) at (1.5, 1.5) {};
	\node[dot] (d1) at (-1.7, 1.7) {};
	\node[dot] (d2) at (-0.7,2.5) {};
	\node[dot] (d3) at (1,2) {};
	\node[dot] (d4) at (2,2) {};
	\node[dot] (e) at (-1.8,2.3) {};

	\draw[kernel] (a1) to (b1);
	\draw[kernel] (a1) to (b2);
	\draw[kernel] (b1) to (d1);
	\draw[kernel] (b1) to (c1);
	\draw[kernel] (b2) to (c2);
	\draw[kernel] (b2) to (c3);
	\draw[kernel] (c1) to (d2);
	\draw[kernel] (c2) to (d3);
	\draw[kernel] (c3) to (d4);
	\draw[kernel] (d1) to (e);
	
	\node at (-0.2,-0.2) {\footnotesize $\rho_S$};
\end{tikzpicture}
\]

\subsubsection{Renormalisation operators}\label{sec: renormalisation operators}
In this section we introduce the operations responsible for cancelling the divergences that appear in our stochastic objects. 
A short summary motivating very similar definitions can be found in \cite[Section~4.1.1]{CH}.
We will be more concise here, mostly recalling notation from \cite{CH} before moving on to rewriting the renormalisation procedure.
\begin{definition}\label{def:collapse}
Given a subtree $S$ and $A \subset \allnodes$ with $N(S)\subset A$, we define the ``collapsing map'' $\Coll_S: (\R^3)^A \to (\R^3)^A $ given by
\begin{equ}
\Coll_S(x)_u \eqdef 
\left\{\begin{array}{cl}
  x_{\rho(S)} & \text{if $u \in \tilde N(S)$,} \\
  x_u & \text{otherwise (i.e. if $u\in A\setminus \tilde N(S)$).}
\end{array}\right.
\end{equ}
Here $x\in (\R^3)^A $ and $u\in A$.
In plain words, $\Coll_S$ maps all points in $\tilde N(S)$ onto the root of $S$ (or more precisely re-defines the coordinates of all points in $\tilde N(S)$ to be the coordinates of $\rho(S)$), with other points fixed.
\end{definition}
\begin{definition}\label{def-Y0Y1}
For $S \in \Div$ we define operators $\mathscr{Y}^{(0)}_{S}, \mathscr{Y}^{(1)}_{S} :\allf \rightarrow \allf$ given by 
\begin{equs}
{}&
(\mathscr{Y}^{(0)}_{S}F)(z) \eqdef F(\Coll_{S}(z))\;
\textnormal{ and}\\
(\mathscr{Y}^{(1)}_{S}F)(z) \eqdef& 
\mathbbm{1}\{ |S^{0}|_{\s} \in (-2,-1)\}
\sum_{
\substack{
u \in \tilde{N}(S)\\
j \in \{1,2\}}}
(z_{u}^{(j)} - z_{\rho_{S}}^{(j)}) \,
(D^{\delta_{u,j}}F)(\Coll_{S}(z))\;,
\end{equs}
We also define $\mathscr{Y}_{S} \eqdef \mathscr{Y}^{(0)}_{S} + \mathscr{Y}^{(1)}_{S}$. 
\end{definition}

For any $\J \in \mfJ$ and $\mcF \in \mathbb{F}$,  we recursively define a family of operators $(H_{\J,\mcF,S}: S \in \mcF)$ by
\begin{equation} \label{def: renormalised kernel 1}
\begin{split}
[H_{\J,\mcF,S}(\phi)](x)\ 
&\eqdef 
\int_{{\nrmod[\mcF,S]}} \back dy \
\J^{L_{\mcF}(S)^{(2)}}(y\sqcup x_{\rho_{S}})\,  
\ke{\mathring{K}_{\mcF}(S)}{}(y \sqcup x_{\rho_{S}})\\
& \quad
\cdot
H_{\J,\mcF, C_{\mcF}(S)}
\left[
\ke{K^{\partial}_{\mcF}(S)}{}
\cdot
(-\mathscr{Y}_{S}\phi)
\right](x_{\tilde{N}(S)^c} \sqcup
y)\;.
\end{split}
\end{equation}
The recursion is initialised by postulating that $H_{\J,\mcF,\emptyset}$ is the identity.
Furthermore, for any sub-forest $\mcG \subset \mcF$ with all trees of $\mcG$ pairwise disjoint,
we used the notation (on the right hand side of \eqref{def: renormalised kernel 1})
\begin{equ}[e:H-is-product-H]
H_{\J,\mcF,\mcG} \eqdef \prod_{T \in G} H_{\J,\mcF,T}\;,
\end{equ}
 where the order of the product is
irrelevant since, for any two disjoint trees $T$, the corresponding operators commute.
We will also use this convention for variants of the operators $H$ introduced later. 

Note that if $\phi \in \allf$ only depends on $(x_{v})_{v \in A}$ where $A \subset \allnodes$ then $H_{\CJ,\mcF,S}[\phi]$ depends only on $(x_{v})_{v \in A \setminus \tilde{N}(S)}$.

\subsection{The first explicit formula for moments}
We need a last bit of notation before we can write out explicitly the renormalised moment graphs we want to estimate. Recall the convention \eqref{e:Z-on-forest}.
For $A = \sT_{j} $ for some $j\in[2p]$ or $A=D_{2p}$, define
\begin{equ}[e:def-NFA-KFA-LFA]
\nodesleft{\mcF}{A}
\eqdef
N(A)
\setminus
\tilde{N}(\mcF)\,,
\quad 
\kernelsleft{\mcF}{A}
\eqdef
K(A) \setminus \bar{K}^{\downarrow}(\mcF)\,,
\quad
\leavesleft{\mcF}{A}
\eqdef
L(A) 
\setminus 
L(\mcF)\,.
\end{equ}
Note that our notation is such that although $N(A)=L(A)$ by Lemma~\ref{lem:neg-trees},
generally we have $\nodesleft{\mcF}{A} \not = \leavesleft{\mcF}{A}$.

For any $\CJ \in \mfJ$ and test function $\psi$ we define $M(\psi,\CJ) \in \R$ via
\begin{equs}  
M[\psi,\CJ]
\eqdef&
\sum_{
\substack{
\mcG \in \mathbb{F},\\ 
\cD \subset \cut \setminus K(\mcG)}
}
\int_{\nodesleft{\mcG}{D_{2p}}}dy\  
\J^{\leavesleft{\mcG}{D_{2p}}^{(2)}}(y)
\cdot
\ke{\kernelsleft{\mcG}{D_{2p}} \setminus \cD}{}(y)\\
&
\quad 
\cdot
\Big(
\prod_{j=1}^{2p}
\psi(y_{\rho_{\sT_{j}}})
\Big)
\cdot
\RKer^{\kernelsleft{\mcG}{D_{2p}} \cap \cD}(z)
\cdot
\powroot{\nodesleft{\mcG}{D_{2p}}}{\sn}{\logof}(z)
\\
& 
\quad 
\cdot 
H_{\CJ,\mcG,\overline{\mcG}} \left[ 
\RKer^{K^{\downarrow}(\bar{\mcG}) \cap \cD}
\cdot 
\ke{K^{\downarrow}(\bar{\mcG}) \setminus \cD}{}
\powroot{\tilde{N}(\bar{\mcG})}{\sn}{\logof}
\right](z)
\label{eq: explicit formula for moment}
\end{equs}
where $z = x_{\logof} \sqcup y$.
Note that $M(\psi,\CJ)$ does not depend on $x_{\logof}$ thanks to translation invariance.
Recall also that $\bar\mcG\subset \CG$ denotes the set of maximal elements of $\CG$.
(These are necessarily disjoint since, by definition of $\mathbb{F}$, any two elements of $\CG$ are either
disjoint or ordered.)

We then have the following result.
\begin{proposition}\label{explicit formula} 
\hao{In this proof we may write one or two sentences about how this model formula connects to the algebra part}
Let $\eps>0$ 
and let $(\hat{\Pi}^{\e},\hat{\Gamma}^{\e})$ be the neutral BPHZ lift of the 
noise $\xi^{\e}$ defined in Section~\ref{sec: neutral BPHZ}. 
Then, for $x_{\logof} \in (\R \times \T^{2})^{\{\logof\}}$, any test function $\psi$, any $j \in [2p]$, and using the shorthand $\sT = \sT_{j}$, 
\begin{equs}\label{eq: explicit formula for single copy}
\hat{\Pi}^{\e}_{x_{\logof}}[\sT^{\sn\sl} ](\psi)
=&
\sum_{
\substack{
\mcG \in \mathbb{F}_{j} \\ 
\cD \subset \cut_j \setminus K(\mcG)}
}
\int_{\nodesleft{\mcG}{\sT}}dy\  
\xi^{\leavesleft{\mcG}{\sT},\e}(y)
\cdot
\ke{\kernelsleft{\mcG}{\sT} \setminus \cD}{}(y)\\
&
\qquad \qquad 
\cdot
\psi(z_{\rho_{\sT}})
\cdot
\RKer^{\kernelsleft{\mcG}{\sT} \cap \cD}(z)
\cdot
\powroot{\nodesleft{\mcG}{\sT}}{\sn}{\logof}(z)
\\
& 
\qquad \qquad 
\cdot 
H_{\CJ(\iota),\mcG,\overline{\mcG}} \left[ 
\RKer^{K^{\downarrow}(\bar{\mcG}) \cap \cD}
\cdot 
\ke{K^{\downarrow}(\bar{\mcG}) \setminus \cD}{}
\powroot{\tilde{N}(\bar{\mcG})}{\sn}{\logof}
\right](z)
\end{equs}
where $z = x_{\logof} \sqcup y$. 
Moreover, for any $\e$-assignment $\iota$, 
\begin{equ}[e:moment-is-M]
\E
\Big[
\Big(
\prod_{j=1}^{p} 
\hat{\Pi}^{\iota(j)}_{x_{\logof}}[\sT^{\sn\sl} ](\psi)
\Big)
\Big(
\prod_{j=p+1}^{2p}
\overline{ 
\hat{\Pi}^{\iota(j)}_{x_{\logof}}[\sT^{\sn\sl} ](\psi)
}
\Big)
\Big]
=
M[\psi,\CJ(\iota)]\;,
\end{equ}
where $M[\psi,\CJ(\iota)]$
is defined in \eqref{eq: explicit formula for moment}.
\end{proposition}

\begin{proof}
The identity \eqref{eq: explicit formula for single copy} is a
modification of \cite[Proposition~4.22]{CH},
and we only point out the differences here.
The  forests in $\mathbb{F}_{j}$  here are required to be neutral (and otherwise the same as in  \cite[Section~4]{CH}), 
and this precisely corresponds to the fact that
our model is  the neutral BPHZ lift of the 
noise $\xi^{\e}$. 

Another  difference  is that here we do not perform a Wiener chaos decomposition
for the random object $\hat{\Pi}^{\e}_{x_{\logof}}[\sT^{\sn\sl} ](\psi)$,
so in \eqref{eq: explicit formula for single copy} we only have products of the noises rather than Wick products as in \cite[Proposition~4.22]{CH},
and we only have summation over $\mcG$ and $\cD$ in \eqref{eq: explicit formula for single copy} (in contrast with \cite[Proposition~4.22]{CH} where a noise set $\tilde L$ whose cardinality stands for the order of the homogenous Wiener chaos and a partition $\pi$ encoding the contracted noises have to be summed as well).

Finally our operator $H$ defined in 
 \eqref{def: renormalised kernel 1}
differs from  \cite[Section~4.1.1]{CH} in that 
the factor called $\mbox{Cu}_\pi^{L_{\mcF}(S)}$ which is a cumulant product
therein
is replaced by $\J^{L_{\mcF}(S)^{(2)}}$
in our case, and this is due to the fact
that we write moments of the noises as \eqref{e:E-xi-J} or \eqref{eq: coulomb interaction},
rather than performing a cumulant expansion.

\eqref{e:moment-is-M} follows straightforwardly from \eqref{eq: explicit formula for single copy}. 
Indeed a couple $(\mcG, \cD)$ in \eqref{eq: explicit formula for moment} amounts to selecting a couple $(\mcG, \cD)$ from each $\bar{T}_j$ as in \eqref{eq: explicit formula for single copy}.
The set of integrated nodes $\nodesleft{\mcG}{D_{2p}}$ \eqref{eq: explicit formula for moment} is just the union of the  sets of integrated nodes $\nodesleft{\mcG}{\sT_j}$ for all $j\in[2p]$.
Various factors in the integrand of \eqref{eq: explicit formula for moment} 
are just products of the corresponding factors in the integrand of \eqref{eq: explicit formula for single copy}, and in particular the $H$ operator in \eqref{eq: explicit formula for moment} is the product of the $H$ operators in  \eqref{eq: explicit formula for single copy} by \eqref{e:H-is-product-H}.
\end{proof}

\begin{example}\label{ex:dipole}
Assume that $\beta^2 \ge 4\pi$ and consider the ``dipole'' $ \xi_{-} \CI \xi_{+}$
whose corresponding tree $\sT^{\sn\sl}$ 
 consists of only one edge $e=(e_+,e_-)$ with $\sn=0$, as shown in the following picture
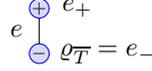
\[
\begin{tikzpicture}
\draw (0,0) node[xi] {\tiny $ -$} -- (0,0.6) node[xi] {\tiny $+$};
\node at (0.9, 0) {\small $\rho_{\sT}=e_-$};
\node at (0.5, 0.6) {\small $e_+$};
\node at (-0.3, 0.3) {\small $e$};
\end{tikzpicture}
\]
Its moments are analysed in \cite[Section~4]{HaoSG}. 
Three terms appear
on the right hand side of \eqref{eq: explicit formula for single copy}:
1) $\mcG=\cD=\emptyset$, 2) $\mcG=\emptyset$ and $\cD=\{e\}$,
3) $\mcG=\{\bar T\}$ and  $\cD=\emptyset$, as shown in the following three pictures respectively
\begin{equ} [e:3terms-dipole]
\Pi^{\xi}_{x_{\logof}}[\sT^{\sn\sl} ](\psi)
\quad =\quad 
\begin{tikzpicture}[baseline=0]
\node[xi] at (0,0) (-) {\tiny $ -$}; \node[xi] at (0,0.6) (+) {\tiny $+$};
\draw[kernel] (+)  to  (-);
\node[root] at (0,-0.6)  (root) {};
\draw[testfcn] (0,-0.1)  to (root);
\end{tikzpicture}
\quad -\quad 
\begin{tikzpicture}[baseline=0]
\node[xi] at (0,0) (-) {\tiny $ -$}; \node[xi] at (0,0.6) (+) {\tiny $+$};
\node[root] at (0,-0.6)  (root) {};
\draw[kernel,bend right = 50] (+)  to  (root);
\draw[testfcn] (0,-0.1)  to (root);
\end{tikzpicture}
\quad -\quad 
\begin{tikzpicture}[baseline=0]
\node[charge] at (0,0) (-) {\tiny $ -$}; \node[charge] at (0,0.6) (+) {\tiny $+$};
\node[root] at (0,-0.6)  (root) {};
\draw[kernel,bend right = 50] (+)  to  (-);
\draw[cov,bend left = 50] (+)  to  (-);
\draw[testfcn] (0,-0.1)  to (root);
\end{tikzpicture}
\end{equ}
Here the graphs represent concrete integrals (as opposed to combinatorial trees), with similar conventions as in for instance \cite{HS15,CS16}, namely,
$\tikz  [baseline=-3] \node[root] {};$ represents $x_{\logof}$,
the arrow $\tikz [baseline=-3] \draw[kernel] (0,0) to (0.8,0);$ represents 
the kernel $K$, 
the arrow $\tikz [baseline=-3] \draw[testfcn] (0.5,0) to (0,0);$
represents $\psi$, and $\tikz [baseline=-3] \draw[cov] (0,0) to (0.5,0);$ represents $\CJ^{-1}$,
 $\tikz [baseline=-3] \node[xi] {\tiny $ -$};$ and  $\tikz [baseline=-3] \node[xi] {\tiny $ +$};$
 represent mollified noises $\xi_{-}$ or $\xi_{+}$ 
respectively with the corresponding space-time variables  integrated, and
 $\tikz [baseline=-3] \node[charge] {\tiny $ -$};$ or  $\tikz [baseline=-3] \node[charge] {\tiny $ +$};$ represents an integration variable over  space-time.
 \footnote{ $\tikz [baseline=-3] \node[charge] {\tiny $ -$};$ and  $\tikz [baseline=-3] \node[charge] {\tiny $ +$};$ are really the same; we prefer to take two notations for obvious bookkeeping purpose.}

In \eqref{e:3terms-dipole}, the arrow $\tikz [baseline=-3] \draw[kernel] (0,0) to (0.8,0);$
pointing to the vertex $\tikz  [baseline=-3] \node[root] {};$
comes from a factor $\RKer^{\kernelsleft{\mcG}{\sT} \cap \cD}$,
and those pointing to other vertices come from factors 
$\ke{\kernelsleft{\mcG}{\sT} \setminus \cD}{}$.
%
All three integrands contain a factor $H_{\CJ,\mcG,\bar\mcG} \left[ 1 \right]$.
For the first and second terms, one has $\CG = \bar\CG = \emptyset$, so that $H_{\CJ,\mcG,\bar\mcG} \left[ 1\right]=1$ by definition.
For the last term however, one has
\begin{equ} [e:H1dipole]
H_{\CJ,\mcG,\bar\mcG} \left[ 1\right] 
=  -\int_{\R^3}  \CJ^{-1} (x_{e_\ch} - x_{e_\p}) K(x_{e_\p} - x_{e_\ch}) \,d x_{e_\ch} \;.
\end{equ}

Regarding the $2p$-th moment, taking $p=2$ as example,
denote the $4$ copies of the tree by $\bar T^{(1)},\cdots,\bar T^{(4)}$ and their edges by $e^{(1)},\cdots, e^{(4)}$.
All the possible $(\mcG,\cD)$ in \eqref{eq: explicit formula for moment} are of the form
$
\mcG = \{\bar T^{(i_1)} ,\cdots,  \bar T^{(i_m)}  \}
$ and
$
\cD =  \{e^{(j_1)} ,\cdots,  e^{(j_n)}  \}
$ 
such that $  \{i_1,\cdots, i_m\} $ 
and $ \{j_1,\cdots, j_n\} $ are disjoint.
Fixing such a pair  $(\mcG,\cD)$,
one has a factor
$H_{\CJ,\mcG, \bar\mcG } \left[ 1\right]$ which is a product of 
$
H_{\CJ,\mcG, \bar T^{(i_k)} } \left[ 1\right]
$
each defined as in \eqref{e:H1dipole}.
Below are pictures with $\mcG$ consisting of 1 and 2 trees respectively, and 
$\cD=\emptyset$ in both cases
\begin{equ} [e:dipole-moments]
- \qquad
\begin{tikzpicture}[baseline=10]
\node at (0,2) [charge] (1) {\tiny $+$};  \node at (1.5,2) [charge] (2) {\tiny $-$};
\node at (0,1) [charge] (3) {\tiny $+$};  \node at (1.5,1) [charge] (4) {\tiny $-$};
\node at (0,0) [charge] (5) { \tiny $+$};  \node at (1.5,0) [charge] (6) {\tiny $-$};
\node at (0,-1) [charge] (7) {\tiny $+$};  \node at (1.5,-1) [charge] (8) {\tiny $-$};
\draw [kernel,bend left=20] (1) to (2);  \draw [cov] (1) to (2);  
\draw [kernel,bend left=20] (3) to (4); \draw [cov] (3) to (4); 
\draw [cov] (3) to (6); \draw [cov] (3) to (8);
\draw [cov] (4) to (5); \draw [cov] (4) to (7);
\draw [kernel,bend left=20] (6) to (5); \draw [cov]  (5) to (6); 
\draw [cov] (5) to (8);
\draw [cov] (6) to (7);
\draw [kernel,bend left=20] (8) to (7); \draw [cov]  (7) to (8);
\draw [covi] (3) to (5);   \draw [covi] (5) to (7);
\draw [covi] (3) to [bend right=30] (7); 
\draw [covi] (4) to (6);  \draw [covi] (6) to (8);
 \draw [covi] (4) to [bend left=30] (8); 
 \end{tikzpicture}
%
%
\qquad + \qquad
\begin{tikzpicture}[baseline=10]
\node at (6,2) [charge] (1) {\tiny $+$};  \node at (7.5,2) [charge] (2) {\tiny $-$};
\node at (6,1) [charge] (3) {\tiny $+$};  \node at (7.5,1) [charge] (4) { \tiny $-$};
\node at (6,0) [charge] (5) {\tiny $+$};  \node at (7.5,0) [charge] (6) {\tiny $-$};
\node at (6,-1) [charge] (7) {\tiny $+$};  \node at (7.5,-1) [charge] (8) {\tiny $-$};
\draw [kernel,bend left=20] (1) to (2);  \draw [cov] (1) to (2);  
\draw [kernel,bend left=20] (3) to (4);  \draw [cov] (3) to (4); 
\draw [cov] (3) to (6); 
\draw [cov] (4) to (5); 
\draw [kernel,bend left=20] (6) to (5);  \draw [cov]  (5) to (6); 
\draw [kernel,bend left=30] (8) to (7);  \draw [cov] (7) to (8);  
\draw [covi] (3) to (5);   
\draw [covi] (4) to (6);  
\end{tikzpicture}
\end{equ}
Here a blue line $\tikz [baseline=-3] \draw[covi] (0,0) to (0.7,0);$ represents the function $\CJ_\eps$
and we omitted the arrows $\tikz [baseline=-3] \draw[testfcn] (0.5,0) to (0,0);$
and vertex $\tikz  [baseline=-3] \node[root] {};$.
Nodes $\tikz[baseline=-3] \node[charge] {\tiny $+$};$ and $\tikz[baseline=-3] \node[charge] {\tiny $-$};$
represent integration variables {\it without} noise, but we choose to keep the signs in our pictorial representation to remind ourselves where
these come from.

The moment formula for the dipole was obtained in \cite[Eq.~4.8-4.10]{HaoSG} (with $k=1$ therein), with difference being that our $\Pi^{\xi}_{x_{\logof}}[\sT^{\sn\sl} ]$ here is the same as $\tilde\Psi_\eps^{k\bar l}$ rather than 
$\Psi_\eps^{k\bar l}$ (with $k=l=1$) therein (see \cite[Eq.~4.1,4.4]{HaoSG}), and the formula obtained there did not expand
the functions $K(x_{e_\p} - x_{e_\ch}) -K(x_{\logof} - x_{e_\ch})$.
\end{example}

Before proceeding to the next section, we sketch why~\eqref{eq: explicit formula for moment} is \emph{not} straightforward to directly estimate. 
Assume the decoration $\bar\mfn=0$. 
Suppose we are given $S \in \Div$ which is minimal, i.e. one cannot find any proper subtree tree of $S$ also 
belonging to $\Div$. Then one expects the two summands corresponding to $(\mcG,\cD) = (\{S\},\emptyset)$ 
and $(\mcG,\cD) = (\emptyset,\emptyset)$ to be individually divergent, but the {\it combination} of the two terms to be finite. 
%
Writing the combination of these two terms together, one gets
\begin{equ} \label{e:explanation-IS}
\int_{N(D_{2p}) \setminus \tilde{N}(S)}dy\  
\J^{  (L(D_{2p}) \setminus L(S))^{(2)}}(y)
\cdot
\ke{K(D_{2p}) \setminus \overline{K}^{\downarrow}(S)}{}(y)
\cdot \psi(y_{\rho_{\sT^{(1)}}})\psi(y_{\rho_{\sT^{(2)}}})
\cdot \CI_{S}(z)
\end{equ}
where
\begin{equs} \label{e:def-I_S}
{}&
\CI_{S}(z)
\eqdef
\int_{\tilde{N}(S)}dx\ 
\J^{L(S)^{(2)}}(x \sqcup y)\cdot
\ke{K(S)}{}(x \sqcup y) 
\\
&
\qquad
\qquad
\qquad
\times
\Big[
\ke{K^{\downarrow}(S)}{}
\J^{\tilde{L}(S)}
-
\mathscr{Y}_{S}
\big(
\ke{K^{\downarrow}(S)}{}
\big)
\Big](x \sqcup z)\;,
\end{equs}
and $\tilde{L}(S)$ is the collection of all $e \in L(D_{2p})^{(2)}$ with $|e \cap L(S)| = 1$. 
The fact that the two contributions mentioned earlier diverge separately corresponds to the fact that if one estimates the two terms being subtracted in \eqref{e:def-I_S} separately, then these two pieces of $\CI_{S}(z)$ each diverge for fixed $z$.

In order to see that the combination of these terms is finite one would like to use a Taylor remainder estimate coming from an occurrence of an operator $(\Id - \mathscr{Y}_{S})$ (where $\Id$ is the identity operator), but the formula for $\CI_{S}(z)$ is missing a factor of $\J^{\tilde{L}(S)}$ for this Taylor expansion.
The observation that we implement in what follows is that one can insert this missing factor ``for free'' by leveraging the fact that $S$ is neutral along with parity considerations, namely $\CI_{S}(z) = \tilde{\CI}_{S}(z)$ where 
\begin{equs}
{}&
\tilde{\CI}_{S}(z)
\eqdef
\int_{\tilde{N}(S)}dx\ 
\ke{K(S)}{}(x \sqcup y) 
\J^{L(S)^{(2)}}(x \sqcup y)\\
{}&
\qquad
\qquad
\qquad
\times
(\Id - \mathscr{Y}_{S})\Big[
\ke{K^{\downarrow}(S)}{}
\J^{\tilde{L}(S)}
\Big](x \sqcup z)\;.
\end{equs}
In the next section we show that such insertions of $\J^{\bullet}$ are allowed in the moment formula so that we are able to rewrite the outcome of renormalisation cancellations as Taylor remainders.
Some care needs to be taken when implementing these insertions of $\J^{\bullet}$ in the general case since we may have multiple divergences and subdivergences which is why in the next section we decide to introduce modified versions of these $H_{\CJ,\mcG,S}$ operators.  

We also illustrate the idea of inserting missing factors $\J^{\bullet}$  by the following example.

\begin{example} \label{ex:combination2}
Consider the two terms in \eqref{e:dipole-moments} ($p=2$ now) with $\mcG=\{T\}$ and $\mcG=\{T,S\}$ respectively,
where
 $T,S \in \Div$ are the dipoles in the top and the bottom respectively.
 We would like to combine these two terms  to cure the divergence of $S$.
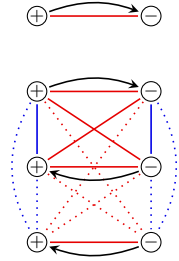
\[
\begin{tikzpicture}[baseline=0]
\node at (0,2) [charge] (1) {\tiny $+$};  \node at (1.5,2) [charge] (2) {\tiny $-$};
\node at (0,1) [charge] (3) {\tiny $+$};  \node at (1.5,1) [charge] (4) {\tiny $-$};
\node at (0,0) [charge] (5) { \tiny $+$};  \node at (1.5,0) [charge] (6) {\tiny $-$};
\node at (0,-1) [charge] (7) {\tiny $+$};  \node at (1.5,-1) [charge] (8) {\tiny $-$};
\draw [kernel,bend left=20] (1) to (2);  \draw[cov] (1) to (2);  
\draw [kernel,bend left=20] (3) to (4); \draw[cov] (3) to (4); 
\draw [cov] (3) to (6); \draw[cov,gone] (3) to (8);
\draw [cov] (4) to (5); \draw[cov,gone] (4) to (7);
\draw [kernel,bend left=20] (6) to (5); \draw [cov]  (5) to (6); 
\draw[cov,gone] (5) to (8);
\draw[cov,gone]  (6) to (7);
\draw [kernel,bend left=20] (8) to (7); \draw [cov]  (7) to (8);
\draw [covi] (3) to (5);   \draw [covi,gone] (5) to (7);
\draw [covi,gone] (3) to [bend right=30] (7); 
\draw [covi] (4) to (6);  \draw [covi,gone] (6) to (8);
 \draw [covi,gone] (4) to [bend left=30] (8); 
 \end{tikzpicture}
 \]
Comparing these two terms, the edges appearing in the first graph
in~\eqref{e:dipole-moments} but missing in the second one
are depicted by  the dotted lines here.

 It is clear that  the product of all the dotted lines  
being acted upon by $\mathscr{Y}_{S}^{(0)}$ is $1$, 
since the blue and red lines represent reciprocal functions and thus
cancel out when the bottom dipole is collapsed into one point. The reason why $\mathscr{Y}_{S}^{(1)}$ will not cause any problem will follow from a parity argument. This means that we can insert these missing dotted lines inside $\mathscr{Y}_{S}$ in the second graph ``for free''.
\footnote{It turns out that in the end we are able to insert all the missing edges in $L(D_{2p})^{(2)}$ in a suitable way, see the proof of Lemma~\ref{lem:multilinear}.}

On the other hand the first graph in~\eqref{e:dipole-moments} does not have this $\mathscr{Y}_{S}$ operation;  imagine there is an $\Id$ instead.  We then get a Taylor remainder from operator $(\Id - \mathscr{Y}_{S})$ as explained before this example, which offsets the some singularity of $S$.
\end{example}
\subsection{Rewriting the moment formula}

It turns out that the moment formula given in Proposition~\ref{explicit formula} above
is not very useful in exploiting the extra cancellation due to ``parity''
as mentioned in the introduction.\footnote{\cite[Theorem~4.3]{HaoSG} only proved the convergence
of the object in Example~\ref{ex:dipole} for $\beta^2 < 6\pi$.
This was because the extra cancellation due to ``parity'' was not exploited there.
In this paper we will show convergence of all the relevant objects
for $\beta^2 < 8\pi$.
}
In order to exploit this extra mechanism we now rewrite the moment formula.

For $\J \in \mfJ$, $\mcF \in \mathbb{F}$, and $S \in \mcF$ we inductively define operators $\mathring{H}_{\J,\mcF,S}: \allf \rightarrow \allf$ via
\begin{equation}\label{def: renormalised kernel 1.5}
\begin{split}
[\mathring{H}_{\J,\mcF,S}(\phi)](x)\ 
&\eqdef 
\int_{{\nrmod[\mcF,S]}} \back dy \,
\J^{L_{\mcF}(S)^{(2)}}(y\sqcup x_{\rho_{S}})\,  
\ke{\mathring{K}_{\mcF}(S)}{}(y \sqcup x_{\rho_{S}})\\
& \qquad \qquad
\cdot
\mathring{H}_{\J,\mcF, C_{\mcF}(S)}
\left[
\ke{K^{\partial}_{\mcF}(S)}{}
\cdot
(-\mathscr{Y}^{(0)}_{S}\phi)
\right](x_{\tilde{N}(S)^c} \sqcup
y)\;,
\end{split}
\end{equation}
with $\mathring{H}_{\J,\mcF,\emptyset}$ defined 
as the identity operator. 
Note that since the factor $-\mathscr{Y}^{(0)}_{S}\phi$ does not depend on the variables for the nodes in $\tilde N(S)$, 
by the definition of $\mathscr{Y}^{(0)}_{S}$ this factor can be pulled out of the entire integral 
in \eqref{def: renormalised kernel 1.5}
(whereas in \eqref{def: renormalised kernel 1} one can not pull out $\mathscr{Y}_{S}\phi$
because the linear polynomial in $\mathscr{Y}^{(1)}_{S}$ will still depend on $\tilde N(S)$). 
We prefer to write \eqref{def: renormalised kernel 1.5} in the above way
since it is clearer to compare it with \eqref{def: renormalised kernel 1}.

\begin{definition} \label{def:Csym}
Given $A \subset \allnodes$ 
 (where $\allnodes$
is defined in \eqref{e:def-allnodes}),
we define $\allf_{\tiny{\textnormal{sym}}}(A)$ to be the set of all functions $f \in \allf$ such that:
\begin{itemize}
\item
$f$ depends on $(x_{v})_{v \in \allnodes}$ only through $(x_{v})_{v \in A}$. Moreover, $f$ can be written as a function of the variables $(x_{a} - x_{b})_{\{a,b\} \in A^{(2)}}$, namely for any $h\in \R \times \T^{2}$
 \[
  f\left(\{x_v\}_{v\in A}\right)
 = f\left(\{x_v+h\}_{v\in A}\right)\;.
 \]
\item 
 $f$ is invariant if one flips the sign of every spatial component of every $x_{v}$, $v \in A$, namely,
 \[
 f\left(\{(x_v^{(0)},x_v^{(1)},x_v^{(2)})\}_{v\in A}\right)
 = f\left(\{(x_v^{(0)},-x_v^{(1)},-x_v^{(2)}) \}_{v\in A}\right)\;.
 \]
\end{itemize}
\end{definition}

\begin{remark}\label{parity of kernels}
We observe that for any $\{a,b\} \in L(D_{2p})^{(2)}$ one has $\J^{\{a,b\}}
 \in \allf_{\tiny{\textnormal{sym}}}(\{a,b\})$ and for any $e \in K(D_{2p})$ one has $\ke{\{e\}}{} \in \allf_{\tiny{\textnormal{sym}}}(\{e_{\ch},e_{\p}\})$. 
\end{remark}  
\begin{remark}\label{rem:property-Csym}
Obviously, given $A,\bar A \subset \allnodes$, 
$f\in\allf_{\tiny{\textnormal{sym}}}(A)$ and
$g\in\allf_{\tiny{\textnormal{sym}}}(\bar A)$,
one has $f\cdot g \in\allf_{\tiny{\textnormal{sym}}}(A\cup \bar A)$.
Moreover,
If $A\cap \bar A = \emptyset$ and $f \in\allf_{\tiny{\textnormal{sym}}}(A\sqcup \bar A)$,
 one has $\int_{\bar A} f(x,y) \,dy \in\allf_{\tiny{\textnormal{sym}}}(A)$.
\end{remark}  
\begin{lemma}\label{lem: trans-invar}
 For any $\CJ \in \mfJ$, $A \subset \allnodes$, $\mcF \in \mathbb{F}$, and $S \in \mcF$ one has that $\mathring{H}_{\J,\mcF,S}$ maps $\allf_{\tiny{\textnormal{sym}}}(A)$ into 
 $\allf_{\tiny{\textnormal{sym}}}((A \cup \{\rho_{S}\}) \setminus \tilde{N}(S))$.
\end{lemma}
\begin{proof}
We prove the statement by induction in the depth of the forest $C_{\mcF}(S)$. 
We only give the inductive step since the base case (occuring when $C_{\mcF}(S) = \emptyset$) is strictly easier. 
Fix a forest $\mcF$ and $S \in \mcF$, our inductive hypothesis is that the lemma has been proven for every operator $\{\mathring{H}_{\J,\mcF,T}\}_{T \in C_{\mcF}(S)}$.

Fix a set $A \subset N^{\ast}$ and $\phi \in \allf_{\tiny{\textnormal{sym}}}(A)$. 
We want to show that the corresponding RHS of \eqref{def: renormalised kernel 1.5} belongs to 
$\allf_{\tiny{\textnormal{sym}}} \Big((A \cup \{\rho_{S}\}) \setminus \tilde{N}(S)\Big)$.
To do this, by Remark~\ref{rem:property-Csym}
it suffices to show that the {\it integrand}
 is in 
 $\allf_{\tiny{\textnormal{sym}}} \Big( (A \cup N(S))  \setminus \tilde{N}(C_{\mcF}(S))\Big)$ since $\tilde{N}(S) = \tilde{N}_{\mcF}(S) \sqcup \tilde{N}(C_{\mcF}(S))$ and thus
 \[
 (A \cup N(S))  \setminus \tilde{N}(C_{\mcF}(S))
 =
 \big((A \cup \{\rho_{S}\}) \setminus \tilde{N}(S)\big)
 \sqcup
 \tilde N_{\mcF}(S)\;.
 \]
This is clearly true for the first line of the integrand which only depends on $N(S) \setminus \tilde{N}(C_{\mcF}(S))$ and has the two  symmetry properties required in Definition~\ref{def:Csym}. 
To check it for the second line we observe that $\ke{K^{\partial}_{\mcF}(S)}{} \in  \allf_{\tiny{\textnormal{sym}}}(N(S))$ and $\mathscr{Y}^{(0)}_{S}\phi \in \allf_{\tiny{\textnormal{sym}}}(A  \setminus \tilde{N}(S))$. 
By our inductive hypothesis $\mathring{H}_{\J,\mcF,C_{\mcF}(S)}$ maps, for any $\tilde{A} \subset \allnodes$, $\allf_{\tiny{\textnormal{sym}}}(\tilde{A})$ into  
\[
\allf_{\tiny{\textnormal{sym}}}((\tilde{A} \cup \{\rho_{T}\}_{T \in C_{\mcF}(S)}) \setminus \tilde{N}(C_{\mcF}(S)))\;.
\] 
Thus by setting $\tilde{A} = N(S) \cup A$ we obtain the desired claim for the second line of the integrand.
\end{proof}
\begin{lemma}\label{lem: parity cancellation}
For any $\CJ \in \mfJ$, $\phi \in \allf$, $\mcF \in \mathbb{F}$, and $S \in \mcF$,
\[
\mathring{H}_{\J,\mcF,S}(\phi) = H_{\J,\mcF,S}(\phi)\;.
\] 
\end{lemma}
\begin{proof}
We prove the statement by induction in the depth of $C_{\mathcal{F}}(S)$. 
The base case occurs when $C_{\mcF}(S)$ is empty, but we prove the inductive step first. 
Fix some $n \ge 1$ and suppose the claim holds whenever $\mcF \in \mathbb{F}$, $S \in \mcF$, and $C_{\mcF}(S)$ is of depth strictly less than $n$. 

Now fix $\mcF \in \mathbb{F}$ and $S \in \mcF$ such that $C_{\mcF}(S)$ is of depth $n$. 
The induction step is immediate if $|S^{0}|_{\s} \in (-1,0)$
because $\mathscr{Y}_{S}=\mathscr{Y}^{(0)}_{S}$;
 so we assume $|S^{0}|_{\s} \in (-2,-1)$. 
Using our inductive hypothesis, $H_{\J,\mcF,S}(\phi)(x)
-
\mathring{H}_{\J,\mcF,S}(\phi)(x)$ is given by
\begin{equs}\label{eq: work for new genvert}
{}&
\int_{{\nrmod[\mcF,S]}} \back dy \
\J^{L_{\mcF}(S)^{(2)}}(y\sqcup x_{\rho_{S}})\,  
\ke{\mathring{K}_{\mcF}(S)}{}(y \sqcup x_{\rho_{S}})
\\
& \qquad\qquad\qquad  \cdot
\mathring{H}_{\J,\mcF, C_{\mcF}(S)}
\left[
\ke{K^{\partial}_{\mcF}(S)}{}
\cdot
(-\mathscr{Y}^{(1)}_{S}\phi)
\right](x_{\tilde{N}(S)^c} \sqcup
y)\\
=&
\sum_{
\substack{
u \in \tilde{N}(S)\\
j \in \{1,2\}}}
\int_{{\nrmod[\mcF,S]}} \back dy \
\J^{L_{\mcF}(S)^{(2)}}(y\sqcup x_{\rho_{S}})\,  
\ke{\mathring{K}_{\mcF}(S)}{}(y \sqcup x_{\rho_{S}})\\
&
\qquad
\qquad
\cdot
(D_{u,j}\phi)(\Coll_{S}(y \sqcup x_{\tilde{N}(S)^c}))
\mathring{H}_{\J,\mcF, C_{\mcF}(S)}
\left[
\ke{K^{\partial}_{\mcF}(S)}{}
X_{u,j}
\right](
y \sqcup x_{\rho_{S}})\\
=&
\sum_{
\substack{
u \in \tilde{N}(S)\\
j \in \{1,2\}}}
(D_{u,j}\phi)(\Coll_{S}(y \sqcup x_{\tilde{N}(S)^c}))
\int_{{\nrmod[\mcF,S]}} \back dy \
\J^{L_{\mcF}(S)^{(2)}}(y\sqcup x_{\rho_{S}})\,  
\ke{\mathring{K}_{\mcF}(S)}{}(y \sqcup x_{\rho_{S}})\\
&
\qquad
\qquad
\qquad
\qquad
\qquad
\qquad
\qquad
\qquad
\cdot
\mathring{H}_{\J,\mcF, C_{\mcF}(S)}
\left[
\ke{K^{\partial}_{\mcF}(S)}{}
X_{u,j}
\right](
y \sqcup x_{\rho_{S}}),
\end{equs}
where $X_{u,j} \in \allf$ is given by $X_{u,j}(z) = z_{u}^{(j)} - z_{\rho_{S}}^{(j)}$.
In the above computation we wrote out the action of $\mathscr{Y}_{S}^{(1)}$ on $\phi$ as a sum according to Definition~\ref{def-Y0Y1}, and then observed that we can pull out the derivative factor $(D_{u,j}\phi)(\Coll_{S}(y \sqcup x_{\tilde{N}(S)^c}))$ thanks to the action of $\Coll_{S}$ \dash in particular we can pull it out from $\mathring{H}_{\J,\mcF, C_{\mcF}(S)}$ since it does not depend on the variables of $\tilde{N}(C_{\mcF}(S))$ and then from the integral since it does not depend on the variables of $\tilde{N}_{\mcF}(S)$.

We now show that for each fixed $u$ and $j$ the summand on the right hand side of \eqref{eq: work for new genvert} vanishes.
By the definition of $\mathring{H}_{\J,\mcF, C_{\mcF}(S)}$, the summand, up to a prefactor, is given by 
\begin{equs}\label{eq: integral for lemma}
\int_{{\nrmod[\mcF,S]}} \back dy \ &
\J^{L_{\mcF}(S)^{(2)}}(y\sqcup x_{\rho_{S}})\,  
\ke{\mathring{K}_{\mcF}(S)}{}(y \sqcup x_{\rho_{S}})\,
(z_{w}^{(j)} - x_{\rho_{S}}^{(j)})
\\
&
\quad
\cdot
\mathring{H}_{\J,\mcF, C_{\mcF}(S)}
\left[
\ke{K^{\partial}_{\mcF}(S)}{}
\right](
y \sqcup x_{\rho_{S}}),
\end{equs}
where $z = y \sqcup x_{\rho_{S}}$ (as a vector indexed by $w$) and   
\[
w
\eqdef
\begin{cases}
\rho_{T} & \textnormal{ if } u \in N(T) \textnormal{ for some }T \in C_{\mcF}(S),\\
u & \textnormal{ otherwise\;.}
\end{cases}
\]
Now by Lemma~\ref{lem: trans-invar} and Remark~\ref{parity of kernels} the integrand of \eqref{eq: integral for lemma} above is translation-invariant function of $y \sqcup x_{\rho_{S}}$, 
i.e. satisfies the first condition of Definition~\ref{def:Csym};
therefore, the value of the $y$-integral does not depend on the variable $x_{\rho_{S}}$. 
At the same time,  flipping the signs of all the spatial components of $y \sqcup x_{\rho_{S}}$ flips the sign of the integrand. 
To see this, again using Lemma~\ref{lem: trans-invar} and Remark~\ref{parity of kernels},
all the factors in  the integrand of \eqref{eq: integral for lemma} except for the factor 
$(z_{w}^{(j)} - x_{\rho_{S}}^{(j)})$ are invariant  under this flipping, 
namely they satisfy the second condition of Definition~\ref{def:Csym};
and
by definition one always has $w\in \tilde N_\mcF (S) \sqcup 
\{\rho_S\}$, so this flipping 
changes the sign of the factor $(z_{w}^{(j)} - x_{\rho_{S}}^{(j)})$.
We can therefore conclude that \eqref{eq: integral for lemma} must vanish.

The proof of the base case for our induction follows by a similar but simpler parity argument:
since $C_{\mcF}(S) =\emptyset$,
the node $w$ above is always equal to $u$, and
 $H_{\J,\mcF, C_{\mcF}(S)}$ and $\mathring{H}_{\J,\mcF, C_{\mcF}(S)}$ are both given by the identity.
\end{proof}

Now we perform the aforementioned ``insertion of $\J$''.
For any $\mcF \in \mathbb{F}$ and $S \in \mcF$ we define 
\begin{equ}[e:Ppartial]
P^{\partial}_{\mcF}(S)
\eqdef
\Big\{ 
\{a,b\} \in L(S)^{(2)}:\
\begin{array}{c}
|\{a,b\} \cap L_{\mcF}(S)|=1 \textnormal{ or}\\ 
\exists \textnormal{ distinct }T,T' \in C_{\mcF}(S) \textnormal{ with } a \in L(T), b \in L(T')
\end{array}
\Big\}\;.
\end{equ}
We also inductively define operators $\bar{H}_{\CJ,\mcF,S}:\mcb{C} \rightarrow \mcb{C}$ via 
\begin{equs}\label{e:def-H-bar}
[\bar H_{\J,\mcF,S}(\phi)](x)\ 
&\eqdef 
\int_{{\nrmod[\mcF,S]}} \back dy \
\J^{L_{\mcF}(S)^{(2)}}(y\sqcup x_{\rho_{S}}) \, 
\ke{\mathring{K}_{\mcF}(S)}{}(y \sqcup x_{\rho_{S}})\\
& \quad
\cdot
\bar{H}_{\CJ,\mcF, C_{\mcF}(S)}
\left[
\CJ^{P^{\partial}_{\mcF}(S)}
\cdot
\ke{K^{\partial}_{\mcF}(S)}{}
\cdot
(-\mathscr{Y}_{S}\phi)
\right](x_{\tilde{N}(S)^c} \sqcup
y)\;,
\end{equs}
where $\bar H_{\CJ,\mcF,\emptyset}$
is defined as the identity operator.
\begin{lemma}\label{lem:H-equals-Hbar}
For any $\e$-assignment $\iota$, $\phi \in \allf$, $\mcF \in \mathbb{F}$, and $S \in \mcF$
\begin{equ} [e:Hbar-H]
\bar{H}_{\CJ(\iota),\mcF,S}(\phi) = H_{\CJ(\iota),\mcF,S}(\phi)\;.
\end{equ}
\end{lemma}
\begin{proof}
In this proof we write $\CJ \eqdef \CJ(\iota)$ to lighten notation. 
We proceed by induction in the depth of $C_\mcF(S)$.
The only difference between $\bar{H}_{\CJ,\mcF,S}$ and $H_{\CJ,\mcF,S}$
is the ``new'' factor
$\CJ^{P^{\partial}_{\mcG}(S)}$
appearing in the recursive definition  \eqref{e:def-H-bar}  
of $\bar{H}_{\CJ,\mcF,S}$.
One then sees that the base case of the induction $C_{\mcF}(S)=\emptyset$ is immediate.
Here $\bar{H}_{\CJ,\mcF, C_{\mcF}(S)}$ in \eqref{e:def-H-bar} is equal to the identity and one also has
$P^{\partial}_{\mcF}(S)=\emptyset$ so
$\CJ^{P^{\partial}_{\mcF}(S)}=1$. 

We now prove the inductive step, namely we assume 
that $C_{\mcF}(S) = \{T_1,\cdots,T_n\} \not = \emptyset$ 
and that $\bar{H}_{\CJ,\mcF,T_{j}} = H_{\CJ,\mcF,T_{j}}$ for $1 \le j \le n$.  
By Lemma~\ref{lem: parity cancellation},
we can replace the operator $H_{\CJ,\mcF, C_{\mcF}(S)}$
by $\mathring{H}_{\CJ,\mcF, C_{\mcF}(S)}$ in
\eqref{def: renormalised kernel 1}.
Next we claim that
\begin{equ}[e:claim-insert-J]
\mathring H_{\CJ,\mcF, C_{\mcF}(S)}
\left[
\ke{K^{\partial}_{\mcF}(S)}{}
\cdot
(-\mathscr{Y}_{S}\phi)
\right]
=
\mathring H_{\CJ,\mcF, C_{\mcF}(S)}
\left[
\CJ^{P^{\partial}_{\mcF}(S)}
\ke{K^{\partial}_{\mcF}(S)}{}
\cdot
(-\mathscr{Y}_{S}\phi)
\right] \;.
\end{equ}
To prove this claim we observe that since each $\mathscr{Y}^{(0)}_{T_j}$ is idempotent
and $\mathring H_{\CJ,\mcF, T_i}$ and the $\mathscr{Y}^{(0)}_{T_j}$ commute when $i\neq j$, one has
\[
\mathring H_{\CJ,\mcF, C_{\mcF}(S)}\left[
\ke{K^{\partial}_{\mcF}(S)}{}
(-\mathscr{Y}_{S}\phi)
\right]
=
\mathring H_{\CJ,\mcF, C_{\mcF}(S)}
\circ
\mathscr{Y}^{(0)}_{T_1} \circ \cdots \circ \mathscr{Y}^{(0)}_{T_n}
\left[
\ke{K^{\partial}_{\mcF}(S)}{}
(-\mathscr{Y}_{S}\phi)
\right] \;.
\]
To show \eqref{e:claim-insert-J}
it remains to prove that 
 $\mathscr{Y}^{(0)}_{T_1} \circ \cdots \circ \mathscr{Y}^{(0)}_{T_n}
(\CJ^{P^{\partial}_{\mcF}(S)})=1$.
Indeed, let 
\begin{equs}
P^{\partial}_{\mcF}(S;T_n,b) & \eqdef\{ \{a,b\}  \in P^{\partial}_{\mcF}(S): a\in T_n\}
	 \subset P^{\partial}_{\mcF}(S) 
\\
\mbox{and}\quad
P^{\partial}_{\mcF}(S;T_n) & \eqdef  \bigsqcup_{b\notin T_n} P^{\partial}_{\mcF}(S;T_n,b) \;.
\end{equs}
Then, since $T_n$ is neutral, 
for each $b\notin T_n$,
one has 
\[
 \mathscr{Y}^{(0)}_{T_n} (\CJ^{P^{\partial}_{\mcF}(S;T_n,b)})
 = \prod_{a} 
 	J^{\e,\bar\e}(z_b - z_{\rho_{T_n}})^{\mcb{q}(a)\mcb{q}(b)}
 =1
 \]
 where the product is over all $a\in T_n$ such that $\{a,b\}\in P^{\partial}_{\mcF}(S;T_n,b)$.
 Here under action of $ \mathscr{Y}^{(0)}_{T_n} $ forces 
all the factors $J^{\e,\bar\e}$ and  $1/J^{\e,\bar\e}$  to evaluate at the root 
$z_{\rho_{T_n}}$, and the neutrality of $T_n$ is crucial for these $J^{\e,\bar\e}$ and  $1/J^{\e,\bar\e}$ to cancel out.
Note that $\CJ$ is of the form $\CJ(\iota)$, in particular for each $T_{n}$ all the noises within each $T_{n}$ are the same noise $\xi_{\pm}^{\e}$ (having the same $\eps$).
As a consequence, $ \mathscr{Y}^{(0)}_{T_n} (\CJ^{P^{\partial}_{\mcF}(S;T_n)})=1$.

The following picture illustrates an example where $P^{\partial}_{\mcF}(S;T_n,b)$
consists of 4 edges, and it is clear that the 4 kernels cancel when $T_n$ is collapsed to $\rho_{T_n}$. We assume that $b = e_{>}$ for each of the four edges $e \in L(D_{2p})^{(2)}$ drawn below. 
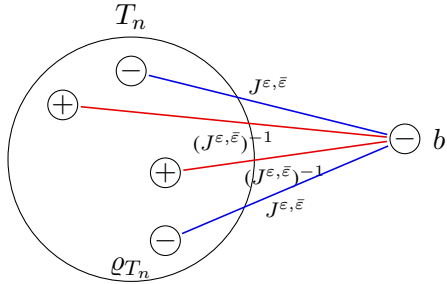
\[
\begin{tikzpicture}[scale=0.9] 
\draw (3,-0.3) circle (1.8cm);
\draw (3,1.8) node {$T_n$};
\node at (3,1) [charge] (1) {$-$};  
\node at (2,0.5) [charge] (2) {$+$}; 
\node at (3.5,-0.5) [charge] (3) {$+$}; 
\node at (3.5,-1.5) [charge] (4) {$-$}; 
\node at (3,-1.9) {$\rho_{T_n}$};
\node at (7,0) [charge] (7) {$-$}; \node at (7.5,0) {$b$};

\draw[covi] (7) to node [above] {\scriptsize $J^{\e,\bar{\e}}$} (1);
\draw[cov] (7) to node [below] {\scriptsize $(J^{\e,\bar{\e}})^{-1}$} (2);
\draw[cov] (7) to node [below] {\scriptsize $(J^{\e,\bar{\e}})^{-1}$} (3);
\draw[covi]  (7) to node [below] {\scriptsize $J^{\e,\bar{\e}}$} (4);
\end{tikzpicture}
\]
Now since $P^{\partial}_{\mcF}(S) \setminus P^{\partial}_{\mcF}(S;T_n)$
is precisely the same set 
$P^{\partial}_{\mcF}(S)$ with $C_{\mcF}(S) = \{T_1,\cdots,T_{n-1}\}$,
recursively applying $\mathscr{Y}^{(0)}_{T_j}$
then proves the claim \eqref{e:claim-insert-J}.

We now prove \eqref{e:Hbar-H} as follows.
By Lemma~\ref{lem: parity cancellation} 
we can replace the operator $H_{\CJ,\mcF, C_{\mcF}(S)}$
in the formula of $H_{\CJ,\mcF,S}(\phi)$
by the operator
$\mathring H_{\CJ,\mcF, C_{\mcF}(S)}   $.
We then apply \eqref{e:claim-insert-J} to insert the factor $\CJ^{P^{\partial}_{\mcF}(S)}$,
followed by applying Lemma~\ref{lem: parity cancellation} again
to replace the operator 
$\mathring H_{\CJ,\mcF, C_{\mcF}(S)}   $
back to
$H_{\CJ,\mcF, C_{\mcF}(S)}$.
We then use the inductive assumption
to replace the operator
$H_{\CJ,\mcF, C_{\mcF}(S)}$
by  the operator $\bar H_{\CJ,\mcF, C_{\mcF}(S)}$
and thus obtain the formula \eqref{e:def-H-bar}
for $\bar{H}_{\CJ,\mcF,S}(\phi)$.
%
\end{proof}
We are now ready to rewrite the moment formula  $M(\psi,\CJ)$
given in \eqref{eq: explicit formula for moment}
and Proposition~\ref{explicit formula}.
Define for $\mcF \in \mathbb{F}$, 
\begin{equ} [e:def-Ppartial-FD2p]
P^{\partial}_{\mcF}(D_{2p})
\eqdef
\left\{ 
\{a,b\} \in L(D_{2p})^{(2)}:\ 
\begin{array}{c}
|\{a,b\} \cap \leavesleft{\mcF}{D_{2p}}| = 1 \textnormal{ or }\\
\exists \textnormal{ distinct }T,S \in \overline{\mcF} \textnormal{ with } a \in L(T), b \in L(S)
\end{array}
\right\}\;.
\end{equ}
We now develop a new formula for moments of the neutral BPHZ model.

For any test function $\psi$, $\CJ \in \mfJ$, and $x \in (\R \times \T^{2})^{\{\logof\}}$, we set
\begin{equ} [e:def-barM-psiJ]
\bar{M}[\psi,\CJ]
\eqdef
\sum_{
\substack{
\mcG \in \mathbb{F},\\ 
\cD \subset \cut \setminus K(\mcG)}
}
\int_{\nodesleft{\mcG}{D_{2p}}}dy\
\CW[\psi,\mcG,\cD,\J](x_{\logof} \sqcup y)
\end{equ}
where
\begin{equs}\label{e:moment-J}
\CW[\psi,\mcG,\cD,\J](z)
\eqdef\ &
\J^{\leavesleft{\mcG}{D_{2p}}^{(2)}}(z)
\cdot
\ke{\kernelsleft{\mcG}{D_{2p}} \setminus \cD}{}(z)\\
&
\quad 
\cdot
\Big(
\prod_{j=1}^{2p}
\psi(z_{\rho_{\sT_{j}}})
\Big)
\cdot
\RKer^{\kernelsleft{\mcG}{D_{2p}} \cap \cD}(z)
\cdot
\powroot{\nodesleft{\mcG}{D_{2p}}}{\sn}{\logof}(z)
\\
& 
\quad 
\cdot 
\bar{H}_{\J,\mcG,\overline{\mcG}} \left[ 
\J^{P^{\partial}_{\mcG}(D_{2p})}
\cdot
\RKer^{K^{\downarrow}(\bar{\mcG}) \cap \cD}
\ke{K^{\downarrow}(\bar{\mcG}) \setminus \cD}{}
\powroot{\tilde{N}(\bar{\mcG})}{\sn}{\logof}
\right](z)\;.
\end{equs}
Here $\bar\mcG=\mathrm{Max}(\mcG)$ as before.
We then have the following statement.
\begin{proposition} \label{prop:mom-formula-z}
Let $\bar{M}[\psi,\CJ]$ be given as above,
and $M[\psi,\CJ]$ be the quantity 
defined in \eqref{eq: explicit formula for moment}.
For any $\e$-assignment $\iota$ one has
\begin{equ}\label{eq: equality of formulae}
\bar{M}[\psi,\CJ(\iota)] = M[\psi,\CJ(\iota)]\;.
\end{equ}
\end{proposition} 
\begin{proof}
Note that the only differences between the LHS and RHS of \eqref{eq: equality of formulae} 
is the quantity $\bar{H}_{\CJ(\iota),\mcG,\overline{\mcG}}$ in place of $H_{\CJ(\iota),\mcG,\overline{\mcG}}$ and the inserted factor $\CJ(\iota)^{P^{\partial}_{\mcG}(D_{2p})}$.
In fact, we show that the corresponding terms 
are equal for each fixed $\mcG$ and $\cD$.

By Lemma~\ref{lem: parity cancellation} and
Lemma~\ref{lem:H-equals-Hbar} it suffices to prove that 
one can insert  the factor
 $\CJ(\iota)^{P^{\partial}_{\mcG}(D_{2p})}$ 
into $\mathring{H}_{\CJ(\iota),\mcG,\overline{\mcG}} [\cdots]$
without changing the moment formula.
The proof follows in the same way as that for the claim \eqref{e:claim-insert-J},
with $D_{2p}$ in place of $S$,
$\mcG$ in place of $\mcF$,
and $\bar\mcG$ in place of $C_{\mcF}(S)$.
For instance, we can assume that $\bar\mcG=\{T_1,\cdots,T_n\}$,
and then show  $\mathscr{Y}^{(0)}_{T_1} \circ \cdots \circ \mathscr{Y}^{(0)}_{T_n}
(\CJ(\iota)^{P^{\partial}_{\mcG}(D_{2p})})=1$ inductively as the proof of  the claim \eqref{e:claim-insert-J}.
\end{proof}

The content of the above Proposition is that while renormalisation of regularity structures only 
``moves''\slash renormalises kernel edges in general, for the sine-Gordon model
 one also needs to renormalise noise-contraction edges 
 (i.e. the $\CJ$'s), but these renormalisations come
 {\it for free} thanks to charge cancellation and parity.
 
 Before ending this subsection we observe a simple fact. Note that 
$M[\psi,\CJ]$ is not multilinear in $\CJ \in \mathfrak{J}$; however, we have:
\begin{lemma}\label{lem:multilinear}
$\bar{M}[\psi,\CJ]$ is multilinear in $\CJ = (J_{e}: e \in L(D_{2p})^{(2)})\in \mathfrak{J}$.
\end{lemma}
\begin{proof}
We claim that fixing $\mcG \in \mathbb{F}$,
$\cD \subset \cut \setminus K(\mcG)$, for every
 $e \in L(D_{2p})^{(2)}$, the factor $J_e$ appears in 
 the term for $(\mcG,\cD)$ in \eqref{e:def-barM-psiJ} once and only once;
this suffices to imply the linearity in $J_e$ since all other other operations (in particular the operator $\bar H$) are linear.

The above claim simply follows from the definitions of various sets, 
but we would like to take this opportunity to help the reader recapitulate the notation of these sets. 
Each $e \in L(D_{2p})^{(2)}$ falls into one of the following cases.

\begin{enumerate}
\item
 Both nodes of $e$ are in $L(D_{2p}) \setminus L(\mcG)$. In this case
$e\in \leavesleft{\mcG}{D_{2p}}^{(2)}$ which is the set appearing in \eqref{e:moment-J}.
\item
 One node of $e$ is in $L(S)$ for some $S\in \bar\mcG$, and the other node of $e$ is not  in $L(S)$ for this $S$. In this case $e\in P^{\partial}_{\mcF}(D_{2p})$ which is the set appearing inside $\bar{H}_{\J,\mcG,\overline{\mcG}} $ 
in \eqref{e:moment-J}.
\item
 There is an $S\in \bar\mcG$ so that both nodes of $e$ are in $L(S)$.
In this case, we again have three ``sub-cases'' if we think of $S$ as playing the role of $D_{2p}$, and $C_{\mcF}(S)$ as playing the role of $\bar\mcG$.
We then get the sets $L_{\mcF}(S)^{(2)}$ and $P^{\partial}_{\mcF}(S)$
from \eqref{e:def-H-bar} in the first and second sub-cases. 
In the third sub-cases there is an $T\in C_{\mcF}(S)$ so that both nodes of $e$ are in $L(T)$ - iterating in this way we exhaust all $e \in L(D_{2p})^{(2)}$.
\end{enumerate}
\vspace{-4ex}
\end{proof}

\subsection{Proof of Theorem~\ref{theorem:mom-bounds}}
\label{sec:proof-mom-bounds}
Fix a smooth test function $\psi$ supported in the unit ball and bounded by $1$. We define, for each $\lambda \in (0,1]$, $\mcG \in \mbbF$, $\cD \subset \cut \setminus K(\mcG)$, and $\CJ \in \mathfrak{J}$
\[
\CW_{\lambda}[\mcG,\cD,\CJ]
\eqdef
\int_{\nodesleft{\mcG}{D_{2p}}}dy\
\CW[\psi^\lambda_{x_{\logof}},\mcG,\cD,\CJ](x_{\logof} \sqcup y)\;.
\] 
Note that both sides do not depend on $x_{\logof}$ by translation invariance. 
We then have by definition of $\bar{M}[\psi,\CJ]$ as in \eqref{e:def-barM-psiJ} that
\begin{equ}[e:def-Q-moments]
\bar{M}_{\lambda}[\CJ] \eqdef 
\bar{M}[\psi^\lambda_{x_{\logof}}, \CJ]
=
\sum_{
\substack{ 
\mcG \in \mbbF\\
\cD \subset \cut \setminus K(\mcG)
}}
\CW_{\lambda}[\mcG,\cD,\CJ]
\;.
\end{equ}

We now start to prove 
the main estimates on the moments of models
claimed in Theorem~\ref{theorem:mom-bounds}.

\begin{proof}[of Theorem~\ref{theorem:mom-bounds}]
Given the moment formula \eqref{e:moment-is-M} in Proposition~\ref{explicit formula}
and 
Proposition~\ref{prop:mom-formula-z},
the first bound in  \eqref{e:moment-bound-theorem} claimed in the theorem will
follow from the bound (uniform in $\CJ \in \mathfrak{J}$ and $\lambda \in (0,1]$)
\begin{equ}\label{eq: the moment functional bound}
\big|
\bar{M}_{\lambda}[\CJ] \big|
\lesssim
\|\CJ\| \lambda^{2p|\tau|_{\s}}\;.
\end{equ}
By the definition of $\bar{M}_{\lambda}[\CJ] $ as in 
 \eqref{e:def-Q-moments},
the claim \eqref{eq: the moment functional bound} follows from combining 
Corollary~\ref{cor:organize-and-resum} and 
Proposition~\ref{prop: main estimate} in the next section. 
(The rest of the paper will be devoted to
proving Corollary~\ref{cor:organize-and-resum} and 
Proposition~\ref{prop: main estimate}.)

For the second bound in  \eqref{e:moment-bound-theorem} 
claimed in the theorem, if we expand the $2p$-th power
we will have joint moments of processes regularised by $\eps$ and by $\bar\e$. The $\e$-assignment defined earlier is introduced to deal with this situation.
 
%

For $0 \le j,k \le p$, we define the $\e$-assignment $\iota_{j,k}$ by setting $\iota_{j,k}(l) = \e$ if $1 \le l \le p-j$ or $p + 1 \le l \le 2p - k$ and we set $\iota_{j,k}(l) = \bar\e$ for all other values of $l$. 
One has
\footnote{
By elementary identity
$
(a-b)^{p} \overline{(a-b)}^{p}
=
\sum_{j,k=0}^{p}
\binom{p}{j}\binom{p}{k}
(-1)^{j+k}
\big(
a^{p-j}b^{j}\bar{a}^{p-k}\bar{b}^{k}
-
a^{p}\bar{a}^{p}
\big)
$}
\begin{equs}
\E \Big|\big(\hat{\Pi}^{\e}_{z} \tau   
-  \hat{\Pi}^{\bar\e}_{z} \tau  
\big)    (\psi^{\lambda}_{z}) \Big|^{2p}
=
\sum_{j,k=0}^{p}
\binom{p}{j}
\binom{p}{k}
(-1)^{j+k}
\Big(
\bar{M}_{\lambda}[\CJ(\iota_{j,k})] 
-\bar{M}_{\lambda}[ \CJ(\iota_{0,0})]
\Big).
\end{equs}
(Note that $\bar{M}_{\lambda}[ \CJ(\iota_{0,0})]$ is just the $\bar{M}_{\lambda}[ \CJ]$ in \eqref{eq: the moment functional bound}, and we are now just comparing the cross terms with it.)

For each fixed $j,k$ we can write $\bar{M}_{\lambda}[\CJ(\iota_{j,k})] -\bar{M}_{\lambda}[ \CJ(\iota_{0,0})]$ into telescoped sum of terms of the form 
$\bar{M}_{\lambda}[\CJ_1] -\bar{M}_{\lambda}[ \CJ_2]$
where $\CJ_1,\CJ_2\in \mathfrak{J}$ only differ by one component $J_e$ for some 
$e \in L(D_{2p})^{(2)}$.
By Lemma~\ref{lem:bound-J}, such a difference of $J_e$'s has  
$\|\cdot\|_{2\beta'+\kappa,m}$ or $\|\cdot\|_{-2\beta'+\kappa,m}$
norm of order $\e^\kappa$. Therefore $\CJ_1-\CJ_2 $
can be viewed as an element in $\mathfrak{J}$,
with $\|\CJ_1-\CJ_2 \| \lesssim \e^\kappa$.
By Lemma~\ref{lem:multilinear} we have $\bar{M}_{\lambda}[\CJ_1] -\bar{M}_{\lambda}[ \CJ_2] = \bar{M}_{\lambda}[\CJ_1- \CJ_2]$ which can then be bounded in the same way as
\eqref{eq: the moment functional bound} with the proportionality constant of order $\e^\kappa$.
\end{proof}
\section{Multiscale expansion and organizing renormalisations}
\label{sec:Multiscale expansion}
The goal of this section is bounding \eqref{e:def-Q-moments}.
Our approach is the following.
(1)
We will expand each term $\CW_\lambda [\CJ,\CG,\cD]$
over ``scale assignments $\mbn$'', see \eqref{e:expand-in-n}.
(2)
Then we will switch the sum over scales $\mbn$ and the sum over $(\CG,\cD)$,
so that for each fixed $\mbn$, we will have a delicate way (using ``intervals'') to group terms in the sum
over $(\CG,\cD)$, see \eqref{e:grouping-FG} and \eqref{e:grouping-and-resum-1}.
It is this organization into groups that really implements cancellations by renormalisations; see \eqref{e:def-I_S} for some intuitive discussion.
(3)
Once we have the terms organized into these groups,
we will switch the sums again, so that  within each fixed group we can sum over scales $\mbn$, see 
\eqref{e:grouping-and-resum-2}
and Proposition~\ref{prop: main estimate}.

\subsection{Multiscale expansion}
To estimate the moment $\bar M$ we start by a multiscale decomposition of all the functions appearing in the integrand.
Here we follow closely \cite[Section~6]{CH},
 with the main difference being that here we work with a moment ``Feynman diagram'' $D_{2p}$ instead of a single stochastic tree as in \cite{CH}.
 This is because the charge and parity cancellations being exploited in the previous section occur at the level of the entire moment ``Feynman diagram''.

In what follows we often abuse notation and view the ordered pairs of $K(D_{2p})$ as unordered two-elements sets. 
We define, for any subset $U \subset N(D_{2p})$
\begin{equ} [e:CE-logof]
 \CE_{\logof}(U) \eqdef \{\{\logof,u\}\}_{u \in U}\;.
 \end{equ}
We then set
\[
\CE 
\eqdef 
\CE_{\logof}(N(D_{2p}))
\sqcup
K(D_{2p})
\sqcup
L(D_{2p})^{(2)}\;.
\]
We define a \emph{global scale assigment} for $D_{2p}$ to be a tuple 
\footnote{Boldface letter $\mbn$ will stand for a scale assignment in contrast with  a letter $\mfn$ representing a tree decoration.}
\[
\mbn = (n_{e})_{e \in \mcE} \in \N^{\mcE}\;.
\]

The multiscale decomposition is implemented by cutoff functions as follows.
We fix $\psi:\R \rightarrow [0,1]$ to be a smooth function supported on $[3/8,1]$ and with the property that $\sum_{n \in \Z}\psi(2^{n}x) = 1$ for $x \not = 0$.
We then define a family of cutoff functions $\{\Psi^{(k)}\}_{k \in \N}$, where each  $\Psi^{(k)}:\R \times \T^{2} \rightarrow [0,1]$ by 
setting $\Psi^{(k)}(0) = \mathbbm{1}\{ k = 0 \}$ and, for $x\neq 0$,
\[
\Psi^{(k)}(x) 
\eqdef 
\begin{cases}
\sum_{n \le 0} \psi(2^{n}|x|) & \textnormal{ if } k=0\\
\psi(2^{k}|x|) & \textnormal{ if } k \not = 0.
\end{cases} 
\]
For any $E \subset \mcE$ and global scale assignment $\mbn = (n_{e})_{e \in \mcE} $ we define $\Psi_{\mbn}^{E} \in \allf$ via
\[
\Psi_{\mbn}^{E}(x) 
\eqdef
\prod_{\{a,b\} \in E} \Psi^{(n_{e})}(x_{a} - x_{b})\;.
\]

With these cutoff functions at hand, we now define single scale slices of  the functions \eqref{e:def-Ker-E} and \eqref{e:def-RKer-D}
that appear in $\CW$,
as well as for the functions
\begin{equ}[e:kerHat]
\KerHat^{\{e\}}
=
\ke{\{e\}}{}
+
\RKer^{\{e\}} \;,
\qquad
\KerHat^{\cD}
\eqdef
\prod_{e \in \cD}
\KerHat^{\{e\}}
 \;.
\end{equ}
For any $\mbn \in \N^{\mcE}$, $E \subset \mcE$, $\cD\subset \cut$ and $N\subset N(D_{2p})$ we set
\begin{equs}[3]
\ke{E}{\mbn} 
&\eqdef
\ke{E}{}
\cdot
\Psi_{\mbn}^{E},
&\enskip
\KerHat_{\mbn}^{\cD}
&\eqdef
\KerHat^{\cD}
\cdot
\Psi_{\mbn}^{\cD},
\\
\RKer_{\mbn}^{\cD}
&\eqdef
\RKer^{\cD}
\cdot
\Psi_{\mbn}^{\cD},
&\enskip
\powroot{N}{\mfn}{\logof,\mbn} 
&\eqdef
\powroot{N}{\mfn}{\logof}
\cdot
\prod_{u \in N}
\Psi_{\mbn}^{\{(\logof,u)\}},
&\enskip
\CJ_{\mbn}^{E}
&\eqdef
\J^{E}\cdot \Psi^{E}_{\mbn} \;.
\end{equs}

In order  to introduce the operators $\bar H^{\mbn}$ which are the $\mbn$ dependent version of the operator $\bar{H}$, we recall the following definition. 

\begin{definition}
$\mbbM \subset \mbbF$ is said to be {\it an interval of forests}
if $\mbbM$ is either empty or there exist two (not necessarily distinct) forests $s(\mbbM), b(\mbbM) \in \mathbb F$ such that
\footnote{It can happen that $s(\mbbM)=\emptyset$. Note that $\mbbM = \{\emptyset\}$ - namely $\mbbM$ just contains the empty forest (rather than being empty) - is allowed.}
\begin{equ}
\mbbM = \{ \mcF \in \mathbb F \,:\, s(\mbbM) \le a \le b(\mbbM)\}\;.
\end{equ}
$\mbbG $ is said to be {\it an interval of cuttings}
if $\mbbG$ is either empty or there exist two cut sets $s(\mbbG), b(\mbbG) \in 2^\cut$ such that
\begin{equ}
\mbbG = \{ \cC \in 2^\cut \,:\, s(\mbbG) \le a \le b(\mbbG)\}\;.
\end{equ}
Here $\le$ stands for the inclusion partial order (which is defined on both $\mathbb F$ and $2^\cut$).
For a non-empty interval $\mathbb{M}$ we also use the notation $\mbbM = \left[ s(\mbbM), b(\mbbM) \right]$. We also write  $\delta(\mathbb{M}) \eqdef b(\mathbb{M}) \setminus s(\mathbb{M})$, and same for $\delta(\mbbG)$.
\end{definition}

We now define operators $\bar{H}_{\CJ,\mbbM,S}^{\mbn}: \allf \rightarrow \mcb{C}_{\tilde{N}(S)^{c}}$, where $\CJ \in \mathfrak{J}$, $\mbbM$ is an interval%
\footnote{In the general discussion here the interval doesn't have to depend on the scale assignment. It will 
however  do so from the next subsection.}
 of forests, $S \in b(\mbbM)$, and $\mbn \in \N^{\mcE}$.
This definition is again recursive and for $\phi \in \allf$ we set, 
for any $x \in (\R_+\times \T^2)^{\tilde{N}(S)^{c}}$, 
\begin{equation*}
\begin{split}
[\bar{H}_{\CJ,\mbbM,S}^{\mbn}(\phi)](x)\ 
\eqdef&\ 
\int_{\tilde{N}_{b(\mbbM)}(S)} dy\
\CJ_{\mbn}^{L_{b(\mbbM)}(S)^{(2)}}(x_{\rho_{S}} \sqcup y) 
\cdot
\ke{\mathring{K}_{b(\mbbM)}(S)}{\mbn}(x_{\rho_{S}} \sqcup y)\\
&\quad
\cdot
\bar{H}^{\mbn}_{\CJ,\mbbM,C_{b(\mbbM)}(S)}
\Big[
\CJ_{\mbn}^{P^{\partial}_{b(\mbbM)}(S)}
\ke{K^{\partial}_{b(\mbbM)}(S)}{\mbn}
\cdot
(\mathscr{Y}_{S,\mbbM}^{\#}\phi)
\Big](x \sqcup y)
\end{split}
\end{equation*}
where  $\bar{H}_{\CJ,\mbbM,\emptyset}^{\mbn}$ is the identity operator and
\begin{equation*}
\mathscr{Y}_{S,\mbbM}^{\#}\phi
\eqdef 
\begin{cases}
- \mathscr{Y}_{S} \phi &\textnormal{ if } S \in s(\mbbM)\\
(\mathrm{Id} - \mathscr{Y}_{S})\phi &\textnormal{ if } S \in \delta(\mbbM).
\end{cases}
\end{equation*}

With all the above $\mbn$ dependent quantities at hand,
we can finally define  the following main object of this section.
For any $\mbn \in \N^{\mcE}$, $\CJ \in \mfJ$, any interval of forests $\mathbb{M}$, and any interval of cuttings $\mathbb{G}$ with $b(\mbbG) \subset \cut \setminus K(b(\mbbM))$, 
we define 
\begin{equs}[e: interval formula]
\CW^{\mbn}_{\lambda} [\J,\mbbM,  & \mathbb{G}]
 \eqdef 
\int_{\nodesleft{b(\mbbM)}{D_{2p}} \sqcup \{\logof\}} 
	\!\!\!\!\!\!\!\!\!\!\!\!\!\!\!\!\!\!\!\!\!\!\!\!\!\!
dy\ 
\delta(y_{\logof})\,
\psi^{2p,\lambda}(y)\,
\powroot{\nodesleft{b(\mbbM)}{D_{2p}}}{\sn}{\logof,\mbn}(y)\,
\CJ_{\mbn}^{\leavesleft{b(\mbbM)}{D_{2p}}^{(2)}}( x_{\rho_{S}} \sqcup y)\\
&
\cdot
\ke{
\kernelsleft{b(\mbbM)}{D_{2p}} \setminus b(\mathbb{G})
}{\mbn}( y)
\,\RKer_{\mbn}^{s(\mbbG) \setminus K^{\downarrow}(b(\mbbM))}(y)
\,\KerHat_{\mbn}^{\delta(\mbbG) \setminus K^{\downarrow}(b(\mbbM))}( y)
\\
\cdot
& 
\bar{H}_{\CJ,\mbbM,\mmax{b(\mbbM)}}^{\mbn} 
\Big[ 
\CJ_{\mbn}^{P^{\partial}_{b(\mbbM)}(D_{2p})}
\RKer_{\mbn}^{s(\mbbG) \cap K^{\downarrow}(\mmax{b(\mbbM)})}
\,\KerHat_{\mbn}^{\delta(\mbbG) \cap K^{\downarrow}(\mmax{b(\mbbM)})} \label{e:multiscale-integrand}
\\
&\qquad\qquad\qquad\qquad\qquad 
\cdot \,\ke{K^{\downarrow}(\mmax{b(\mbbM)}) \setminus b(\mathbb{G})}{\mbn}
\powroot{\tilde{N}(\mmax{b(\mbbM)})}{\sn,\logof}{\mbn}
\Big]( y) \, dy\;.
\end{equs}
Here $\psi^{2p,\lambda} \in \allf$ is defined via $\psi^{2p,\lambda}(z)
\eqdef
\prod_{j=1}^{2p}
\psi_{\lambda,z_{\logof}}(z_{\rho_{j}})$. 

A particularly useful case, see \eqref{e:expand-in-n} below, is $\mbbM = \{\mcF\}$ and 
$ \mathbb{G} = \{\cC\}$ (both sets consist of singletons).
In this case $\CW^{\mbn}_{\lambda} [\CJ,\{\mcF\}, \{\cC\} ]$
is given by the above formula with
\[
b(\mbbM) = \mcF, \quad
b(\mbbG) = \cC, \quad
s(\mbbG)=\cC, \quad
\delta(\mbbG)=\emptyset
\]
 (so that the factors $\KerHat_{\mbn}$ disappear).
In fact, the integrand of $\CW^{\mbn}_{\lambda} [\CJ,\{\mcF\}, \{\cC\} ]$
is simply \eqref{e:moment-J} with each function
replaced by its $\mbn$ dependent version.%
\footnote{To see this, we need to note that for the set associated to $\RKer$, one has that
$s(\mbbG) \setminus K^{\downarrow}(b(\mbbM))
= \cC \setminus K^{\downarrow}(\mcF)$
is the same set as $K(\mcF,D_{2p})\cap \cC$ as in  \eqref{e:moment-J}, using our assumption that $b(\mbbG) \subset \cut \setminus K(b(\mbbM))$.}
The reason that we need to define 
$\CW^{\mbn}_{\lambda} [\J,\mbbM,  \mathbb{G}]$
for general $\mbbM,  \mathbb{G}$ will be clear
in \eqref{e:grouping-FG}.

One then has the following lemma (see also \cite[Lemma~6.3]{CH}),
which gives an multiscale decomposition for $\CW_{\lambda}[\CJ,\mcF, \cC]$,
and for each scale assignment $\mbn$ gives a ``grouping'' of terms according to intervals.
\begin{lemma}\label{lem:reorderSum}
For any $\J \in \mfJ$, $\cC \subset \cut$, and $\mcF \in \mathbb{F}_{\cC}$ (where $\mathbb{F}_{\cC}$ is defined in \eqref{e:def-FavoidC})
one has the absolute convergence
\begin{equ} [e:expand-in-n]
\sum_{\mbn \in \N^{\mcE}}
\CW^{\mbn}_{\lambda}\left[\J,\{\mcF\}, \{\cC\} \right]
=
\CW_{\lambda}[\J,\mcF, \cC]\;.
\end{equ}
On the other hand, for fixed $\mbn \in \N^{\mcE}$, interval of forests $\mathbb{M}$, and any interval of cut sets $\mathbb{G}$ with $b(\mbbG) \subset \cut \setminus K(b(\mbbM))$, one has
\begin{equ} [e:grouping-FG]
\sum_{
\substack{
\mcF \in \mathbb{M}\\
\cC \in \mathbb{G}}
}
\CW_{\lambda}^{\mbn}\left[\CJ,\{\mcF\}, \{\cC\} \right]
=
\CW_{\lambda}^{\mbn}\left[\CJ,\mathbb{M}, \mathbb{G} \right].
\end{equ} 
\end{lemma} 
\begin{proof}
This statement corresponds quite closely to \cite[Lemma~6.3]{CH}. The first identity follows from linearity. 
The second identity is not immediately straightforward to obtain from \eqref{eq: explicit formula for moment} for the reasons described at the end of Section~\ref{sec: renormalisation operators} but it is not so difficult to obtain from \eqref{e:moment-J}.  

Note that the fact that we implemented a single scale slice of our integral corresponding to $\mbn$ plays no role for proving the identity (i.e., a similar identity holds before the multiscale expansion) but we have stated the identity this way since this is the way we will use it. 
We first show that for any fixed choice of $\mbn \in \N^{\CE}$ and interval of forests $\mbbM$ one has 
\begin{equ}\label{eq: interval sum work}
\sum_{\mcF \in \mathbb{M}}
\CW_{\lambda}^{\mbn}\left[\J,\{\mcF\}, \emptyset \right]
=
\CW_{\lambda}^{\mbn}\left[\J,\mbbM, \emptyset \right]\;.
\end{equ}
We prove the above identity via induction on $|\delta(\mbbM)|$. The base case, which occurs when $|\mbbM| = 1$ and $|\delta(\mbbM)| = 0$, is immediate. 

For the inductive step fix $l > 0$, assume the claim has been proven whenever $|\delta(\mbbM)| < l$, and fix $\mbbM$ with $|\delta(\mbbM)| = l$.

Fix $T \in \delta(\mbbM)$, we prove the claim in the case where there exists $\tilde{T} \in b(\mbbM)$ with $T \in C_{b(\mbbM)}(\tilde{T})$. 
The case when there is no such $\tilde{T}$ is easier. 

Let $\mbbM_{1} \eqdef \{ \mcF \in \mbbM: \mcF \not \ni T\}$ and $\mbbM_{2} \eqdef \{ \mcF \sqcup \{T\}: \mcF \in \mbbM_{1} \}$. 
Note that $\mbbM_{1}$ and $\mbbM_{2}$ are intervals that partition $\mbbM$. Therefore, by our inductive hypothesis, 
\begin{equ}\label{eq: interval work 3}
\CW_{\lambda}^{\mbn}[\J,\mbbM, \emptyset ] = \CW_{\lambda}^{\mbn}[\J,\mbbM_{1}, \emptyset ] + \CW_{\lambda}^{\mbn}[\mbbM_{2}, \emptyset ]\;.
\end{equ} 
Clearly one has, for all $S \in b(\mbbM)$ with $T \not \le S$, 
\begin{equ}\label{interval work 3.5}
\bar{H}_{\CJ,\mbbM_{1},S}^{\mbn} = \bar{H}_{\CJ,\mbbM_{2},S}^{\mbn}\;.
\end{equ}
Next we claim that for every $S \in b(\mbbM)$ with $S > T$ one has
\begin{equ}\label{eq: interval work 4}
\bar{H}_{\CJ,\mbbM_{1},S}^{\mbn} + \bar{H}_{\CJ,\mbbM_{2},S}^{\mbn} = \bar{H}_{\CJ,\mbbM,S}^{\mbn}\;. 
\end{equ}
Proving this claim finishes our proof since the combination of \eqref{interval work 3.5} and \eqref{eq: interval work 4} yields \eqref{eq: interval sum work}. 
We prove the claim using an auxilliary induction in 
\[
\mathrm{depth}(\{ S' \in b(\mbbM): S > S' > T\})\;.
\] 
The inductive step for this induction is immediate upon writing out both sides of \eqref{eq: interval work 4} and remembering that $\mathscr{Y}_{S,\mbbM_{1}}^{\#} = \mathscr{Y}_{S,\mbbM_{2}}^{\#} = \mathscr{Y}_{S,\mbbM}^{\#}$. 
What remains is to check base case of this induction which occurs when $S = \tilde{T}$. 

To obtain \eqref{eq: interval work 4} when $S = \tilde{T}$ we first observe that $C_{b(\mbbM)}(\tilde{T}) = C_{b(\mbbM_{2})}(\tilde{T}) =  C_{b(\mbbM_{1})}(\tilde{T})~\sqcup~\{T\}$ and then rewrite, for $i = 1,2$,  $\bar{H}_{\CJ,\mbbM_{i},\tilde{T}}^{\mbn}[\phi]$ as 
\begin{equs}[e: final work for interval sum]
{}&
\int_{\tilde{N}_{b(\mbbM)}(\tilde{T})} dy\
\CJ_{\mbn}^{L_{b(\mbbM)}(\tilde{T})^{(2)}}(y)
\cdot
\ke{\mathring{K}_{b(\mbbM)}(\tilde{T})}{\mbn}(x_{\rho_{\tilde{T}}} \sqcup y)
\int_{\tilde{N}_{b(\mbbM)}(T)}dz\ 
\CJ_{\mbn}^{L_{b(\mbbM)}(T)^{(2)}}(z)\\
& \quad
\cdot
\ke{\mathring{K}_{b(\mbbM)}(\tilde{T})}{\mbn}(w)
\big(
\mathscr{Y}^{(i)}
\circ
\bar{H}^{\mbn}_{C_{b(\mbbM)}(\tilde{T}) \setminus \{T\}}
[
\CJ_{\mbn}^{P^{\partial}_{b(\mbbM)}(\tilde{T})}
\ke{K^{\partial}_{b(\mbbM)}(S)}{\mbn}
\mathscr{Y}_{\tilde{T},\mbbM}^{\#}\phi
]
\big)(w),
\end{equs}
where $w = x_{\rho_{\tilde{T}}} \sqcup y \sqcup z$,  $\mathscr{Y}^{(1)} = \Id$, and $\mathscr{Y}^{(2)} = (-\mathscr{Y}_{T})$. 
This step would not work if we were using the operators $H_{\CJ,\bullet,\bullet}$ instead of $\bar{H}_{\CJ,\bullet,\bullet}$.

The corresponding identity for summing over $\cC \in \mathbb{G}$ is easier to check, one just expands, for each $e \in \delta(\mbbG)$, the factor $\KerHat_{\mbn}^{\{e\}}$ as $\ke{\{e\}
}{\mbn} + \RKer_{\mbn}^{\{e\}}$.  
\end{proof}
\subsection{Organising renormalisations}
Our objective in this subsection is to explain how, for each fixed $\mbn \in \N^{\CE}$, we take advantage of the cancellation between terms $\CW^{\mbn}[\CJ,\{\mcF\},\{\cC\}]$ where $(\cC,\mcF)$ ranges over
\[
\{ (\cC,\mcF): \cC \subset \mfC,\ \mcF \in \mbbF_{\cC}\}\;.
\] 
\subsubsection{Organising negative renormalisations}
We now move towards defining the ``projections onto safe forests'' $\{P^{\mbn}\}_{\mbn \in \N^{\mcE}}$ of \cite{FMRS85} which are used to organize negative renormalisations, namely organizing the sum over $\mcF$ 
These projections tell us how to choose intervals based on the scale assignments $\mbn$.

Fix $S \in \Div$ and $\mcF \in \mathbb{F}$ .
We define the immediate ancestor of $S$ in $\mcF$ by
\[
A_{\mcF}(S)
\eqdef 
\begin{cases}
T & \textnormal{ if }\mathrm{Min}(\{\tilde{T} \in \mcF:\ \tilde{T} > S\}) = \{T\},\\
D_{2p}^{\ast} & \textnormal{ if }\{\tilde{T} \in \mcF:\ \tilde{T} > S\} = \emptyset
\end{cases}
\]
where we view $D_{2p}^{\ast}$ as an undirected multigraph with node set $\allnodes$ and edge set $\CE$. Note that the minimum above consists of a unique tree $T$ because $\CF$ is a forest.

Recalling the notations  in Section~\ref{sec:nodes-and-edges},
we also define the edge set $\CE^{\inte}_{\mcF}(S)$ by
\[
\CE^{\inte}_{\mcF}(S)
\eqdef
K_{\mcF}(S) \sqcup L_{\mcF}(S)^{(2)}
\]
and define the edge set $\CE^{\exte}_{\mcF}(S)$ in the following way
\begin{equs}
\CE^{\inte}(S)
&  \eqdef
K(S)
\sqcup
L(S)^{(2)},
	\quad
\CE^{\inte}(D_{2p}^{\ast})
\eqdef
\CE,
	\\
\CE^{\exte}(S)
&\eqdef
\CE_{\logof}(N(S))
\, \sqcup \,
K^{\downarrow}(S)
\, \sqcup \,
\{e \in L(D_{2p})^{(2)}: |e \cap L(S)| = 1\},
	\\
\CE^{\exte}_{\mcF}(S)
& \eqdef 
\CE^{\inte}(A_{\mcF}(S))
\cap
\CE^{\exte}(S)
\;,
\end{equs}
where $\CE_{\logof}(\cdot)$ is defined in  \eqref{e:CE-logof}.
We use these definitions to define the following internal and external scale information.
Given $\mbn$ we set
\begin{equs}[inner and outer scales def]
\mathrm{int}_{\mcF}^{\mbn}(S) 
&
\eqdef 
\min 
\{ n_{e}:\ 
e \in \CE^{\inte}_{\mcF}(S) 
\}\;,
\\
\mathrm{ext}_{\mcF}^{\mbn}(S)
&\eqdef
\max
\{
n_{e}:\  
e \in \CE^{\exte}_{\mcF}(S)  
\}\;.
\end{equs}
We now introduce the forest projections we will use. For each $\mbn \in \N^{\CE}$, we 
define $P^{\mbn}: \mathbb{F} \rightarrow \mathbb{F}$ by
\begin{equ}\label{def: projection onto safe forests}
P^{\mbn}(\mcF) \eqdef
\left\{ 
S \in \mcF:\  
\inte_{\mcF}^{\mbn}(S) \le \exte_{\mcF}^{\mbn}(S)
\right\}\;.
\end{equ}

\begin{example}
In the example of the dipole (Example~\eqref{ex:dipole}), consider $S$ shown in the gray area, and assume that $S\in \CF$:
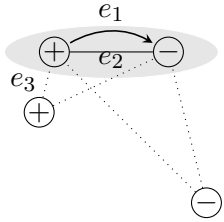
\[
\begin{tikzpicture}
  \filldraw[draw=white,fill=black!10]   (0.75,2)     ellipse (40pt and 10pt);
\node at (0,2) [charge] (1) {$+$};  \node at (1.5,2) [charge] (2) {$-$};
\node at (-0.2,1.2) [charge] (3) {$+$};  
\node at (2,0) [charge] (4) {$-$};
\draw [kernel, bend left = 30] (1) to node[above] {$e_1$} (2);  
\draw  (1) to   (2);  
	\node at (0.75,1.9) {$e_2$};
\draw[dotted] (3) to  node[left] {$e_3$} (1); 
\draw[dotted] (3) to (2); 
\draw[dotted] (4) to (1); 
\draw[dotted] (4) to (2); 
\end{tikzpicture} 
\]
Here $K_{\mcF}(S)=\{e_1\}$ 
and
 $ L_{\mcF}(S)^{(2)}=\{e_2\}$,
 so that $\CE^{\inte}_{\mcF}(S)=\{e_1,e_2\}$.
 All the dotted lines represent some of the edges (if not all)
  in $\CE^{\exte}_{\mcF}(S)$;
 in particular $e_3\in \CE^{\exte}_{\mcF}(S)$. 
 The picture   reflects the actual distances, namely, 
 $2^{-n_{e_1}}=2^{-n_{e_2}}>2^{-n_{e_3}}$.
So 
in this situation the condition $\inte_{\mcF}^{\mbn}(S) \le \exte_{\mcF}^{\mbn}(S)$
in \eqref{def: projection onto safe forests}
is satisfied, 
thus $S$ is ``safe'' (i.e. $S\in P^{\mbn}(\mcF)$). $S$ is called a ``good pair'' in \cite{HaoSG}.
\end{example}
The following  proposition shows
the key property of the maps $P^{\mbn}$ that justifies our use of them as an organizing tool.
\begin{proposition}	\label{Pn-1 is interval}
For every $\mbn \in \N^{\CE}$ and $\mcF \in \mathbb{F}$ one has that $(P^{\mbn})^{-1}(\mcF)$ is an interval of forests.
\end{proposition}
\begin{proof}
One can easily translate the proof of \cite[Proposition 7.3]{CH} and the statements \cite[Lemma~7.1 and Corollary~7.2]{CH} used in the proof\dash in our setting the set of forests $\mathbb{F}_{\pi}$ is replaced by $\mathbb{F}$ and the ambient multigraph $\sT^{\ast}$ is replaced by $D_{2p}^{\ast}$. 
The original proof of this statement can be found in \cite{FMRS85}. 
\end{proof}
\begin{remark}
Note that since one always has $P^{\mbn}(\mcF) \subset \mcF$ it follows that if $(P^{\mbn})^{-1}(\mcF) \not = \emptyset$ then $\mcF$ must be the minimal element $(P^{\mbn})^{-1}(\mcF)$. 
\end{remark}
Our next step is to obtain Corollary~\ref{cor:organize-and-resum} below
which states that we can group terms in the summation on the left hand side of \eqref{e:grouping-and-resum-1}
 into sets of intervals called
 $\mfM^{\mbn}$ and
$ \mfG^{\mbn}(\mbbM)$ where $\mbbM \in  \mfM^{\mbn}$. We define these two sets now.

For each $\mbn \in \N^{\CE}$ and $\cC \subset \mfC$ we define
\begin{equ}\label{collections of intervals of forests}
\mfM^{\mbn}(\cC)
\eqdef 
\Big\{
\mbbM \subset \mbbF:\
\exists \mcS \in \mbbM \textnormal{ such that }
(P^{\mbn}_{\cC})^{-1}[\mcS]
=
\mbbM
\Big\}\;,
\end{equ}
where we have introduced the shorthand $(P^{\mbn}_{\cC})^{-1}(\cdot) \eqdef (P^{\mbn})^{-1}(\cdot) \cap \mbbF_{\cC}$
where the notation $\mbbF_{\cC}$ was defined in \eqref{e:def-FavoidC}. One sees that $\mfM^{\mbn}(\cC)$ is a collection of intervals thanks to Proposition~\ref{Pn-1 is interval} and the fact that the intersection of an interval with $\mbbF_{\cC}$ is again an interval. 

For each $\cC \subset \mfC$ the set $\mfM^{\mbn}(\cC)$ partitions $\{\cC\} \times \mbbF_{\cC}$ into $\{ \{\cC\} \times \mbbM: \mbbM \in \mfM^{\mbn}(\mcC)\}$. 
The negative renormalisation cancellations we take advantage of will all take place as we sum over such a single such set $\{\cC\} \times \mbbM$ as opposed to occurring between different sets $\{\cC\} \times \mbbM$ and $\{\cC\} \times \mbbM'$ with $\mbbM \not = \mbbM'$. 

However we are not done since taking advantage of cancellations for positive renormalisation requires us to group terms with different choices of $\cC \subset \mfC$. In the next subsection we will introduce a strategy for organizing positive renormalisations and show that there is some compatibility in the sense that we can use this organization simultaneously with our above described method for organizing negative renormalisations. 
\subsubsection{Organising positive renormalisations}
Again, we fix a choice of $\mbn \in \N^{\CE}$ for this subsection. 
Given $u,v \in \allnodes$ and $\mcF \in \mathbb{F}$, we define 
\begin{equ}\label{pathfinder scale}
\mbn_{\mcF}(u,v) 
\eqdef
\max
\Big\{ 
\min\{ n_{e}: e \in \CE' \setminus \CE^{\inte}(\mcF) \}:\ 
\CE' \subset \CE \textnormal{ connects }u,v
\Big\}\;.
\end{equ}
Here, a subset $\CE' \subset \CE$ connects $u,v \in \allnodes$ 
simply means that 
 one can find a sequence  $e_{1},\dots,e_{k} \in \CE'$ with $ u \in e_{1}$, $v \in e_{k}$ and $e_{j} \cap e_{j+1} \not = \emptyset$ for $1 \le j \le k-1$. 
Note that the minimum is over $e$ and the maximum is over $\CE'$. 
Also recall that by our convention 
\[
\CE^{\inte}(\mcF)
\eqdef
\bigcup_{S \in \mcF} \CE^{\inte}(S)
\]
Now for each $\mcF \in \mathbb{F}$ we set
\begin{equation}\label{def: cuts to harvest}
\cG^{\mbn}(\mcF) \eqdef
\{ e \in \fullcuts:\ 
\mbn_{\mcF}(\logof,e_{\p}) > \mbn_{\mcF}(e_{\p},e_{\ch})
\}\;.
\end{equation}
For a given choice of negative renormalisations $\mcF \in \mbbF$ we assume that have been made $\cG^{\mbn}(\mcF)$ is the set of edges for which we want to harvest positive renormalisation cancellations. 

Now we describe how one shows that both sets of renormalisations can be organised simultaneously. 
We first set
\[
\mfM^{\mbn} \eqdef \bigcup_{\cC \subset \mfC} \mfM^{\mbn}(\cC)\;,
\] 
namely $\mfM^{\mbn}$ is the set of all groupings of negative renormalisations we see as we vary $\cC \subset \mfC$. 
Then for each $\mbbM \in \mfM^{\mbn}$ we define a set of all cuts that allow $\mbbM$ to appear, namely we set 
\[
\mfC^{\mbn}(\mbbM)
\eqdef
\{ \cC \subset \mfC:\ (P^{\mbn}_{\cC})^{-1}(s(\mbbM)) = \mbbM \}\;,
\]
We now have 
\[
\{ (\cC,\mcF): \cC \subset \mfC,\ \mcF \in \mbbF_{\cC}\}
=
\bigsqcup_{\mbbM \in \mfM^{\mbn}} 
\bigsqcup{\cC \in \mfC^{\mbn}(\mbbM)}
\{ (\cC, \mcF) \}\;.
\]
Again, fixing $\mbbM$ above handles the organisation of negative renormalisations but now we break up the union over $\cC \in \mfC^{\mbn}(\mbbM)$ so that we obtain cancellations for the edges of $\cG^{\mbn}(b(\mbbM))$. 
We then have the following proposition.
%
\begin{proposition}\label{prop: compatibility}
For any $\mbn \in \N^{\CE}$, any $\mbbM \in \mfM^{\mbn}$, 
$e \in \big(\mfC \setminus K(b(\mbbM)) \big)\cap \cG^{\mbn}(b(\mbbM))$, and any $\cC \subset \fullcuts$,
\[
\cC \setminus \{e\} \in \fullcuts^{\mbn}(\mbbM)
\quad\Leftrightarrow\quad
\cC \cup \{e\} \in \fullcuts^{\mbn}(\mbbM)\;.
\]
\end{proposition}
\begin{proof}
This proposition says that $P^{\mbn}$ is \emph{compatible} with the cut rule $\cG^{\mbn}(\cdot)$ in the sense of \cite[Definition~5.6]{CH}.
The proof of this statement then follows the same argument used in the proof of \cite[Proposition~7.4]{CH}. 
\end{proof}
This proposition explains that we can group terms to allow for us to harvest positive and negative renormalisations in a satisfactory manner. 
Namely, for each fixed $\mbbM \in \mfM^{\mbn}$ $\mathfrak{G}^{\mbn}(\mbbM)$ is a partition of $\mfC^{\mbn}(\mfM)$ where
\[
\mathfrak{G}^{\mbn}(\mbbM)
\eqdef
\Big\{
\big[\cS,\cS \sqcup \cG^{\mbn}(b(\mbbM)) \big] \subset 2^{\mfC}:\ 
\cS \subset \mfC \setminus \Big(K(b(\mbbM)) \cup \cG^{\mbn}(b(\mbbM))\Big)
\Big\}\;.
\]
In order to state the re-summation result \eqref{e:grouping-and-resum-2} in Corollary~\ref{cor:organize-and-resum} below,
we define
\[
\mathfrak{R} \eqdef 
\left\{
(\mbbM, \mbbG) 
:\ 
\exists \mbn \in \N^{\CE} 
\textnormal{ such that } 
\mathbb{M} \in \mfM^{\mbn} \textnormal{ and }
\mathbb{G} \in \mathfrak{G}^{\mbn}(\mbbM)
\right\}\;,
\] 
and for each $(\mbbM, \mathbb{G}) \in \mathfrak{R}$ we define
\[
\mcN_{\mbbM,\mathbb{G},\lambda}
\eqdef 
\left\{ 
\mbn \in \N^{\CE}
:\ 
\begin{array}{c}
\mbbM \in \mfM^{\mbn},  \;
\mathbb{G} \in \mathfrak{G}^{\mbn}(\mbbM) 
\textnormal{ and}\\
\forall e \in \CE_{\logof}(\{\rho_{j}\}_{j=1}^{2p}),
\quad
n_{e} \ge \lfloor - \log_{2}(\lambda) \rfloor
\end{array}
\right\}\;.
\]
We restate the result of all of this organisation as the following corollary.
\begin{corollary} \label{cor:organize-and-resum}
For each fixed $\mbn \in \N^{\CE}$ one has
\begin{equ} [e:grouping-and-resum-sets]
\bigsqcup_{\mcF \in \mathbb{F}}
\bigsqcup_{\cC \subset \mfC \setminus K(\mcF)}
\{ (\mcF, \cC)  \}
=
\bigsqcup_{\mbbM \in \mfM^{\mbn}}
\bigsqcup_{\mathbb{G} \in \mathfrak{G}^{\mbn}(\mbbM)}
\bigsqcup_{
\substack{
\mcF \in \mbbM\\
\cC \in \mathbb{G}
}
}
\{ (\mcF , \cC)  \}\;.
\end{equ}
%
%
Furthermore, one has, for any $\J \in \mfJ$, 
\minilab{e:grouping-and-resum}
\begin{equs}
\sum_{
\substack{
\mcF \in \mathbb{F}\\
\cC \subset \mfC \setminus K(\mcF)
}}
\CW_{\lambda}[\J, \mcF,\cC]
=&
\sum_{\mbn \in \N^{\CE}}
\sum_{
\substack{
\mbbM \in \mfM^{\mbn}\\
\mbbG \in \mfG^{\mbn}(\mbbM)
}}
\CW^{\mbn}_{\lambda}[\J, \mbbM,\mbbG]
	\label{e:grouping-and-resum-1}
	\\
=&
\sum_{(\mbbM,\mbbG) \in \mfR}
\sum_{\mbn \in \mcN_{\mbbM,\mbbG,\lambda}}
\CW^{\mbn}_{\lambda}[\J, \mbbM,\mbbG]\;.
	\label{e:grouping-and-resum-2}
\end{equs}
\end{corollary}

\begin{proof}
To prove \eqref{e:grouping-and-resum-1},
we apply \eqref{e:expand-in-n} to each term $\CW_{\lambda}[\J, \mcF,\cC]$,
followed by switching the sum over $\mbn$ and the sum over $(\mcF ,\cC)$, which yields
\[
\sum_{\mbn \in \N^{\CE}}
\sum_{
\substack{
\mcF \in \mathbb{F}\\
\cC \subset \mfC \setminus K(\mcF)
}}
\CW^{\mbn}_{\lambda}[\J, \mcF,\cC] \;.
\]
By \eqref{e:grouping-and-resum-sets} and \eqref{e:grouping-FG} one then obtains the right hand side of  \eqref{e:grouping-and-resum-1}.
The identity  \eqref{e:grouping-and-resum-2} follows by switching the two sums again and 
the definitions of $\mfR$ and $\mcN_{\mbbM,\mathbb{G},\lambda}$.
\end{proof}
We spend the remainder of the paper working to obtain the following estimate.
\begin{proposition}\label{prop: main estimate}
For any $(\mathbb{M},\mathbb{G}) \in \mathfrak{R}$, the following bound holds uniformly in $\lambda \in (0,1]$,  
\begin{equ}\label{eq: main estimate on moment}
\sum_{\mbn \in \mcN_{\mbbM,\mbbG,\lambda}}
\left|
\CW^{\mbn}_{\lambda}[\J,\mbbM,\mbbG]
\right|
\lesssim
\|\J\|
\lambda^{2p|\sT^{\sn\sl}|_{\s}}\;.
\end{equ} 
\end{proposition}
We {\it fix for the remainder of the paper} a choice of $(\mathbb{M},\mathbb{G}) \in \mathfrak{R}$.
We also introduce the shorthands 
\begin{equ} [e:SBDSBD]
\mcS \eqdef s(\mbbM),\;
\mcB \eqdef b(\mbbM),\;
\mcD \eqdef \delta(\mbbM),\;
\cS \eqdef s(\mathbb{G}),\;
\cB \eqdef b(\mathbb{G}),\;
\cD \eqdef \delta(\mathbb{G}).
\end{equ}
\section{Estimating the moment}
	\label{sec: renormalisation bound}
\subsection{Summing over scales inductively}
To give a streamlined argument we find it convenient to factorize the sum over 
$\mbn \in\mcN_{\mbbM,\mbbG,\lambda}$ in a way informed by the nested-ness structure of $\mcB$ (defined in \eqref{e:SBDSBD}) as follows:
By the time it comes to control the contribution of an element $S \in \mcB$ we will have already conditioned on a ``partial'' scale assignment $\mbj$ living in $\N^{\CE^{\inte}_{\mcB}(A_{\mcF}(S))}$.  
We will then sum over partial scale assignments $\mbk \in \N^{\CE^{\inte}_{\mcB}(S)}$ which are ``consistent'' with $\mbj$; here consistency means that one can find $\mbn \in \mcN_{\mbbM,\mbbG,\lambda}$ such that $\mbn$'s restriction to $\CE^{\inte}_{\mcB}(A_{\mcF}(S))$ is given by $\mbj$ and $\mbn$'s restriction to $\CE^{\inte}_{\mcB}(S)$ is given by $\mbk$.

Then treating $\mbk$ as fixed we will sum, for each $T \in C_{\mcB}(S)$, over partial scale assignments $\tilde{\mbj} \in \CE^{\inte}_{\mcB}(T)$ which are consistent with $\mbk$.

\[
\begin{tikzpicture}
\draw   (0,0)     ellipse (60pt and 30pt);
	\node at (1,0.6) {$\mbj$};
	\node at (-1.2,0.6) {$A_{\mcF}(S)$};
\draw   (0,0)     ellipse (30pt and 20pt);
	\node at (0.8,0.2) {$\mbk$};
	\node at (-0.8,0.2) {$S$};
\draw   (0,0)     ellipse (15pt and 10pt);
	\node at (0.3,0) {$\tilde \mbj$};
	\node at (-0.3,0) {$T$};
\end{tikzpicture}
\]

A convenient observation (made rigorous as \cite[Lemma~8.3]{CH}) is that the constraints of the set $\mcN_{\mbbM,\mbbG,\lambda}$ are ``Markovian'' in the sense that knowing the condition that $\tilde{\mbj}$ be consistent with $\mbj \sqcup \mbk$ is equivalent to knowing $\tilde{\mbj}$ is consistent with $\mbk$.

To that end, we will decompose the single sum over global scale assignments $\mcN_{\mbbM,\mathbb{G},\lambda}$ into {\it a family} of sums, which facilitate summing the scales internal to a single $T \in \mcB$ conditioned on the values of relevant external scales (these two sets of quantities being dependent through the requirement that $T \in \mcS$ or $T \in \mcD$). 
In what follows, for any $\tilde{\CE} \subset \CE$ and $\mbj \in \N^{\CE}$ we write $\mbj \restr {\CE'} \in \N^{\CE'}$ for the restriction of $\mbj$ to the edges in $\CE'$. 

The outermost set of edges is given by
\[
\CE^{\inte}_{\mcB}(D_{2p}^{\ast})
\eqdef
K^{\downarrow}(\overline{\mcB})
 \, \sqcup \,
	 \kernelsleft{\mcB}{D_{2p}}
\, \sqcup \,
L_{\mcB}(D_{2p})^{(2)}
\,\sqcup\,
P^{\partial}_{\mcB}(D_{2p})
\,\sqcup  \,
\CE_{\logof}\;,
\]
and its corresponding set of scale assignments is given by
\begin{equ} [e:def-partialN-Blam]
\d\mcN_{\mcB,\lambda} \eqdef
\{
\mbk \in \N^{\CE^{\inte}_{\mcB}(D_{2p}^{\ast})}:\ 
\exists \mbj \in \mcN_{\mbbM,\mathbb{G},\lambda} \textnormal{ with } \mbj \restr {\CE^{\inte}_{\mcB}(D_{2p}^{\ast})} = \mbk
\}\;.
\end{equ}
For any $S \in \mcB$, $\CE' \subset \CE$ with $\CE' \supset \CE^{\exte}_{\mcB}(S)$, and $\mbj \in \N^{\CE'}$ we define 
\begin{equ}[e:defNring]
\mathring{\mcN}_{S}(\mbj) 
\eqdef
\left\{ \mbk \in  \N^{\CE^{\inte}_{\mcB}(S)}:\ 
\exists \tilde{\mbj} \in \mcN_{\mbbM,\mathbb{G},\lambda} 
\textnormal{ with }
\tilde{\mbj} \restr {\CE'} = \mbj 
\textnormal{ and }
\tilde{\mbj}
\restr {\CE^{\inte}_{\mcB}(S)} 
=
\mbk
\right\}\;.
\end{equ}
Note that for every $\mbk \in \mathring{\mcN}_{S}(\mbj)$ one then has 
\begin{equs}[2]
\inte^{\mbk}_{\mcB}(S) &\le \exte^{\mbj}_{\mcB}(S) &\qquad&\textnormal{ if }S \in \mcS,\\
\inte^{\mbk}_{\mcB}(S) &> \exte^{\mbj}_{\mcB}(S) &\qquad&\textnormal{ if }S \in \mcD.
\end{equs}

We now inductively define a family of operators $\hat{H}^{\mbj}_{\CJ,S}:\allf \rightarrow \allf$, where $S \in \mcB$ and $\mbj \in \N^{\CE'}$ with $\CE' \supset \CE^{\exte}_{\mcB}(S)$ by setting
\begin{equation}\label{genvert def}
\begin{split}
[\hat{H}^{\mbj}_{\CJ,S}\phi](x)\ 
\eqdef\  
&
\sum_{ \mbk \in \mathring{\mcN}_{S}(\mbj)}
\int_{\tilde{N}_{\mcB}(S)}\back dy \
\J^{L_{\mcB}(S)^{(2)}}_{\mbk}(y) 
\,
\ke{\mathring{K}_{\mcB}(S)}{\mbk}(y \sqcup x_{\rho_{S}})\\
&\quad
\cdot
\hat{H}^{\mbk}_{\CJ,C_{\mcB}(S)}
\left[
\CJ^{P^{\partial}_{\mcF}(S)}_{\mbk}
\ke{K^{\partial}_{\mcB}(S)}{\mbk}
\,
[\mathscr{Y}_{S,\mbbM}^{\#}\phi]
\right](x_{\tilde{N}(S)^c} \sqcup y)\;,
\end{split}
\end{equation}
with the base case of the induction given by setting $\hat{H}^{\mbj}_{\emptyset}$ to be the identity operator.

We define, for each $\mbj \in \partial \mcN_{\mcB,\lambda}$, a function $\hat{\CW}_{\lambda}^{\mbj}[\J,\mbbM,\mbbG] \in \mcb{C}_{\tilde{N}(\overline{\mcB})^{c}}$ via
\begin{equation}\label{def of inductively summed kernels}
\begin{split}
\hat{\CW}_{\lambda}^{\mbj}[\J,\mbbM,\mbbG]
\eqdef&\; 
\psi^{2p,\lambda}
\cdot
\ke{
\kernelsleft{\mcB}{\sT}\setminus \cB}{\mbj}
\J_{\mbj}^{\leavesleft{\mcB}{\sT}^{(2)}}
\KerTilde_{\mbj}^{\cB \setminus K^{\downarrow}(\mcB)}
\powroot{\nodesleft{\mcB}{\sT}}{\sn,\logof}{\mbj}\\
&
\quad \cdot
\hat{H}_{\CJ,\bar{\mcB}}^{\mbj} \left[ 
\CJ_{\mbj}^{P^{\partial}_{\mcB}(D_{2p})}
\KerTilde_{\mbj}^{\cB \cap K^{\downarrow}(S)} 
\ke{K^{\downarrow}(S) \setminus \cB}{\mbj}
\powroot{\tilde{N}(S)}{\sn,\logof}{\mbj}
\right]\;,
\end{split}
\end{equation}
where for any $\cC \subset \cut$ and $\mbj \in \partial \mcN_{\mcB,\lambda}$ we set 
\[
\KerTilde_{\mbj}^{\cC} 
\eqdef 
\RKer_{0,\mbj}^{\cC \cap \cS}
\cdot
\KerHat_{\mbj}^{\cC \cap \cD}\;.
\]
Applying the analogs of \cite[Lemma~8.3 and Corollary~8.4]{CH} in our setting yields the following lemma.
\begin{lemma}\label{inductively summed integrand}
For any $\lambda \in (0,1]$,
\begin{equation}\label{inductive sum identity}
\sum_{\mbj \in \partial \mcN_{\mcB,\lambda}}
\int_{\nodesleft{\mcB}{D_{2p}} \sqcup \{\logof\}} 
  \!\!\!\!\!\!\!\!\!
\delta(y_{\logof})
\hat{\CW}_{\lambda}^{\mbj}[\J,\mbbM,\mbbG](y) \,dy
= 
\sum_{\mbn \in \mcN_{\mbbM,\mathbb{G},\lambda}}
\CW_{\lambda}^{\mbn}[\J,\mbbM,\mbbG]\;.
\end{equation}
\end{lemma}
%
\subsection{Kernel and renormalisation estimates}
As in \cite{CH} our estimates on the renormalisation of nested divergence structures will require control of local supremums of derivatives of the various kernels appearing in our integrand, this control is needed in order to implement the generalised Taylor remainder estimate of \cite[Prop.~A.1]{Regularity}. 

We recall from \cite[Definition~8.7]{CH} the definition of the seminorms $\| \cdot \|_{\CF,\mbj}(x)$ where $\CF \subset \CB$ with $\mathrm{depth}(\CF) \le 1$, $\mbj \in \N^{\CE'}$ for some $\CE' \supset \CE_{\mcB}(\CF)$, and $x \in \tilde{N}^{c}(\CF)$.

We also recall notation for various domain constraints used in \cite{CH}. 
For $z,w \in \R \times \T^{2}$ and $t \in \R$ we write $\link{z}{w}{t}$ for the condition
\begin{equ}\label{defining annular region}
C^{-1} 2^{-t} \le |z-w| \le C2^{-t}
\end{equ}
write $\ulink{z}{w}{t}$ for the condition
\begin{equ}\label{defining circular region}
|z-w| \le C2^{-t}\;.
\end{equ}
In both \eqref{defining annular region} and \eqref{defining circular region} one chooses a
fixed value $C > 0$ (not dependent on $t$). 
\begin{remark}
Note that from line to line the constant $C$ implicit in the notations \eqref{defining annular region} and \eqref{defining circular region} may change but remains suppressed from the notation.   

In the end, all of these constants influence the overall constant of proportionality appearing in \eqref{eq: main estimate on moment}, see \cite[Remark~8.6]{CH}. 
\end{remark}
\hao{Is it the case that this whole page is only for defining $\|\cdot \|_{\mcF,\mbj}$,
and won't be used again anywhere else? If so can we just refer to your paper for definition of $\|\cdot \|_{\mcF,\mbj}$, and give properties such as $\|fg \|_{\mcF,\mbj}\lesssim....$,
and make a remark that it's defined slightly simpler (no fictitious). Or there's something new because of $\CJ$? A reader may feel frustrated by many notations followed by two lemmas without proof...}\ajay{Removed long definition and put in references to \cite{CH}}
We adopt, for any $S \in \mcB$, the notation 
$\CE^{\partial}_{\mcB}(S) \eqdef K^{\partial}_{\mcB}(S) \sqcup P^{\partial}_{\mcB}(S)$.

Key lemmas that we will use from \cite{CH} are \cite[Lemmas~8.9,8.11]{CH}, which carry over to the present setting immediately. 
One can also easily translate the proof of \cite[Lemma~8.10]{CH} to prove the following lemma. 
\begin{lemma}\label{lem: kernel bound and support}
Let $S \in \mcB$ and $ \mcF \eqdef C_{\mcB}(S)$. 
Then uniform in $\mbk \in \N^{\CE'}$ with $\CE' \supset \CE^{\inte}_{\mcB}(S)$, $x \in (\R^d)^{\tilde{N}(\mcF)^{c}}$ one has
\begin{equs}[e:boundKernel]
\Big\|
\CJ^{P^{\partial}_{\mcB}(S)}_{\mbk}
\ke{K^{\partial}_{\mcB}(S)}{\mbk}
\Big\|_{\mcF,\mbk}(x)
& \lesssim
\|\J\|_{P^{\partial}_{\mcB}(S)}\\
&
\quad
\cdot
\Big(
\prod_{e \in K^{\partial}_{\mcB}(S)}
2^{2k_{e}}
\Big)
\Big(
\prod_{e \in P^{\partial}_{\mcB}(S)}
2^{ - 2 \,\bar{\beta} \,\sign(e)\, k_{e}}
\Big)
\;,
\end{equs}
where $\sign(e)$ was defined in \eqref{e:def-sign}.
Furthermore, the left hand side vanishes unless, for all $T \in \mcF$, $v \in N(T)$, and 
$e = \{u,v\} \in \CE^{\partial}_{T}(S)$ one has $\link{x_{u}}{x_{\rho_{T}}}{k_{e}}$.
\end{lemma}

The key renormalisation estimate driving our bounds is given below. 
\begin{lemma}\label{lem: renorm factors}
Let $S \in \mcB$ and $ \mcF \eqdef C_{\mcB}(S)$. 
Then, uniform in $\mbj \in \N^{\CE'}$ with $\CE' \supset \CE^{\exte}_{\mcB}(S)$, $\mbk \in \mathring{\mathcal{N}}_{S}(\mbj)$, and $x \in (\R \times \T^{2})^{\tilde{N}(\mcF)^{c}}$ satisfying the constraints 
\begin{equs}\label{domainconstraint}
\ulink{x_{u}}{x_{v}}{k_{e}} 
&\,\textnormal{for all $e  = \{u,v\} \in \CE^{\inte}_{\mcB}(S)$,}\\
\ulink{x_{u}}{x_{\rho_{T}}}{k_{e}}
&\,\textnormal{for all $T \in \mcF$, $v \in N(T)$, $e = \{u,v\} \in \CE^{\partial}_{\mcB}(S)$,}
\end{equs}
one has the bound
\begin{equ}\label{eq: renorm factor}
\big\|\mathscr{Y}^{\#}_{S,\mbbM} \phi 
\big\|_{\mcF, \mbk}( x)
\lesssim 
\|\phi\|_{S,\mbj}(x_{\tilde{N}(S)^{c}})
\cdot
2^{\bar{\omega}^{\#}(S)[\exte_{\mcB}^{\mbj}(S) - \inte_{\mcB}^{\mbk}(S)]
}\;,
\end{equ}
where  
\begin{equ}\label{def: barred omega}
\bar{\omega}^{\#}(S)
\eqdef
\begin{cases}
\lfloor - |S^{0}|_{\SG} \rfloor  
& 
\textnormal{ if } 
S \in \mcS,\\
\lfloor - |S^{0}|_{\SG} \rfloor   + 1 
& \textnormal{ if }S \in \mcD\;.
\end{cases}
\end{equ}
\end{lemma}
\begin{proof}
This is essentially the same as \cite[Lemma~8.11]{CH} and the proof there carries over to our setting quite easily.
Note that since our $S\in \mcD\cup \mcS$ is neutral,
we can rewrite $|S^0|_\s$ therein as $|S^0|_{\SG}$.
\end{proof}
One also has the following analog of \cite[Lemma~8.12]{CH}. 
\begin{lemma}\label{lem: genvertbd}
Let $\mcF \subset \mcB$ with $\mathrm{depth}(\mcF) \le 1$. 
Then, uniform in $\CJ \in \mfJ$, $x \in (\R^d)^{\tilde{N}(\mcF)^{c}}$, $\mbj \in \N^{\CE'}$ with $\CE' \supset \CE^{\exte}_{\mcB}(\mcF)$, and $\phi \in \allf$, one has the bound
\begin{equ}\label{genvertbd}
\left|
\hat{H}_{\CJ,\mcF}^{\mbj}
[\phi](x)
\right|
\lesssim 
\Big(
\prod_{S \in \mcF}
2^{ - |S^{0}|_{\SG} \exte_{\mcB}^{\mbj}(S)}
\|\CJ\|_{L(S)}
\Big)
\|\phi\|_{\mcF, \mbj}(x)\;.
\end{equ} 
\end{lemma}
\begin{proof}
This lemma can be proved by adapting the proof of \cite[Lemma~8.12]{CH};
we will only point to the alterations present for our current setting. 
As before, an induction argument shows that it suffices to give the proof to \eqref{genvertbd} in the case where $\mcF = \{S\}$ for some $S \in \Div$, with \eqref{genvertbd}  assumed to be true when the forest $\mcF $ on the LHS of \eqref{genvertbd} is $C_{\mcB}(S)$. 

Our concern is to control the corresponding sum over $\mbk \in \mathring{\mcN}_{S}(\mbj)$ appearing in the definition of $\hat{H}^{\mbj}_{\CJ,S}$ (see \eqref{genvert def}). 

We use Theorem~\ref{multicluster 1} to control the integral over $\tilde{N}_{\mcB}(S)$ in  \eqref{genvert def}. 
We work out our argument with $x \in (\R \times \T^{2})^{\tilde{N}(S)^{c}}$ fixed but our estimates will be uniform in $x$.

Since we have already integrated out the nodes of $\tilde{N}(T)$ for $T \in C_{\mcB}(S)$,
 the multigraph underlying our application of Theorem~\ref{multicluster 1} will be given by a quotient of the multigraph $\CE^{\inte}_{\mcB}(S)$ where, for each $T \in C_{\mcB}(S)$, one performs a contraction and identifies the collection of vertices $N(T)$ as a single equivalence class of points which we identify with $\rho_{T}$.  
 
More precisely,
we define $\hat{\mfq}_{S}:\tilde{N}(S) \rightarrow N_{\mcB}(S)$ via setting $\hat{\mfq}_{S}(u) = \rho_{T}$ if $u \in N(T)$ for $T \in C_{\mcB}(S)$ and $\hat{\mfq}_{S}(u) = u$ otherwise. 

Then our set of vertices is given by $\CV \eqdef N_{\mcB}(S)$ with $\CV_{0} \eqdef \tilde N_{\mcB}(S)$ (so $\rho_{S}$ serves the role of the pinned vertex) and our multigraph $\go{G}$ is given by 
\begin{equation}\label{multigraph for inductive bound}
\go{G}
\eqdef\ 
\mathring{K}_{\mcB}(S) 
\sqcup 
L_{\mcB}(S)^{(2)}
\sqcup
\left\{ 
\{u,\rho_{T}\}:\ 
T \in C_{\mcB}(S),\ 
u \in e_{\ch}[K^{\partial}_{T}(S)]
\right\}
\sqcup
\tilde{P}^{\partial}_{\mcB}(S),
\end{equation}
where
\begin{equ}\label{eq: covariance edges after quotient}
\tilde{P}^{\partial}_{\mcB}(S) 
\eqdef
\bigsqcup_{\{u,v\} \in P^{\partial}_{\mcB}(S)}
\{\hat{\mfq}_{S}(u),\hat{\mfq}_{S}(v)\}
\end{equ}
We remind the reader that the RHS 
of \eqref{eq: covariance edges after quotient} 
is treated as a multi-set of edges where duplicates are distinguishable. Let $\mfq_{S}:\CE^{\inte}_{\mcB}(S) \rightarrow \go{G}$ be the bijection that maps $\mathring{K}_{\mcB}(S)$ onto $\mathring{K}_{\mcB}(S) $ and $L_{\mcB}(S)^{(2)}$ onto $L_{\mcB}(S)^{(2)}$. For each $T \in C_{\mcB}(S)$ and $e \in K^{\partial}_{T}(S)$ we set $\mfq_{S}(e) =\{\rho_{T},e_{\ch}\}$. Finally for each $\{u,v\} \in P^{\partial}_{\mcB}(S)$ to we set $\mfq_{S}(\{u,v\})$ to be the element $\{\hat{\mfq}_{S}(u),\hat{\mfq}_{S}(v)\}$ appearing in the RHS of \eqref{eq: covariance edges after quotient}.

The map $\mfq_{S}$ induces a bijection between $\N^{\CE^{\inte}_{\mcB}(S)}$ and $\N^{\go{G}}$, we abuse notation and identify these two sets in what follows below, in particular we define $\mcN_{\go{G}} \subset \N^{\go{G}}$ by setting $\mcN_{\go{G}} = \mathring{\mcN}_{S}(\mbj))$.

For each $\mbk \in \mcN_{\go{G}}$ we define a function $F^{\mbk} \in \mcb{C}_{\CV}$ by setting 
\[
F^{\mbk}(y)
\eqdef
\J_{\mbk}^{L_{\mcB}(S)^{(2)}}(y) 
\ke{\mathring{K}_{\mcB}(S)}{\mbk}(y)
\hat{H}^{\mbk}_{\CJ,C_{\mcB}(S)}
\left[
\J_{\mbk}^{P^{\partial}_{\mcB}(S)}
\ke{K^{\partial}_{\mcB}(S)}{\mbk}
\,
[\mathscr{Y}_{S,\mbbM}^{\#}\phi]
\right](x \sqcup 
y)\;.
\]
So we then have
\[
\hat{H}^{\mbj}_{\CJ,S}[\phi](x \sqcup y_{\rho_{S}})
=
\sum_{ \mbk \in \mathring{\mcN}_{S}(\mbj)}
\int_{\tilde{N}_{\mcB}(S)}\back dy\ F^{\mbk}(y)\;.
\] 
Thus the proof of the lemma falls into the scope of Theorem~\ref{multicluster 1}.
We define a total homogeneity $\varsigma$ on the coalescence  trees of $\widehat{\CU}_{\CV}$
(see Section~\ref{Sec: Multiclustering} for definitions of the set $\widehat{\CU}_{\CV}$ of coalescence trees over $\CV$ and total homogeneity) by setting 
\footnote{One can alternatively write the last sum as over $e\in S^{(2)}\setminus \bigsqcup_{T\in C_{\mcB}(S)} T^{(2)}$. Since every $T \in C_{\mcB}(S)$ is neutral,
 we can also use $|T^{0} |_{\SG} $ in place of $|T^{0} |_{\s}$ in \eqref{total homogeneity for inductive bound}.}
\begin{equs}[total homogeneity for inductive bound]
\varsigma
\eqdef
&
- \bar{\omega}^{\#}(S)\delta^{\uparrow}[\CV]
- \sum_{T \in C_{\mcB}(S)}
	|T^{0}|_{\s}
	\delta^{\uparrow}[\rho_{T}]
\\
&
+ 2\sum_{
e \in K_{\mcB}(S)
}
\delta^{\uparrow}[\{\hat{\mfq}_{S}(e_{\p}),\hat{\mfq}_{S}(e_{\ch})\} ]
-
2
\bar{\beta}
\!\!\!\!\!\!
\sum_{e \in L_{\mcB}(S)^{(2)} \sqcup P^{\partial}_{\mcB}(S)}
\!\!\!\!\!\!
\sign(e) \delta^{\uparrow}[\{\hat{\mfq}_{S}(e_{<}),\hat{\mfq}_{S}(e_{>})\}]\;.
\end{equs}
It is straightforward to show that $F$ is bounded by $\varsigma$ in the sense of 
Definition~\ref{def:boundedBy} if one uses as input \cite[Lemmas~8.9]{CH}, Lemmas~\ref{lem: kernel bound and support} and \ref{lem: renorm factors}, and the fact that we have assumed the claimed bound for the operator $\hat{H}^{\mbk}_{\CJ,C_{\mcB}(S)}$. 

We split the proof that $\varsigma$ is subdivergence free on $\CV$ for the set of scales $\mcN_{\go{G}}$ in Lemma~\ref{lem: proof of subdivfree}.
The claim then follows by applying Theorem~\ref{multicluster 1}. Here, $r = \exte_{\mcB}^{\mbj}(S)$ and is straightforward to check that $\varsigma$ is of order $\alpha = - \lfloor - |S^{0}|_{\s} \rfloor - |S^{0}|_{\s}$ if $S \in \mcS$ and of order $\alpha = - \lfloor - |S^{0}|_{\s} \rfloor - |S^{0}|_{\s} - 1$ if $S \in \mcD$.
\end{proof}
\begin{lemma}\label{lem: proof of subdivfree}
We claim that total homogeneity on \eqref{total homogeneity for inductive bound} is subdivergence free on $\CV$ for the set of scales $\mcN_{\go{G}}$. 
\end{lemma}
\begin{proof}
Fix an arbitrary $\bo{T} \in \mcU_{\mcV}$.
For any $a \in \mathring{\bo{T}}$ we define 
\begin{equ}\label{def: rexpanded set}
\go{N}(a)
\eqdef 
\bo{L}_{a}
\sqcup 
\Big( 
\bigsqcup_{
\substack{
T \in C_{\mcB}(S)\\
\rho_{T} \in \bo{L}_{a}
}
}
\tilde{N}(T)
\Big)\;.
\end{equ}
We also set 
\begin{equ}\label{def: family of re-expanded sets}
\mcQ \eqdef 
\left\{
\go{N}(a):\ 
a \in \mathring{\bo{T}} \setminus \{\rho_{\bo{T}}\}
\right\}\;.
\end{equ}
Observe that  $N(T) \not \in \mcQ$ for any $T \in C_{\mcB}(S)$ and that any node-set $M \in \mcQ$ is edge connected by $\CE(S)$. 

We define a map $\tilde{\varsigma}: 2^{N(S)} \setminus \{\emptyset\} \rightarrow \R$ as follows: for $M \subset N(S)$ we set
\footnote{One might actually want to consider $ M \cap L(S)$ rather than $M$; however, Lemma~\ref{lem:neg-trees} states that they are the same sets. We also note that 
the map $\tilde{\varsigma}$
is defined on set of vertices, while 
$\varsigma$ is on internal nodes of coalescent trees.}
\begin{equ}\label{eq: summed homogeneity for genvertbd}
\tilde{\varsigma}(M) 
\eqdef
2
|
\{
e \in K(S):\ 
e \subset M\}
|
-
|M|_{\SG}
-
(|M| - 1) |\s|\;.
\end{equ}
Observe that if $M \subset N(S)$ is $K(S)$-connected then $\tilde{\varsigma}(M) = - |T^{0}|_{\SG}$ where $T$ is the subtree of $S$ formed by the nodes of $M$. 
Furthermore, it is straightforward to check that for any $a \in \mathring{\bo{T}} \setminus \{ \rho_{\bo{T}}\}$, 
\begin{equ}\label{eq: summed homogeneity equality - genvertbd}
\sum_{
b \in \bo{T}_{\ge a}
}
\varsigma_{\bo{T}}(b) 
-
(|\bo{L}_{a}| - 1)|\s| 
= 
\tilde{\varsigma}(\go{N}(a))\;.
\end{equ} 
Therefore in view of the definition of the subdivergence free condition \eqref{eq: divergence free condition},
the proof of the lemma is completed  once we show that $\tilde{\varsigma}(M) < 0$ for all $M \in \mcQ$. 

First we assume we are in the case that $M \in \mcQ$ is not $K(S)$-connected\dash that is there is some $k \ge 0$ such that one has subtrees $T_1,\dots,T_k$ and $\tilde{M} \subset L(S)$ such that the sets $\{N(T_j)\}_{j=1}^{k} \sqcup \{\tilde{M}\}$ are a partition of $M$, each $N(T_j)$ has cardinality at least two, $\tilde{M}$ is the union of all $K(S)$-components of $M$ of cardinality one, and one has $k + |\tilde{M}| \ge 2$. 
Then using the bound 
$|\tilde{M} |_{\SG} \ge  |\tilde{M} |_{\s}$ 
(see \eqref{e:SGnorm-formula} ) and
\footnote{When we write $ |\tilde{M}|_{\SG}$ and $|\tilde{M}|_{\s}$ we are viewing $\tilde{M}$ as a set of nodes where each node is assigned a type in $\mfL_{-}$, as in Section~\ref{sec:charge}.}
\[
|\{e \in K(S):\ e \subset M\}|
-
|M \setminus \tilde{M} |_\s
-
(|M \setminus \tilde{M}| - k)|\s|
= -
\sum_{j=1}^{k} |T_j^{0}|_{\s}
\]
one has
\begin{equs}
\tilde{\varsigma}(M)
&  \le
-
\sum_{j=1}^{k} |T_{j}^{0}|_{\s}
-
|\tilde{M}|_{\s}
-
(k + |\tilde{M}| - 1 )|\s|
\\
&  \le
(k + |\tilde{M}|)\bar{\beta} - (k + |\tilde{M}| - 1 )|\s| < 0\;,
\end{equs}
where in the penultimate inequality we used Lemma~\ref{lem: SG trees are nice}
to bound $|T_{j}^{0}|_{\s} > -\bar\beta$
(the inequality is not necessarily strict since $k$ could be zero)
and in the last inequality we used $\bar\beta<2$ and $k + |\tilde{M}| \ge 2$. 

The remaining case is where $M = N(T)$ for some proper subtree $T$ of $S$. 
We finish the proof by observing that by the definition of $\mcN_{\go{G}}$ one cannot have $|T^{0}|_{\SG} < 0$ since then one would have to have $T \in \mcB \setminus \mcS$ which would mean $T \in C_{\mcB}(S)$.
\end{proof}

We now turn to estimates for positive renormalisation.
To make formulas more readable we introduce the notation $\exp_{2} [ t ]
\eqdef
2^{t}$ for $t \in \R$. 
Also, for $u,v \in \tilde{N}(\mcB)^{c}$ we define $\mbj_{\mcB}(u,v)$ as in \eqref{pathfinder scale}. 
\begin{lemma}\label{lem: cut edges estimate}
Let $e \in \mathscr{B}$ and $c > 0$. 
Uniform in $\mbj \in \partial \mcN_{\mcB,\lambda}$ satisfying  
\begin{equ}\label{triangle inequality for positive renormalised}
|\mbj_{\mcB}(e_{\ch},e_{\p}) - j_{e}|,\ |\mbj_{\mcB}(e_{\p},\logof) - j_{\{e_{p},\logof\}}| < c\;,
\end{equ} 
any multi-index $k$ supported on $\{e_{\ch},e_{\p}\}$, and $x = (x_{\logof},x_{e_{\ch}},x_{\e_{\p}})$ such that $\link{x_{e_{\ch}}}{x_{\logof}}{j_{\{\logof,e_{\ch}\}}}$ and $\link{x_{e_{\p}}}{x_{\logof}}{j_{\{\logof,e_{\p}\}}}$, one has the estimate
\begin{equ}
\Big|
D^{k}
\KerTilde^{\{e\}}_{\mbj}(x)
\Big|
\lesssim
2^{
\eta(j_{e},j_{\{\logof,e_{\p}\}},j_{\{\logof,e_{\ch}\}},k_{e_{\ch}},k_{e_{\p}},e)
}\;
\end{equ}
where $\eta(j_{e},j_{\{\logof,e_{\p}\}},j_{\{\logof,e_{\ch}\}},k_{e_{\ch}},k_{e_{\p}},e)$ is given by
\begin{equ}
\begin{cases}
\gamma(e)(j_{e} - j_{\{\logof,e_{\p}\}}) + (2 + |k_{e_{\ch}}|_{\s})j_{e}
+|k_{e_{\p}}|_{s}j_{\{\logof,e_{\p}\}}
&
\textnormal{if }e \in \cD,\\
-(\gamma(e) - 1 + |k_{e_{\p}}|_{\s})j_{\{\logof,e_{\p}\}}
+
(|k_{e_{\ch}}|_{\s} + 2 + \gamma(e)-1)j_{\{\logof,e_{\ch}\}}
&
\textnormal{if }e \in \cS\;.
\end{cases}
\end{equ}
Furthermore, the LHS vanishes unless the condition $\link{x_{e_{\ch}}}{x_{e_{\p}}}{j_{e}}$ holds. 
\end{lemma}

\begin{lemma}\label{lem: top genvert bound}
Let $c > 0$, and $k \in (\N^{d})^{\allnodes}$ be supported on $N^{\downarrow}(S) \sqcup \{ \rho_{S} \}$. 
One has, uniform in $\mbj \in \partial \mcN_{\mcB,\lambda}$ satisfying \eqref{triangle inequality for positive renormalised} for every $e \in \mfC \setminus K(\mcB)$, $\CJ \in \mfJ$, and $x \in \R^{\tilde{N}(S)^{c}}$ satisfying  
\begin{equs}
\link{x_{\rho_{S}}}{x_{\logof}}{j_{\{\logof,\rho_{S}\}}}\;,  \quad
\link{x_{e_{\ch}}}{x_{\logof}}{j_{\{\logof,e_{\ch}\}}}
\quad 
\forall
e \in K^{\downarrow}(S)\;,
\end{equs}
the bound, 
\begin{equation}\label{e: toplevel genvert bound}
\begin{split}
&
\left|
D^{k}
\hat{H}_{\CJ,\overline{\mcB}}^{\mbj} \left[ 
\KerTilde_{\mbj}^{\cB \cap K^{\downarrow}(\mcB)} 
\ke{K^{\downarrow}(\mcB) \setminus \cB}{\mbj}
\powroot{\tilde{N}(\mcB)}{\sn,\mbj}{\logof}
\J^{P^{\partial}_{\mcB}(D_{2p})}_{\mbj}
\right](x)
\right|\\
\lesssim&\ 
\|\CJ\|_{P^{\partial}_{\mcB}(D_{2p})}
\exp_{2} 
\Bigg[
\omega(S) \exte_{\mcB}^{\mbj}(S) + 
j_{(\rho_{S},\logof)} |\sn(\tilde{N}(S))|_{\s}
+
\sum_{e \in K^{\downarrow}(S) \setminus \cB} (2 + |k_{e_{\ch}}|_{\s}) 
j_{e}
\\
&
\enskip
\qquad
+
\Big(
\sum_{e \in K^{\downarrow}(S) \cap \cS}
-(\gamma(e) - 1)j_{\{\logof,\rho_{S}\}} + (2 + |k_{e_{\ch}}|_{\s} + \gamma(e) - 1)j_{\{\logof,e_{\ch}\}}
\Big)\\
&
\enskip
\qquad
+
\Big(
\sum_{e \in K^{\downarrow}(S) \cap \cD} \gamma(e)(j_{e} - j_{(\logof,\rho_S)}) + (2 + |k_{e_{\ch}}|_{\s}) 
j_{e} \Big)
\Bigg]
\end{split}
\end{equation}
Furthermore, there exists a combinatorial constant $C > 0$ such that the LHS vanishes unless $\link{x_{\rho_{S}}}{x_{e_{\ch}}}{j_{e}}$ for all $S \in \overline{\mcB}$, $e \in K^{\downarrow}(S)$, and $|j_{\{\logof,u\}} - j_{(\logof,\rho_S)}| \le C$ for each $u \in \tilde{N}(S)$.
\end{lemma}
\subsection{Combining estimates for the proof of Proposition~\ref{prop: main estimate}}\label{sec: moment graph bounds}
In this section we complete the proof of Proposition~\ref{prop: main estimate} by using the bounds of the previous section and deploying Theorem~\ref{thm: both bounds}. 

We start by quotienting the multigraph $\CE^{\inte}(D_{2p}^{\ast})$ by contracting every $T \in \mcB$ into its root $\rho_{T}$.
More precisely, we define $\hat{\mfq}:N^{\ast} \rightarrow \tilde{N}(\overline{\mcB})^{c}$ via setting $\hat{\mfq}(u) \eqdef \rho_{T}$ if there exists $T \in \overline{\mcB}$ with $u \in N(T)$ and setting $\hat{\mfq}(u) \eqdef u$ otherwise.
We then obtain a multigraph $\go{G}$ on the vertex set $\CV \eqdef \tilde{N}(\overline{\mcB})^{c}$. 
Its set of edges is given by
\begin{equation}\label{multigraph for single tree}
\go{G}
\eqdef\ 
K(\mcB,D_{2p}) 
\sqcup 
\Big\{ 
\{e_{\ch},\rho_{T}\}:\ 
T \in \overline{\mcB}, e \in K^{\downarrow}(T)
\Big\}
\sqcup
\tilde{P}^{\partial}(\mcB,D_{2p}),
\end{equation}
where  the notation $K(\mcB,D_{2p})$
was defined in  \eqref{e:def-NFA-KFA-LFA} and
\[
\tilde{P}^{\partial}(\mcB,D_{2p})
\eqdef
\bigsqcup_{\{u,v\} \in P^{\partial}_\mcB(D_{2p})}
\{\hat{\mfq}(u), \hat{\mfq}(v)\}
\]
and $P^{\partial}_\mcB(D_{2p})$ was defined in \eqref{e:def-Ppartial-FD2p}.

We write $\mfq:  \N^{\CE^{\inte}(\sT^{\ast})} \rightarrow \N^{\go{G}}$ for the natural bijection between these two sets 
and then set $\mcN_{\go{G},\lambda} \eqdef \mfq(\partial \mcN_{\mcB,\lambda})$
where $\partial \mcN_{\mcB,\lambda}$ is defined in \eqref{e:def-partialN-Blam}. 
We also define $\mcN_{\mathbf{G}} \eqdef \bigcup_{\lambda > 0} \mcN_{\mathbf{G},\lambda}$

From now on we switch viewpoints and change indexing sets, writing 
\[
(\hat{\CW}_{\lambda}^{\mbj}[\J,\mbbM,\mbbG])_{\mbj \in \partial \mcN_{\mcB,\lambda}}
\textnormal{ with } 
(\hat{\CW}_{\lambda}^{\mbn})_{\mbn \in \mcN_{\go{G}}}.
\] 
In particular, for $\lambda \in (0,1]$ and $\mbn \in \mcN_{\go{G}}$, one sets $\hat{\CW}_{\lambda}^{\mbn}[\J,\mbbM,\mbbG] \eqdef  \hat{\CW}_{\lambda}^{\mfq^{-1}(\mbn)}[\J,\mbbM,\mbbG]$ if $\mbn \in \mcN_{\go{G},\lambda}$ and $\hat{\CW}_{\lambda}^{\mbn}[\J,\mbbM,\mbbG] \eqdef 0$ otherwise.  

We define $\CV_{0} = \CV \setminus \{\logof\}$, viewing $\logof$ as the pinned vertex. 
We also set $\CV_{\ast} \eqdef \{\rho_{1},\dots,\rho_{2p},\logof\}$.

We define 
\footnote{The notations $\varsigma$ and $\tilde\varsigma$ used in this section here
are different from those in the previous section.}
a total homogeneity
 $\varsigma$ on the trees of $\widehat{\CU}_{\CV}$ (with $\CV$ as defined immediately before \eqref{multigraph for single tree})
  by setting, 
\begin{equs}[e:total-homogeneity]
\varsigma
\eqdef & 
- \!\!\!\!\! \sum_{u \in N(D_{2p})} \!\!\!\!
	|\sn(u)|_{\s} 
	\delta^{\uparrow}[ \{ \hat{\mfq}(u), \logof\} ]
- \sum_{T \in \overline{\mcB}}
	|T^{0}|_{\s}
	\delta^{\uparrow}[\rho_{T}]
\\
&
+ 
2
\sum_{
\substack{ 
e \in \kernelsleft{\mcB}{D_{2p}}\sqcup K^{\downarrow}(\overline{\mcB})\\
e \not \in \enS
}} \!\!\!\!\!\!\!\!
\delta^{\uparrow}[\{\hat{\mfq}(e_{\p}),\hat{\mfq}(e_{\ch})\}]\\
&
-2\bar\beta   \!\!\!\!\!\!\!\!
\sum_{
	\substack{ e \in \\ L(\mcB,D_{2p})^{(2)} 
	\sqcup P^{\partial}_{\mcB}(D_{2p})}}   \!\!\!\!\!\!\!\!
\sign(e) \delta^{\uparrow}[\{\hat{\mfq}(e_{<}),\hat{\mfq}(e_{>})\}]\\
&
+
\sum_{e \in \cD} 
\gamma(e)
\left(
\delta^{\uparrow}[\{\hat{\mfq}(e_{\p}),\hat{\mfq}(e_{\ch})\}]
-\delta^{\uparrow} 
[\{\logof,\hat{\mfq}(e_{\p})\}]
\right)\\[1.5ex]
&
+
\sum_{e \in \enS} 
\Big[
(\gamma(e)-1)
\left(
\delta^{\uparrow}[\{e_{\ch},\logof\}]
-
\delta^{\uparrow} 
[
\{\logof,\hat{\mfq}(e_{\p})\}]
\right)
+
2
\delta^{\uparrow}
[
\{e_{\ch},\logof\}]
\Big].
\end{equs} 
Recall that $\gamma(e)$ is defined in Definition~\ref{def:cut-sets}, $P^\partial_{\mcB}(D_{2p})$ is defined in \eqref{e:def-Ppartial-FD2p},
the cut sets $\cD,\cS$ are defined in \eqref{e:SBDSBD}.
Note that $|\sn(u)|_{\s}\in\{0,1\}$ by Lemma~\ref{lem:neg-trees}.

The estimates of the Section~\ref{sec: renormalisation bound} give us the following proposition.
\begin{lemma} 
Let $\varsigma$ be as in \eqref{e:total-homogeneity}. 
One has that the family $(\hat{\CW}_{\lambda}^{\mbn})_{\mbn \in \mcN_{\go{G}}}$ is bounded by $\varsigma$ with 
\[
\|\hat{\CW}_{\lambda}\|_{\varsigma,\mcN_{\go{G}}}
\lesssim
\lambda^{-2p|\s|}
\|\CJ\|\;,
\]
uniform in $\lambda \in (0,1]$ and $\CJ \in \mathfrak{J}$.
\end{lemma}
\begin{proof}
This domain condition \eqref{domainconstraint} and supremum bound \eqref{supremum bound} are straightforward consequences of the combination of Lemma~\ref{lem: kernel bound and support}, \ref{lem: cut edges estimate}, and \ref{lem: top genvert bound}.
We note that the factor $\lambda^{-2p|\s|}$ comes from the presence of $\psi^{2p,\lambda}$. 
\end{proof}
The two lemmas below are analogs of Lemmas~9.2 and 9.3 of \cite{CH}. 
\begin{remark}
Instead of giving full proofs of the two lemmas below we describe the pre-processing needed to apply the arguments of \cite[Lemmas~9.2 and 9.3]{CH}. 
There are cosmetic differences in the setting but the arguments used in the proofs of \cite[Lemmas~9.2 and 9.3]{CH}  involving $K(\sT)$-connected subsets $A$ of $N(\sT)$ clearly apply equally well to $K(D_{2p})$-connected subsets $A$ of $N(D_{2p})$. 
In our setting the set $A$ will be sitting entirely in one of our $2p$ copies of $\sT$, but the presence of the other  $2p-1$  copies will make no difference to the argument.
\end{remark}
  
\begin{lemma}
The total homogeneity $\varsigma$ defined in \eqref{e:total-homogeneity} is of order $- 2p(|\sT^{\sn\sl} |_{\s} + |\s|)$ and is subdivergence free on $\CV$ for the set of scales $\mcN_{\go{G}}$. 
\end{lemma}
\begin{proof}
We quickly check the statement regarding the order of $\varsigma$. 
It is straightforward to check that for any $\go{S} \in \widehat{\CU}_{\CV}$ one has
\begin{equs}
{}&
\sum_{a \in \mathring{\go{S}}} \varsigma_{\go{T}}(a) - (|\mvert| - 1)|\s|\\
{}=&
-|L(\mcB,D_{2p})|_{\SG}
+
\sum_{j=1}^{2p}
\Big(
\begin{array}{c}
-\sum_{u \in N(T^{(j)})}
|\sn(u)|_{\s}
+
2|K(\mcB,T^{(j)})|\\
-
\sum_{
S \in \overline{\mcB} \cap \Div_{j}
}|S|^{0}_{\s}
-
|N_{\mcF}(\sT_{j})| \cdot |\s|
\end{array}
\Big)\;.
\end{equs}
Now for fixed $j \in [2p]$, one can check that the bracketed quantity is given by $-|\sT^{\sn\sl}|_{\s} + |L(\mcB,\sT_{j})|_{\s} - |\s|$ 
(here the $|L(\mcB,\sT_{j})|_{\s}$ term appears because we are missing those noises.)

The claim follows by observing that 
\[
|L(\mcB,D_{2p})|_{\SG} 
= |L(\mcB,D_{2p})|_{\s} 
= 
\sum_{j=1}^{2p} |L(\mcB,\sT_{j})|_{\s}\;.\] 
The second equality above is trivial. 
For the first one needs to show that $\mcb{q}(L(\mcB,D_{2p})) = 0$ but this is an immediate consequence of the fact that 
\[
\mcb{q}(L(D_{2p})) = \mcb{q}\big( L(\overline{\mcB}) \big) = 0
\textnormal{ and }L(D_{2p}) = L(\mcB,D_{2p}) \sqcup L(\overline{\mcB})\;.
\]
We move on to proving the subdivergence free condition. 
We define a map $\tilde{\varsigma}: 2^{\allnodes} \rightarrow \R$ as follows, for $M \subset \allnodes$ we set
\begin{equs}[eq: subdiv free big graph homogeneity]
\tilde{\varsigma}(M) 
\eqdef
& 
-(|M| - 1) |\s|
+2|\{e \in K(D_{2p}) \setminus \enS:\ e \subset M\}|
- |M|_{\SG}\\
&
-\mathbbm{1}\{\logof \in M\}
\left[
|\sn(M \setminus \{\logof\})|_{\s}
+
\sum_{
e \in \cD}
\gamma(e) 
\mathbbm{1}
\left\{ 
\begin{array}{c}
e_{\ch} \not \in M\\
e_{\p} \in M
\end{array} 
\right\}
\right]
\\
&+\mathbbm{1}\{\logof \in M\}
\sum_{e \in \enS}
\left[
2+
(\gamma(e) - 1) \mathbbm{1}\{e_{\p} \not \in M\}
\right]
\mathbbm{1}\{ e_{\ch} \in M\}
\end{equs}
We fix, for the remainder of this proof, $\bo{T} \in \CU_{\CV}$. 
We define
\begin{equ}\label{def: rexpanded set2}
\go{N}(a)
\eqdef 
\bo{L}_{a}
\sqcup 
\Big( 
\bigsqcup_{
\substack{
T \in \overline{\mcB}\\
\rho_{T} \in \bo{L}_{a}
}
}
\tilde{N}(T)
\Big),
\end{equ}
and also set $
\mcQ \eqdef 
\left\{
\go{N}(a):\ 
a \in \mathring{\bo{T}} \setminus \{\rho_{\bo{T}}\}
\right\}$.
We claim that for any $a \in \mathring{\bo{T}} \setminus \{\rho_{\bo{T}}\}$, 
\[
\Big(
\sum_{
b \in \bo{T}_{\ge a}
}
\varsigma_{\go{T}}(b)
\Big) 
- 
(|\bo{L}_{a}| -1) |\s|
\le
\tilde{\varsigma}(\go{N}(a))\;.
\]
The above equality is not hard to check once one observes that for $M \in \mcQ$
\begin{enumerate} 
\item $e \in \enS$ and $e_{\p}, \logof \in M$ together imply $e_{\ch} \in M$.
\item $e \in \cD$ and $e_{\ch},e_{\p} \in M$ together imply $\logof \in M$.
\end{enumerate}
To prove the lemma it suffices to show that for all $M \in \mcQ$ one has $\tilde{\varsigma}(M) < 0$.

When $M \in \mcQ$ satisfies $M \not \ni \logof$ one can show $\tilde{\varsigma}(M) < 0 $ by copying the argument of Lemma~\ref{lem: proof of subdivfree} nearly verbatim.

We now turn to the case of $M \in \mcQ$ with $M \ni \logof$. We will show in fact that $\hat{\varsigma}(M) < 0$ 
where $\hat{\varsigma}$ is defined as in $\tilde{\varsigma}$ but with $|M |_{\SG}$ replaced $|M |_{\s}$. 
Now for for such $M$, writing $\{M_{j}\}_{j=1}^{n}$ for the $K(D_{2p})$ connected components of $M \setminus \{ \logof \}$ it is straightforward to see that
\[
\hat{\varsigma}(M) 
= 
\sum_{j = 1}^{n} \hat{\varsigma}(M_{j} \sqcup \{ \logof\})
\]
so it suffices to prove $\hat{\varsigma}(\widehat{M} \sqcup \{\logof\}) < 0$ for all $\widehat{M} \in \widehat{\mcQ}$ where we have defined
\[ 
\widehat{\mcQ}
\eqdef
\left\{ 
\widehat{M} \subset N(D_{2p}):\ 
\begin{array}{c}
\widehat{M} \textnormal{ is } K(D_{2p})-\textnormal{connected and }
\exists\ M \in \mcQ,\ u \in \CV_{0}\\
\textnormal{ such that }
M \ni \logof
\textnormal{ and }
\widehat{M} = \sT_{\ge}(u) \cap M
\end{array} 
\right\}\;.
\]
The rest of the proof continues in same way as it does following the relevant point in \cite[Lemma~9.2]{CH}.
\end{proof}
\begin{lemma}
Let $\varsigma$ be defined in \eqref{e:total-homogeneity}. 
Then for every $\go{T} \in \CU_{\mvert}$ and $u \in \mathring{\go{T}}$ with both $u \le (\mvert_{\ast})^{\uparrow}$ and $\go{T}_{\not\ge u} \not = \emptyset$, the inequality \eqref{second cond of HQ} holds.
\end{lemma} 
\begin{proof}
We start by defining a map $\tilde{\varsigma}: 2^{\allnodes} \rightarrow \R$ as follows.
 For $M \subset \allnodes$ we set
\begin{equs} [def:large-scale-homo]
\tilde{\varsigma}(M)
&
\eqdef
2
|\{e \in K(D_{2p}) \setminus \enS:\ e \cap M \not = \emptyset\}|
-
|M |_{\s}
+
\sum_{
e \in \cD}
\mathbbm{1}
\left\{
\begin{array}{c}
e_{\ch} \in M\\
e_{\p} \not \in M
\end{array}
\right\}\gamma(e)\\
&
+
\Big[
\sum_{
e \in \enS
}
2
\mathbbm{1}
\left\{ 
e_{\ch} \in M
\right\} 
- 
(\gamma(e) - 1)  
\mathbbm{1}
\left\{ 
\begin{array}{c}
e_{\p} \in M\\
e_{\ch} \not \in M
\end{array}
\right\} 
\Big]\\
&
-
|\sn(M)|_{\s}
-
|M| \, |\s|.
\end{equs}
Fix for the remainder of the proof $\bo{T} \in \CU_{\CV}$.
We claim that for $a \in \mathring{\bo{T}}$ with $a \le \{\logof,\mainroot\}^{\uparrow}$ one has, writing $M \eqdef \allnodes \setminus \go{N}(a)$ (where $\go{N}(a)$ is defined as in \eqref{def: rexpanded set2}), 
\[
\sum_{
b \in \go{T}_{\not \ge a}
}
\varsigma_{\go{T}}(b) 
-
|\CV \setminus \bo{L}_{a}| \, |\s|
\ge
\tilde{\varsigma}(M)
\]
The claim is justified via two observations. First note that $e \in \fullcuts$, $e_{\c} \in \allnodes \setminus \go{N}(a)$ and $e_{\p} \not \in \allnodes \setminus \go{N}(a)$ together imply $e \in \cD$.

The second observation is that one has the bound
\begin{equ}  \label{eq: large scale decay factors from noises}
-2\bar{\beta}
\sum_{
\substack{
e \in L(D_{2p})^{(2)}\\
e \cap M \not = \emptyset}}
\sign(e)
\ge
-|M|_{\s}\;.
\end{equ}
To prove \eqref{eq: large scale decay factors from noises} we just check that
\begin{equs}
\sum_{
\substack{
e \in L(D_{2p})^{(2)}\\
e \cap M \not = \emptyset}}
\sign(e)
&
=
\sum_{
e \in L(D_{2p})^{(2)}}
\sign(e)
-
\sum_{
\substack{
e \in L(D_{2p})^{(2)}\\
e \subset L(D_{2p}) \setminus M}}
\sign(e)\\
=&
-\frac{1}{2\beta}
\left[
|L(D_{2p})|_{\s}
-
\big(
|L(D_{2p}) \setminus M|_{\s}
-
\mcb{q}(L(D_{2p}) \setminus M)^{2}
\big)
\right]
\end{equs}
and since 
$|L(D_{2p})|_{\s} - |L(D_{2p}) \setminus M|_{\s}
 = |L(D_{2p}) \cap M|_{\s}=|M|_{\s}$ 
 we are done.

We now define a collection of node sets $\mcZ \subset 2^{N^{\ast} \setminus \CV_{\ast}}$ by setting
\[
\mcZ
\eqdef
\left\{ 
M \subset \allnodes \setminus \CV_{\ast}:\ 
\begin{array}{c}
\exists a \in \mathring{\bo{T}} \textnormal{ with } a \le \logof \textnormal{ and } a \le \mainroot
\textnormal{ such that }\\
M\textnormal{ is an }K(D_{2p})\textnormal{-connected }
\textnormal{component of } N^{\ast} \setminus \go{N}(a)
\end{array}
\right\}\;.
\]
The lemma will be proved if we show that, for every $M \in \mcZ$, $\tilde{\varsigma}(M) > 0$. 
Here we used that if $M \subset \allnodes \setminus \CV_{\ast}$ decomposes into $K(D_{2p})$-components $\{M_{j}\}_{j=1}^{n}$ then $\tilde{\varsigma}(M) = \sum_{j=1}^{n} \tilde{\varsigma}(M_{j})$. 
The remainder of the proof of this lemma is exactly the same as the what follows from the corresponding point in \cite[Lemma~9.3]{CH}\dash in the final step one can use the inequality \eqref{eq: SG trees are nice} as a replacement for the invocation of \cite[Definition~2.27 or Assumption~2.23]{CH}.
\end{proof}

\appendix
\section{Additional lemmas}

\begin{lemma}\label{lem: actual parity cancellation for ell}
For any $S \in \Div$ and any  decoration $\mfn$ on $N(S)$ with $|\mfn(N(S))|_{\s} = 1$, 
 one has 
\[
\ell^{\rho}_{\overline \BPHZ}[S^{\mfn\mfl}] = 0\;.
\]
\end{lemma}
\begin{proof}
Note that one must have $\mfn = \delta_{u,j}$ for some $u \in N(S)$ and $j \in \{1,2\}$. We fix this $u$ and $j$.

For this proof we use the explicit formulae given by \cite[Lemma~4.7 and Lemma~4.14]{CH} for the LHS above. 
The context for \cite[Lemma~4.7 and Lemma~4.14]{CH} is for the BPHZ model and is written in terms of cumulants instead of moments but it is straightforward to perform a resummation to go from cumulants to moments and modify the definition of forests to only allow for neutral divergent subtrees. 

Thus, if we define $\mathbb{F}[S]$ to be the collection of all $\mathcal{F} \in \mathbb{F}$ with $\overline{\mcF} = \{S\}$ then one has
\begin{equs}
\ell^{\rho}_{\overline \BPHZ}[S^{\mfn\mfl}]
=&
\mathbbm{1}\{ u \not = \rho_{S}\}
\sum_{\mcF \in \mathbb{F}[S]}
\int_{{\nrmod[\mcF,S]}} \back dy \
\CJ_{\underline{\rho}}^{L_{\mcF}(S)^{(2)}}(y\sqcup x_{\rho_{S}})\,  
\ke{\mathring{K}_{\mcF}(S)}{0}(y \sqcup x_{\rho_{S}})\\
{}
& \qquad
\cdot
\genvert{\mcF, C_{\mcF}(S)}
\left[X_{u,j}
\ke{K^{\partial}_{\mcF}(S)}{}
\right](x_{\tilde{N}(S)^c} \sqcup
y)\;,
\end{equs}
where $X_{u,j} \in \allf$ is given by $X_{u,j}(z) = z_{u}^{(j)} - z_{\rho_{S}}^{(j)}$.

Clearly we can assume that $u \not = \rho_{S}$. 
By Lemma~\ref{lem: parity cancellation} one can replace $\genvert{\mcF, C_{\mcF}(S)}$ with $\mathring{H}_{\J,\mcF,C_{\mcF}(S)}$, then the fact that the above integral (which in fact does not depend on $x_{\rho_{S}}$) vanishes follows by the same parity argument as that used at the end of Lemma~\ref{lem: parity cancellation}. 
\end{proof}
\section{Multiscale Clustering}\label{Sec: Multiclustering}
We briefly describe some of the key ideas and definitions
in the multiscale analysis for the many-variable space-time integrals that appear in our moment formulas. 
A more detailed explanation can be found in \cite[Appendix A]{CH}. 

Our integrals will be in the form 
\[
\int_{\CV_{0}} dx\ \mathrm{Int}(x_{\CV})
\]
where $\CV = (v_{0},\dots,v_{p})$, $p \ge 1$, is an abstract set of vertices, each of which is associated with a position variable $x_{v_{i}} \in \R \times \T^{2}$. 
One of the variables, $v_{0}$, is treated as pinned to a fixed value while the others $\CV_{0} \eqdef \CV \setminus \{v_{0}\}$ are integrated.    

Here $\mathrm{Int}(x_{\CV})$ is often a complicated function but suppose one can write 
\[
\mathrm{Int}(x_{\CV})
=
\prod_{i \in I}F_{i}(x_{\CV_{i}})\;,
\]
where $\CV_{i} \subset \CV$. 
A key step in our method will be obtaining a family of good bounds on each $F_{i}$ in terms of 
\[
(|x_{u} - x_{v}|_{\s}: \{u,v\} \in \CV_{i}^{(2)})\;,
\]
each such bound will hold in a region defined by the \emph{relative} sizes of the distances above and will be of power-law type \dash this will be elaborated on later in our explanation.  

We also remark that we will always be in the situation where $\mathrm{Int}$ has compact support in $(\R \times \T^{2})^{\CV_{0}}$ for fixed $x_{v_{0}}$ so integrability at $\infty$ is not a problem.

Our multiscale analysis will proceed as follows. 
First, we construct a multigraph\footnote{For us a multigraph is a multi-set of elements of $\CV^{(2)}$ where duplicates are distinguishable} via setting $\go{G} = \sqcup_{i \in I} \CV_{i}^{(2)}$.  
Then, for some some combinatorial\footnote{Depending on $|\CV|$} constant $C > 1$, we will obtain\footnote{Often via some partition of unity} a family of functions $(\mathrm{Int}^{\mbn}(x_{\CV}): \mbn = (n_{e})_{e \in \mcN_{\go{G}}} )$ for some set of scales $\mcN_{\go{G}} \subset \mcN^{\go{G}}$ such that
\begin{itemize}
\item $\mathrm{Int}^{\mbn}(x_{\CV}) = 0$ unless for every $e \in \go{G}$ one has 
\begin{equ}\label{annular constraint}
C^{-1}2^{-n_{e}} < |x_{u} - x_{v}|_{\s} \le C 2^{-n_{e}}
\end{equ}
\item For every $x_{\CV} \in (\R \times \T^{2})^{\CV}$ with no co-inciding points, that is $x_{v} \not = x_{v'}$ for all distinct $v,v' \in \CV$, one has 
\[
\sum_{\mbn \in \mcN_{\go{G}}} 
\mathrm{Int}^{\mbn}(x_{\CV}) = \mathrm{Int}(x_{\CV})\;.
\]
\end{itemize}
Our objective will then be to estimate the sum
\begin{equ}\label{eq: bound one slice at a time}
\sum_{\mbn \in \mcN_{\go{G}}}
\left|
\int_{\CV_{0}} dx \,
\mathrm{Int}^{\mbn}(x_{\CV})
\right|
\end{equ}
For each $\mbn \in \mcN_{\go{G}}$ we will ``brutally'' estimate the summand above by 
\begin{equ}\label{eq: single slice bound}
\Big( \sup_{x_{\CV_{0}}} |\mathrm{Int}^{\mbn}(x_{\CV})| \Big)
\times
\mathrm{Vol}\Big( \Big\{x_{\CV_{0}} \in (\R \times \T^{2})^{\CV_{0}}: \mathrm{Int}^{\mbn}(x_{\CV}) \not = 0 \Big\} \Big)\;.
\end{equ}
For the second factor we will choose a spanning tree of $\go{G}$ and then estimate the volume by iteratively integrating the positions of vertices of the trees from the leaves inward to the root (which is the pinned variable $x_{v_{0}}$). This choice of spanning tree will be encoded via what we call a \emph{labeled coalesence tree} on $\CV$

\begin{definition}\label{def: coalescence tree}
A \emph{coalescence tree} $\go{T}$ on a vertex set $\CV$ is a rooted tree with at least three nodes with the following structures and properties:
\begin{itemize}
\item The set of leaves of $\go{T}$ is identified with the set $\CV$.
\item Writing $\mathring{\go{T}}$ for the set of internal nodes (i.e. nodes that are not leaves) and $\rho_{\go{T}}$ for the root of $\go{T}$,  we require that every $u \in \mathring{\go{T}} \setminus \{\rho_{\go{T}}\}$ has degree at least $3$ and that $\rho_{\go{T}}$ has degree at least $2$. 
\end{itemize}
We also equip the set of nodes of $\go{T}$ with a poset structure induced by the tree structure to with the root $\rho_{\go{T}}$ as the unique minimal element.
\end{definition}
We write $\widehat{\CU}_{\CV}$ the collection of all coalescence trees on the vertex set $\CV$. 

\begin{definition}\label{def: labels for coalescence trees}
Given a vertex set $\CV$ and $\go{T} \in \widehat{\CU}_{\CV}$, we set $\go{\mathrm{Lab}}_{\go{T}}$ to be the collection of all maps $\go{s}: \mathring{T} \rightarrow \N$ with the property that $u < v \Rightarrow \go{s}(u) < \go{s}(v)$.
The pair $(\go{T},\go{s})$ is then called a \emph{labeled coalescence tree} and we denote the set of labeled coalescence trees by $\widehat{\CU}_{\CV} \ltimes \go{\mathrm{Lab}}_{\bullet}$. 
\end{definition}

{\it Mapping scale assignments to labeled coalesence trees}
For any fixed $\mbn \in \N^{\go{G}}$ and $r \in \N$ we define the sub-multigraph $\go{G}^{\mbn}_{r} \eqdef \{e \in \go{G}:\  n_{e} \ge r\}$ of $\go{G}$ and also define $\mvert^{\mbn}_{r} \subset 2^{\mvert}$ to be the collection of vertex sets of the connected components of $\go{G}^{\mbn}_{r}$. We consider singletons as connected components so that, for every 
$r$, $\CV^{\mbn}_{r}$ is a partition of $\CV$.

The sequence $(\mvert^{\mbn}_{r})_{r \in \N}$ determines a labelled coalescence tree $(\go{T},\go{s})$ via the following procedure. 
The set of nodes for $\go{T}$ is given by 
${\go{T}} = \bigcup_{r=0}^{\infty} \mvert^{\mbn}_{r}$. Since elements
of ${\go{T}}$ are themselves subsets of $\CV$, they are partially ordered by inclusion.
Given two distinct nodes $a,b \in \mathring{\go{T}}$ we then connect $a$ and $b$ if $a \subset b$ 
maximally in ${\go{T}}$. In this way, the set of leaves is indeed given by $\CV \subset \go{T}$ since,
 for $r$ sufficiently large, $\mvert^{\mbn}_{r}$ consists purely of singletons.
The root is always given by $\rho_{\go{T}} = \CV$, by considering $r$ sufficiently small. 
It is easy to verify that the required properties hold for $\go{T}$ as a consequence of the fact that 
the children of any node, viewed as subsets of $\CV$, form a non-trivial partition of that node.
The labeling $\go{s}(\cdot)$ on internal nodes is defined as follows. For each $a \in \mathring{\go{T}}$, 
we set
\[
\go{s}(a) 
\eqdef
\max \{ r \in \N:\ a \in \mvert^{\mbn}_{r} \}\;.
\]
This is always finite since elements of $\mathring{\go{T}}$ are not singletons, while there always 
exists some $r$ such that $\mvert^{\mbn}_{r} = \{\{v\}\,:\, v \in \CV\}$.
This completes our construction of the labeled coalescence tree $(\go{T},\go{s})$ with the caveat that for purely aesthetic reasons we identify the ``singleton'' leaves of $\go{T}$ with their lone constituent element.
We will henceforth always treat the elements of $\mathring{\go{T}}$ as ``abstract'' nodes, once we are done constructing the tree we forget how they correspond to non-singleton subsets of $\CV$.  
 
The above procedure gives us a map $\hat{\mcT}: \N^{\go{G}} \rightarrow \widehat{\CU}_{\mvert} \ltimes \go{\mathrm{Lab}}_{\bullet}$ taking scale assignments to labeled coalescence trees, we write it $\mbn \mapsto (\mcT(\mbn),\go{s}(\mbn))$.

We also write $\CU_{\CV} \eqdef \mcT(\mcN_{\go{G}})$ for the set of coalescence trees arising from the set of scales appearing in our expansion. 

Then it is straightforward to see that by iteratively integrating variables one will have the bound, uniform in $\mbn \in \mcN_{\go{G}}$, 
\[
\mathrm{Vol}\Big( \Big\{x_{\CV_{0}} \in (\R \times \T^{2})^{\CV_{0}}: \mathrm{Int}^{\mbn}(x_{\CV}) \not = 0 \Big\} \Big)
\lesssim
\prod_{a \in \mathring{\go{T}}}
2^{-\go{s}(a)|\s|}
\]
The supremum factor in \eqref{eq: single slice bound} is estimated via by using our the earlier mentioned estimates on the $F_{i}$, by exploiting the triangle inequality one can, at the cost of another combinatorial constant, get a bound purely in terms of the distances $( 2^{-\go{s}(a)}: a \in \mathring{\go{T}})$.

When trying to estimate the sum of \eqref{eq: single slice bound} over all $\mbn \in \mcN_{\go{G}}$ we first sum over $\go{T} \in \CU_{\CV}$ (which is a finite sum) and then for each fixed $\go{T} \in \CU_{\CV}$ we sum over $\go{s} \in \Lab_{\go{T}}$ (this is an infinite sum) and then we sum over all $\mbn$ in a set called $\mcN_{\mathrm{tri}}(\go{T},\go{s})$ which we will define later. The key properties of $\mcN_{\mathrm{tri}}(\go{T},\go{s})$ are that (i) $|\mcN_{\mathrm{tri}}(\go{T},\go{s})|$ has a uniform finite bound as one varies  $(\go{T},\go{s}) \in \widehat{\CU}_{\CV} \ltimes \Lab_{\bullet}$ and (ii) $\mcN_{\mathrm{tri}}(\go{T},\go{s})$ contains all $\mbn \in \mcN_{\go{G}}$ which satisfy both $\widehat{\mcT}(\mbn) = (\go{T},\go{s})$ and $\supp(\mathrm{Int}^{\mbn}) \not = \emptyset$. 

Given $\go{T} \in \widehat{\CU}_{\CV}$ and $f \subset \mvert$ we write $f^{\uparrow}$ for the maximal internal node which is a \emph{proper} ancestor of all the elements of $f$. 
When $f = \{a\}$ we may write $a^{\uparrow}$ instead of $\{a\}^{\uparrow}$. We define $f^{\Uparrow}$ to be the maximal internal node which is a proper ancestor of $f^{\uparrow}$ if $f^{\uparrow} \not = \rho_{\go{T}}$, otherwise we set $f^{\Uparrow} = f^{\uparrow} = \rho_{\go{T}}$.
For $a \in \mathring{\go{T}}$ we write $\go{L}_{a}$ for the set of leaves of $\go{T}$ which are descendants of $a$.

We also define, for any $(\go{T},\go{s}) \in \widehat{\CU}_{\CV} \ltimes \go{\mathrm{Lab}}_{\bullet}$, $\mcN_{\mathrm{tri}}(\go{T},\go{s}) \subset \mcN_{\go{G}}$ to be the set of all those 
scale assignments $\mbn$ with $\hat{\mcT}(\mbn) = (\go{T},\go{s})$ and the property that for every $e \in \go{G}$
\begin{equ}\label{triangle condition for scales}
| n_{e} - \go{s}(e^{\uparrow}) | < 2 C \cdot|\mvert|,
\end{equ}
where $C > 0$ is chosen to be the same as \eqref{equ: domain for tree}. 
Clearly $|\mcN_{\mathrm{tri}}(\go{T},\go{s})|$ is finite and bounded uniform in $(\go{T},\go{s}) \in \widehat{\CU}_{\CV} \ltimes \go{\mathrm{Lab}}_{\bullet}$.

For each $(\go{T},\go{s}) \in \widehat{\CU}_{\CV} \ltimes \go{S}_{\bullet}$ we define
\begin{equ}\label{equ: domain for tree}
\DDom(\go{T},\go{s},x_{v_{0}}) \eqdef 
\{ x
\in 
(\R \times \T^{2})^{\mvert_{0}}:\ \forall e = \{v_i,v_j\} \in \mvert^{(2)},\link{x_{v_i}}{  x_{v_j}}{\go{s}(e^{\uparrow})}
\},
\end{equ} 
where we've used the notation \eqref{defining annular region}.
\begin{definition}\label{def:homogeneity}
For $\go{T} \in \widehat{\CU}_{\CV}$ we call a map $\varsigma_{\go{T}}: \mathring{\go{T}} \rightarrow \R$ a 
$\go{T}$-homogeneity. 
A collection of such maps $\varsigma = (\varsigma_{\go{T}})_{\go{T} \in \widehat{\CU}_{\CV}}$ is called 
a \textit{total homogeneity}. Addition of total homogeneities is defined pointwise. 
\end{definition}
We introduce two special families of total homogeneities which will play the role of ``Kronecker deltas'' out of which we will build other total homogenieties. 
Given any subset $\tilde{\CV} \subset \CV$, the total homogeneities $\delta^{\uparrow}[\tilde{\CV}]$ and $\delta^{\Uparrow}[\tilde{\CV}]$ are given by setting, for every $\go{T} \in \widehat{\CU}_{\CV}$ and $a \in \mathring{\go{T}}$, 
\begin{equ}\label{def: delta total homogeneities}
\delta_{\go{T}}^{\uparrow}[\tilde{\CV}](a)
\eqdef
\mathbbm{1}\left\{ a = \tilde{\CV}^{\uparrow,\go{T}} \right\}
\quad
\textnormal{and}
\quad
\delta_{\go{T}}^{\Uparrow}[\tilde{\CV}](a)
\eqdef
\mathbbm{1}\left\{ a = \tilde{\CV}^{\Uparrow,\go{T}} \right\},
\end{equ} 
where the superscript $\go{T}$ is used to remind readers that these operations are $\go{T}$-dependent. 
For $u \in \CV$ we will write $\delta^{\uparrow}[u]$ or $\delta^{\Uparrow}[u]$ instead of $\delta^{\uparrow}[\{u\}]$ or $\delta^{\Uparrow}[\{u\}]$.

\begin{definition}\label{def:boundedBy}
Given a set of scale assignments $\mcN_{\go{G}}$ and a total homogeneity $\varsigma$ we say that a family of continuous compactly supported functions $F = \left(F^{\mbn}\right)_{\mbn \in \mcN_{\go{G}}}$ on $(\R \times \T^{2})^{\CV_{0}}$ {\it is bounded by} $\varsigma$ if the following conditions hold.
\begin{enumerate}
\item There exists $x_{v_{0}} \in \R \times \T^{2}$ such that for each $(\go{T},\go{s}) \in \mtree_{\mvert} \ltimes \go{\mathrm{Lab}}_{\bullet}$, and $\mbn \in \mcN_{\mathrm{tri}}(\go{T},\go{s})$, one has
\begin{equation}\label{domain condition}
\supp
\left( 
F^{\mbn}(\cdot)
\right)
\subset
\DDom(\go{T},\go{s},x_{v_{0}})\;.
\end{equation} 
\item One has the bound
\begin{equation}\label{supremum bound} 
\| 
F
\|_{\varsigma, \mcN_{\go{G}}}
\eqdef
\sup_{
\substack{\go{T} \in \mtree_{\mvert}\\
\go{s} \in \go{\mathrm{Lab}}_{\go{T}}\\
\mbn \in \mcN_{\mathrm{tri}}(\go{T},\go{s})
}
}
\Big(
\prod_{a \in \mathring{\go{T}}}
2^{- \varsigma_{\go{T}}(a)\go{s}(a)}
\Big)
\sup_{x \in (\R \times \T^{2})^{\mvert_{0}}}
\big| F^{\mbn}(x) \big|
< \infty.
\end{equation}
(In the particular case $\mcN_{\go{G}} = \N^{\go{G}}$ we will also just write $\|F\|_{\varsigma}$.)
\end{enumerate}
\end{definition}

\begin{remark} Because of the domain constraint \eqref{domain condition}, it is clear that $F^{\mbn}$ must vanish unless $\mbn \in \mcN_{\mathrm{tri},\go{G}}$, where
\[
\mcN_{\mathrm{tri},\go{G}} \eqdef 
\bigsqcup_{(\go{T},\go{s}) \in \mtree_{\mvert} \ltimes \mathrm{Lab}_{\bullet}}
\mcN_{\mathrm{tri}}(\go{T},\go{s})\;.
\]
\end{remark}
\begin{remark} 
The notion of being ``bounded'' by a total homogeneity $\varsigma$ depends on a invisible ``combinatorial'' constant $C$ hidden in \eqref{equ: domain for tree} -- this affects both the domain constraint \eqref{domain condition} and the definition of $\| \cdot \|_{\varsigma, \mcN_{\go{G}}}$. 
In practice we want to be able to formulate that this constant $C$ can be chosen independently of certain parameters. 
Thus, if we have a collections of families of functions $F_{\theta} = (F^{\mbn}_{\theta})_{\mbn \in \mcN_{\go{G}}}$ where $\theta$ varies as a parameter in some set $\Theta$ we say that a the collection of families  $F_{\theta}$ are bounded uniform in $\theta \in \Theta$ by a total homogeneity $\varsigma$ if one can use the same constant $C$ in \eqref{equ: domain for tree} for all values of $\theta \in \Theta$.
\end{remark}
\begin{definition}
Given a set of scale assignments $\mcN_{\go{G}}$ and a total homogeneity $\varsigma$, we say that $\varsigma$ is subdivergence free in $\bar{\mvert} \subset \mvert$ for the set of scales $\mcN_{\go{G}}$ if for every $\go{T} \in \mcU_{\mvert}$ and every $a \in \mathring{\go{T}} \setminus \{\rho_{\go{T}}\}$ with $\go{L}_{a} \subset \bar{\mvert}$ one has
\begin{equ}\label{eq: divergence free condition}
\sum_{
b \in \mathring{\go{T}}_{\ge a}}
\varsigma_{\go{T}}(b)
< (|\go{L}_{a}| -1 ) |\s|.
\end{equ}
\end{definition}
\begin{definition}\label{order of total homogeneity}
We say a total homogeneity $\varsigma$ is of order $\alpha \in \R$ if for every $\go{T} \in \widehat{\mtree}_{\mvert}$ one has $\sum_{a \in \mathring{\go{T}}} \varsigma_{\go{T}}(a) - (|\mvert| - 1)|\s| = \alpha$.
\end{definition}
The following theorem will be useful in getting good bounds on various integrals.

\begin{theorem}\label{multicluster 1}
Suppose we are given a set of scales $\mcN_{\go{G}}$ and a family of smooth compactly supported functions $F = (F^{\mbn})_{\mbn \in \mcN_{\go{G}}}$ on $(\R \times \T^{2})^{\CV_{0}}$ and a total homogeneity $\varsigma$ on the trees of $\widehat{\CU}_{\CV}$ which is of order $\alpha$ and subdivergence free on $\mvert$ for the set of scales $\mcN_{\go{G}}$. 
Furthermore, suppose that $F$ is bounded by $\varsigma$ on the set of scales $\mcN_{\go{G}}$.

For $r \in \N$ we define
\begin{equ}[e:defN>]
\mcN_{\go{G}, > r} \eqdef \{ \mbn \in \mcN_{\go{G}}:\ \min_{e \in \go{G}} n_{e} > r \}
\enskip
\textnormal{and}
\enskip
\mcN_{\go{G}, \le r} \eqdef \{ \mbn \in \mcN_{\go{G}}:\ \min_{e \in \go{G}} n_{e} \le r \}\;.
\end{equ}
Then, for $\alpha > 0$, one has
\[
\sum_{\mbn \in \mcN_{\go{G}, \le r}}
\int_{\mvert_{0}} \back dy\ 
\left|F^{\mbn}(y)
\right|
\le 
\mathrm{const}(|\CV|)
2^{\alpha r}
\|F\|_{\varsigma,\mcN_{\go{G}, \le r}}
\]
while for $\alpha < 0$ one has
\[
\sum_{\mbn \in \mcN_{\go{G}, > r}}
\int_{\mvert_{0}} \back dy\ 
\left|F^{\mbn}(y)
\right|
\le 
\mathrm{const}(|\CV|)
2^{\alpha r}
\|F\|_{\varsigma,\mcN_{\go{G},> r}}\;.
\]
Here, $\mathrm{const}(|\CV|)$ is a combinatorial factor depending only 
on $|\CV|$ and not on $r$. 
\end{theorem}
\begin{proof}
This is essentially a special case of \cite[Lem.~A.10]{KPZJeremy} with $\nu_\star$
equal to the root of $\go{T}$. The only difference is that our ``subdivergence-free condition''
does not include the root itself. In the case $\alpha < 0$, Definition~\ref{order of total homogeneity}
implies that \eqref{eq: divergence free condition} also holds for the root and we
can apply \cite[Lem.~A.10]{KPZJeremy}. In the case $\alpha > 0$,  this is not the case, but 
the proof of \cite[Lem.~A.10]{KPZJeremy} still applies, the only difference being that the 
sum appearing in the base case $|\mathring{\go{T}}| = 1$ runs over large scales instead of
small scales.
\end{proof}
For the next theorem and what follows, for any $\go{T} \in \widehat{\CU}_{\CV}$ and $a \in \mathring{\go{T}}$ we define 
\[
\go{T}_{\not\ge a} \eqdef \{ b \in \mathring{\go{T}}:\ b \not \ge a\}\;.
\]
\begin{theorem}\label{thm: both bounds}
Let $\mcN_{\go{G}}$ be fixed and let $\go{G}_{\ast} \subset \go{G}$ be non-empty subset of edges which connects the collection of vertices of $\CV$ it is incident with, we denote this collection of vertices by $\CV_{\ast}$. 

Suppose that we are given a family of functions $F = (F^{\mbn})_{\mbn \in \mcN_{\go{G}}}$ and a total homogeneity $\varsigma$ which is sub-divergence free on $\CV$ for the scales $\mcN_{\go{G}}$, of order $\alpha < 0$, and satisfies the following large scale integrability condition: for every $\go{T} \in \mtree_{\mvert}$ and $u \in \mathring{\go{T}}$ with both $u \le (\mvert_{\ast})^{\uparrow}$ and $\go{T}_{\not\ge u} \not = \emptyset$ one has 
\begin{equ}\label{second cond of HQ}
\sum_{ 
w \in \go{T}_{\not\ge u}
}
 \varsigma_{\go{T}}(w) 
> 
|\s| \, |\CV \setminus \bo{L}_{u} |\;.
\end{equ}
Furthermore, suppose that $F$ is bounded by $\varsigma$ on the set of scales $\mcN_{\go{G}}$. 
Then if we set, for any $r \in \N$,  
\[
\mcN_{\go{G}, > r, \go{G}_{\ast}}
\eqdef
\Big
\{ \mbn \in \mcN_{\go{G}}:\ \min_{
e \in \go{G}_{\ast}
}
n_{e}
\ge
r
\Big\}
\]
we have the bound, uniform in $r$, 
\[
\sum_{\mbn \in \mcN_{\go{G}, > r, \go{G}_{\ast}}}
\int dy_{\mvert_{0}}\ 
\left|F^{\mbn}(y)
\right|
\lesssim
2^{\alpha r}
\inf_{\tilde{F} \in \mathrm{Mod}(F)}
\|\tilde{F}\|_{\varsigma, \mcN_{\go{G}}}\;.
\]
\end{theorem}
\begin{proof}
This is precisely the content of \cite[Lem.~A.10]{KPZJeremy}.
\end{proof}
\endappendix
\bibliographystyle{./Martin}
\bibliography{./refs}

\end{document}